\documentclass[11pt]{amsart}
\usepackage{amsmath, amsthm, amssymb, amsfonts, enumerate}
\usepackage[colorlinks=true,linkcolor=blue,urlcolor=blue]{hyperref}
\usepackage[utf8]{inputenc}
\usepackage{xcolor}
\usepackage{color}
\usepackage{geometry}
\usepackage{graphicx, comment}
\usepackage[numbers]{natbib}
\usepackage{caption}
\usepackage{subcaption}
\usepackage{xspace}

\allowdisplaybreaks[1]

\newtheorem{theorem}{Theorem}[section]
\newtheorem{remark}[theorem]{Remark}
\newtheorem{assumption}[theorem]{Assumption}
\newtheorem{definition}[theorem]{Definition}
\newtheorem{lemma}[theorem]{Lemma}
\newtheorem{proposition}[theorem]{Proposition}
\newtheorem{corollary}[theorem]{Corollary}

\theoremstyle{definition}
\newtheorem{example}{Example}
\newtheorem{excont}{Example}

\def \cB{{\mathcal B}}
\def \cC{{\mathcal C}}

\def \cF{{\mathcal F}}
\def \cG{{\mathcal G}}
\def \cI{{\mathcal I}}

\def \cL{{\mathcal L}}
\def \cO{{\mathcal O}}

\def \cS{\mathcal{S}}
\def \cT{\mathcal{T}}
\def \E{\mathsf{E}}
\def \P{\mathsf{P}}
\def \R{\mathbb{R}}

\def \ve{\varepsilon}

\def \tl{\tilde}
\DeclareMathOperator*{\esssup}{ess\,sup}
\DeclareMathOperator*{\argmax}{arg\,max}

\newcommand{\rom}[1]{\uppercase\expandafter{\romannumeral #1\relax}}
\newcommand\ind[1]{1_{#1}}
\newcommand\indd[1]{\ind{\{#1\}}}
\newcommand\lambdaSet{\mathcal{O}}
\newcommand\volatility{b}
\def \cadlag{c\`{a}dl\`{a}g\xspace}

\definecolor{brightmaroon}{rgb}{0.7, 0.23, 0.2}
\definecolor{magenta}{rgb}{0.4, 0.1, 0.5}

\newcommand\edt[1]{{#1}}

\title[Exit game]{Exit game with private information}
\author[H.D.~Kwon]{H. Dharma Kwon}
\author[J.~Palczewski]{Jan Palczewski}
\address{H.D.~Kwon: Gies College of Business, University of Illinois at Urbana-Champaign,
Champaign, Illinois 61820, USA}
\email{\href{mailto:dhkwon@illinois.edu}{dhkwon@illinois.edu}}
\address{J.~Palczewski: School of Mathematics, University of Leeds, Woodhouse Lane, LS2 9JT Leeds, UK.}
\email{\href{mailto:j.palczewski@leeds.ac.uk}{j.palczewski@leeds.ac.uk}}
\date{\today}

\numberwithin{equation}{section}

\makeatother

\begin{document}
\begin{abstract}
The timing of strategic exit is one of the most important but difficult
business decisions, especially under competition and uncertainty.
Motivated by this problem, we examine a stochastic game of exit in
which players are uncertain about their competitor's exit value. We
construct an equilibrium for a large class of payoff flows driven
by a general one-dimensional diffusion. In the equilibrium, the players
employ sophisticated exit strategies involving both the state variable
and the posterior belief process. These strategies are specified explicitly
in terms of the problem data and a solution to an auxiliary optimal
stopping problem. The equilibrium we obtain is further shown to be
unique within a wide subclass of symmetric Bayesian equilibria. 
\end{abstract}

\maketitle

\section{Introduction}

The timing of strategic exit is one of the most important but difficult
business decisions. According to anecdotes and empirical studies,
many firms in declining industries miss the optimal time to exit and
amass substantial financial loss (Horn et al. \cite{Horn2006}, Elfenbein
and Knott \cite{Elfenbein2015}). Exit decisions are even more complicated
when the firms are uncertain about the future profits such as in the
cases of 7 declining industries studied by Harrigan \cite{Harrigan2003}.
Furthermore, firms are generally uncertain about their rival firms'
exit value from the outside option. Even though these two types of
uncertainty pose practical and managerial difficulties, there has
been a paucity of attempts to investigate their combined impact on
the exit strategy. The goal of this paper is to study an exit game
under both types of uncertainty and obtain an equilibrium exit strategy.

In the model that we examine, we incorporate two salient features
of an exit game: a stochastic profit stream and private random exit
values, both of which are realistic features of an exit game between
competing businesses. Initially, firms operate in a duopoly and earn
identical profit streams dependent on a one-dimensional diffusion
modelling economic factors. Each firm is allowed to exit at any point
in time, but the remaining firm becomes a monopolist and enjoys a
monopoly profit flow. \edt{The exiting firm obtains an exit value which
is its private information unknown to the rival firm. The exit value
incorporates the outside option for the firm as well as the cost of
shutting down the enterprise. The exit values of both firms have the
same distribution and are independent.}

\edt{
We assume that each firm’s profit stream is publicly known, as it
depends on the revenue and the public demand. This is a common assumption
in many game-theoretic models of duopoly exit games \cite{Ghemawat1985,Fudenberg1986,Murto2004,Steg2015,Georgiadis2019}.
The underlying dynamics of economic factors (the state
process) is a general one-dimensional diffusion. A firm’s exit value is private information, hidden from the rival
firm. This reflects the fact that a firm’s exit value depends on many
internal factors that are not observable by outsiders, such as alternative
business opportunities, salvage values \cite{Harrigan2003}, or even
managerial behavioural biases \cite{Elfenbein2015}. The uncertainty
about the rival’s exit value is also a standard assumption in many
economic models of exit games \cite{Riley1980,Nalebuff1985,Milgrom1985,Fudenberg1986}.

The first main result of the paper is to obtain a novel equilibrium.
Specifically, we obtain a perfect Bayesian equilibrium that is succinctly
characterised by two variables: the underlying state $X_{t}$ and
the belief $Y_{t}$. According to the equilibrium strategy, a player
of type $\theta_{i}$ exits when $(Y_t)$ falls below $\theta_i$ and $(X_t, Y_t)$ is in an explicitly given action region, see \eqref{eqn:tau_PBE}.
The belief process $(Y_{t})$ has the meaning of the maximum type of
the opponent that has remained in the game, i.e., a player believes
that his opponent's type  at time $t$ is less than $Y_{t}$.

The novel characteristic of the equilibrium lies in the complexity
of its strategies not found in the deterministic counterpart. The
value of $(Y_{t})$ depends on the history of the sample path $(X_{t})$,
so it is not a simple function of the current value of $(X_{t})$; see
Eq. \eqref{eqn:dynamics_Y}. Most of the extant models of exit decisions
prescribe either a deterministic timeline to exit in the deterministic
exit games \cite{Nalebuff1985,Ghemawat1985,Fudenberg1986} or a profit-threshold
policy in the case of a single-player model. In contrast, our results
suggest that a much more elaborate strategy is called for: the players
should continuously monitor the evolution of $(X_{t})$ and update their
beliefs $(Y_{t})$ regarding their opponent's type and exit as soon
as $(Y_{t})$ falls below their own type; see an alternative expression
for the strategy in (\ref{eqn:tau_form}).

To our knowledge, this is the first equilibrium solution obtained
for a stochastic exit game with a diffusive state variable and a continuously
distributed private type. In the deterministic model such as in Fudenberg
and Tirole \cite{Fudenberg1986}, the time variable is a sufficient
state variable, so the equilibrium generating process depends on time
alone. In contrast, in the stochastic game, the state of the market
evolves separately from the time variable, and hence, the dimensionality
of the problem increases. 

It is worthy of note that our
equilibrium is a natural extension of the known results in the extant
literature. Recall that our exit game model incorporates both a diffusive
state variable and asymmetric information. Previous studies have obtained
an equilibrium of exit games with one of the two features: either
a diffusive state variable \cite{Steg2015,Georgiadis2019} or private
types \cite{Milgrom1985,Fudenberg1986}. In this paper, we bridge
these two strands of literature by showing that our result coincides
with the extant results when one of the two features is absent; see
Section \ref{sec:Discussions}.

The striking feature of our equilibrium solution is that it is given
explicitly for a large class of underlying diffusion processes and
payoff functions, so it cannot be obtained by a guess-and-verify approach.
Instead, we only require that an auxiliary optimal stopping problem
of exit from a duopoly of one player has a solution of a threshold
type and the threshold depends continuously on the exit value; see
Section \ref{sec:single_player}. 

Our second main result concerns the possibility of other symmetric
equilibria. Non-zero sum games typically have multiple equilibria,
and identifying them is a formidable task (see Feinstein et al. \cite{Feinstein2022}
for recent results in discrete time games). Although uniqueness is
rarely studied in the continuous-time literature, we are able to demonstrate
that our exit game has a unique symmetric equilibrium in a large class
of symmetric equilibria in which the belief process
$(Y_{t})$ (or, more precisely, the generating process $(A_{t})$ in a one-to-one map correspondence with $(Y_{t})$) has a generalised
derivative that satisfies certain semi-continuity criteria (Thm.~\ref{thm:unique}). 

The proof of uniqueness is a significant mathematical result. There
are mathematical difficulties stemming from the continuum of player
types and the diffusive dynamics of the underlying state process.
The proof requires a combination of probabilistic and analytical methods,
and it demonstrates technical complexity involved in establishing
such results in a diffusive setting with asymmetric information. 
Further mathematical details about our approach are summarised in the beginning of Section \ref{sec:uniqeness}.

In the context of exit games, the uniqueness question is also of game-theoretic
interest. A war of attrition under incomplete information is known
to have, in general, a continuum of equilibria \cite{Riley1980,Nalebuff1985},
but there are variants of exit game models that have unique equilibria
due to special conditions \cite{Fudenberg1986,Ponsati1995}. This
paper adds to this strand of the literature by establishing uniqueness
results in an exit game with state diffusive dynamics.

In addition to the two main results, our paper also provides a new
framework for stopping games with continuously distributed private
types. In our equilibrium solution, each player employs a pure strategy
stopping time that depends on the private type. However, from a player's
perspective, the opponent's exit times resemble a mixed strategy \cite{Steg2015}, albeit with the mixing variable which is not uniform on $(0,1)$ but distributed as a player's type. Mathematically, a mixed stopping strategy can be represented as a randomised stopping time characterised by an increasing process adapted to a player's filtration (a generating
process) and a randomisation device which is independent from the
underlying randomness in the game (see, e.g., De Angelis et al. \cite[Def.~2.2]{DAMP2021}); the stopping time is defined as the first time that the generating process exceeds the value of the randomisation device. Despite the similarity, there is one fundamental
difference: our equilibrium does not introduce a private randomisation
device because it is not a mixed strategy equilibrium. The apparent
randomness comes from not knowing the opponent's type, i.e., the asymmetric
information. Nevertheless, because a player's actions resemble randomised
stopping when perceived from the rival's perspective, we can exploit
a similar mathematical framework. This observation is key to the reformulation
of the problem in terms of best response optimal stopping problems
in Lemma \ref{lem:J_int}, where the exit time of the opponent is
replaced with a functional of the belief process. 

The best response formulation recasts an equilibrium as a solution to a fixed-point problem whereby the strategy of a player's opponent (driven by the belief process) is also the best response for any value of the player's type. We emphasise that the symmetry of the game in which players are identical is crucial for this approach and cannot be naturally relaxed.

Lastly, one of the mathematical challenges of our model is the construction
of $(Y_{t})$. Just as in our paper, the introduction of an appropriate
belief process is often encountered in papers studying games with
asymmetric information (Grün \cite{Grun2013}, Gensbittel and Grün
\cite{Gensbittel2019}, Ekström et al. \cite{Ekstrom2022Salami},
De Angelis and Ekström \cite{DeAngelis2020ghosts}) and is akin to
the filter process in problems with partial information. Our model
is unique because the belief process $(Y_{t})$ is defined as a solution to an ordinary differential equation (ODE) \eqref{eqn:dynamics_Y} with a
discontinuous right-hand side. Solutions to ODEs with discontinuous
right-hand side are usually non-unique (Filippov \cite[Ch.~2]{Filippov})
and related to differential inclusions. There are several definitions
of solutions to such ODEs. In this paper, we adopt Caratheodory's
approach in which the solution is a continuous function which satisfies
the integral version of the ODE with probability one. Classical results
from the ODE theory require that the discontinuities are located on
smooth surfaces, the condition that is not satisfied for our ODE in
which the discontinuity points are determined by a path of a diffusion
process. Instead, we obtain a maximal Caratheodory solution as a monotone
limit of upper approximations of the right-hand side. 
}

\subsection{Literature review}

Our paper extends the literature on stochastic stopping games under
asymmetric information. The asymmetry of information poses mathematical
challenges, but there has been a flurry of recent contributions,
particularly, in zero-sum games. Games examined in the literature
possess various information structures and sources of uncertainty.
One-sided asymmetry, where one player has a strictly larger information
flow, were studied by Grün \cite{Grun2013}, Lempa and Matom\"aki
\cite{Lempa2013}, and De Angelis et al. \cite{DeAngelis2020}. Gensbittel
and Grün \cite{Gensbittel2019} examine a zero-sum game in which
each player can only observe a private continuous time finite-state
Markov chain while the payoff is a function of both players' processes.
A recent paper De Angelis et al. \cite{DAMP2021} shows the existence
of a Nash equilibrium (saddle point) for general payoff processes
in the framework with asymmetric and partial information.

In non-zero sum games, players' payoffs do not have to sum up to zero,
which results in a much richer set of equilibria even in the case
of full information. PDE results are often in the form of a verification
theorem for a solution to a system of quasi-variational inequalities
(Bensoussan and Friedman \cite{Bensoussan1977}). The existence of
an appropriate solution of this system is studied in Nagai \cite{Nagai1987}
for symmetric Markov processes and continued in Cattaiaux and Lepeltier
\cite{Cattaiaux1990} for Ray-Markov processes. Superharmonic characterisation
of players' value functions for strong Markov processes is provided
in Attard \cite{Attard2018}. Hamadene and Zhang \cite{Hamadene2010}
and Hamadene and Hassani \cite{Hamadene2014} use iterative methods
to construct equilibria in games with two or more players which, even
in a Markovian setting, are not in the form of hitting times. Sub-game
perfect equilibria are examined in Steg \cite{Steg2015} and Riedel
and Steg \cite{Riedel2017}.

In economics literature, the framework of non-zero sum games has been
applied to a game of exit from a declining industry. Murto \cite{Murto2004}
investigates an exit game with a geometric Brownian motion as the
profit flow and characterises Markov perfect equilibria. Steg \cite{Steg2015}
studies the subgame perfection concept in a class of stochastic exit
games and finds a mixed strategy equilibrium analogous to the one
in the deterministic war of attrition. Georgiadis et al. \cite{Georgiadis2019}
investigate an exit game under complete information with a stochastic
profit flow and find that the stochasticity combined with asymmetry
between the players destabilise the mixed strategy equilibrium.

Closest to this paper are studies of non-zero sum games with asymmetric
information. In particular, Fudenberg and Tirole \cite{Fudenberg1986}
examines a duopoly game of exit with a continuous distribution of
private types as in our paper, but it studies a deterministic game
unlike our model. Décamps and Mariotti \cite{Decamps2004}
examines a duopoly game of investment in a common project with Poisson
signals about its quality. Players have incomplete information about their opponent's investment
costs, so the problem is cast as a stopping game under asymmetric information.
In another strand of research, ghost games in which a player does not know if his
opponent exists (De Angelis and Ekström \cite{DeAngelis2020ghosts},
Ekström et al. \cite{Ekstrom2022Salami}) are solved using a verification
approach and result in an equilibrium in randomised stopping times.
Pérez et al. \cite{Perez2021} consider a game where one player can
only stop at random times indicated by a Poisson process. Using a
fixed point theorem, the authors show that the game has a Nash equilibrium
in threshold strategies, i.e., in pure stopping times. The optimality
of this equilibrium is then extended to the class of all stopping times using optimal
stopping theory arguments applied to the best response problems. Conceptually,
our paper has similarities to both lines of research. The equilibrium
strategies we find are akin to randomised stopping times. However,
instead of postulating a PDE for value functions, we use probabilistic
optimal stopping methods for best response problems to prove that
a postulated pair of stopping strategies is a perfect Bayesian equilibrium.

\subsection{Summary of results}\label{subsec:Summary-of-results}

\edt{
Our model considers an economy in which the underlying economic factors
are described by a one-dimensional diffusion 
\begin{equation*}
dX_{t}=\mu(X_{t})dt+\volatility(X_{t})dW_{t},\quad t\ge0,
\end{equation*}
where $(W_{t})_{t\ge0}$ is a Brownian motion. There are two players,
each of whom has a private type $\theta_{i}$, $i=1,2$, that describes
their exit value. The distribution of both private types, denoted
by $F$, is identical and known to both players, but the player's
own exits value is private and unknown to his opponent. Players decide
when to exit the market by choosing the stopping times $\tau_{1}(\theta_{1})$
and $\tau_{2}(\theta_{2})$ that depend on their types. The player
who remains in the market becomes a monopolist and never exits. The
expected payoff to Player $i$ is then given by 
\[
\int\E_{x}\bigg[\int_{0}^{\tau_{i}\wedge\tau_{j}(\theta_{j})}e^{-rt}D(X_{t})dt+1_{\tau_{i}\le\tau_{j}(\theta_{j})}e^{-r\tau_{i}}\theta_{i}+1_{\tau_{i}>\tau_{j}(\theta_{j})}\int_{\tau_{j}(\theta_{j})}^{\infty}e^{-rt}M(X_{t})dt\bigg]dF(\theta_{j}),
\]
where $x$ is the initial value of the process $(X_{t})$, $D$ is
the duopoly profit flow and $M$ is the monopoly profit flow.

The model is completely symmetric, and this symmetry will be exploited
in construction of a symmetric equilibrium. We first note that in
equilibrium a player of type $\theta'$ would want to exit earlier
than a player of type $\theta''$ if $\theta'>\theta''$ because a
player would exit earlier if his outside option is more attractive.
Therefore, it is natural that the player's stopping time $\tau_{i}(\theta)$
be monotone in the type $\theta$. Based on this monotonicity, we
can deduce the existence of a stochastic process $(Y_{t})$ that has
the interpretation of the highest value of the remaining type for
both players. In turn, we hypothesise 
\begin{equation}
\tau_{i}(\theta_{i})=\inf\{t\ge0:\theta_{i}>Y_{t}\}.\label{eqn:SUMM_tau}
\end{equation}
We show that a symmetric equilibrium ensues when the process $(Y_{t})$
solves the differential equation 
\[
dY_{t}=-\frac{F(Y_{t})}{F'(Y_{t})}\ \frac{rY_{t}-D(X_{t})}{m(X_{t})-Y_{t}}\ind{X_{t}\le\alpha(Y_{t})},
\]
where $\alpha(\theta)$ is the optimal exit threshold for a player
with exit value $\theta$ whose opponent is committed to never exit
the market while $m$ is the expected total discounted monopoly profit:
$m(x)=\E_{x}\int_{0}^{\infty}e^{-rt}M(X_{t})dt$. The results are
stated as Theorems \ref{thm:Nash} and \ref{thm:PBE}. It turns out
that this equilibrium is unique (Theorem \ref{thm:unique}) in the
class of symmetric equilibria of the form \eqref{eqn:SUMM_tau} for
a large class of processes $(Y_{t})$.
}

\subsection{Outline of the paper and terminology}

The paper is organised as follows. In Section \ref{sec:model} we
introduce the framework for the exit game and a sufficient condition
for a Nash equilibrium in terms of best response optimal stopping
problems. An optimal stopping problem of exit from a duopoly is briefly
discussed in Section \ref{sec:single_player}. Its solution plays
a pivotal role in the construction of an equilibrium of the exit game
in Section \ref{sec:equil}. \edt{Section \ref{subsec:heuristic} provides a heuristic
derivation of the equilibrium in \eqref{eqn:lambda}. The uniqueness of the equilibrium is demonstrated in Section \ref{sec:uniqeness}. Extreme cases when the underlying dynamics are deterministic
or the distribution of exit value collapses to a point are covered
in Section \ref{sec:Discussions}. Appendix develops asymptotic
bounds for exit times of a diffusion and contains detailed calculations
for an example discussed in the text. To get
a basic understanding of the motivation, definition, and properties
of the equilibrium without having to read technical details of the
mathematical framework, we recommend reading Sections \ref{subsec:Summary-of-results},
\ref{subsec:The-strategy-profile}, \ref{subsec:heuristic}, \ref{subsec:Nash},
\ref{subsec:PBE}, and \ref{sec:Discussions}. }

As a matter of convention throughout the paper,
we write increasing/decreasing for non-strict monotonicity, and strictly
increasing/decreasing for strict monotonicity.

\section{Model}

\label{sec:model}

Consider a complete probability space $(\Omega,\cF,\P)$ with filtration
$(\cF_{t})_{t\ge0}$ satisfying the usual conditions. The underlying
state of the system is described by a one-dimensional diffusion 
\begin{equation}
dX_{t}=\mu(X_{t})dt+\volatility(X_{t})dW_{t},\quad t\ge0,\label{eqn:X}
\end{equation}
where $(W_{t})_{t\ge0}$ is an $(\cF_{t})$-Brownian motion. We assume
that $\mu(\cdot)$ and $\volatility(\cdot)$ are Lipschitz continuous
so that $(X_{t})_{t\ge0}$ is a unique strong solution which is a
strong Markov process. We further assume that $\volatility(\cdot)>0$.
We denote by $\cI=(x_{L},x_{U})$ the (potentially infinite) interval
in which $X_{t}$ takes values and assume that its boundaries are
not attainable.

The model includes two agents (players), each having a private random
variable $\theta_{i}$, $i=1,2$, describing their exit value (type).
At the outset, the players do not know the exit value of either player
although the probability distribution of the types is public knowledge.
At time $0$, each player learns his own private type which remains
unknown to his opponent throughout the game. At any time $t>0$, a
player is aware of his own type, but he only holds a belief (a probability
distribution) about his opponent's type. This flow of information
simulates a game of exit among firms that learn of their own exit
value upon entering an industry or a new market while remaining uncertain
about their opponent's exit value until the end of the game. The exit
value represents the reward from exit that may include the salvage
value and the value of an alternative business venture. Once a player
exits, the opponent player is assumed to hold perpetual monopoly of
the industry/market.

Supporting those random variables are two complete probability spaces
$(\Omega^{\theta_{i}},\cF^{\theta_{i}},\P^{\theta_{i}})_{i=1,2}$.
Put 
\begin{equation}
(\tl\Omega,\tl\cF,\tl\P)=(\Omega,\cF,\P)\otimes(\Omega^{\theta_{1}},\cF^{\theta_{1}},\P^{\theta_{1}})\otimes(\Omega^{\theta_{2}},\cF^{\theta_{2}},\P^{\theta_{2}}),\label{eqn:tl_omega}
\end{equation}
and denote by $(\tl\cF_{t})_{t\ge0}$ the embedding of the filtration
$(\cF_{t})_{t\ge0}$ onto $(\tl\Omega,\tl\cF)$, i.e., $\tl\cF_{t}=\sigma(\{A\times\Omega^{\theta_{1}}\times\Omega^{\theta_{2}}:\ A\in\cF_{t}\})$.
With an abuse of notation, we will write $(X_{t})_{t\ge0}$, $\theta_{1}$
and $\theta_{2}$ for an embedding of the process $(X_{t})_{t\ge0}$
and the random variables $\theta_{1},\theta_{2}$ into $(\tl\Omega,\tl\cF,\tl\P)$.
We will denote by $\E$ the expectation with respect to $\P$ and
by $\tl\E$ the expectation with respect to $\tl \P$. With the lower
index by $\P$,$\E$ and $\tl \P$, $\tl\E$ we indicate the initial
value $X_{0}$.

The information flow of player $i$ is modelled by filtration $(\cF_{t}^{i})_{t\ge0}:=\tl\cF_{t}\vee\sigma(\theta_{i})$,
i.e., $\cF_{t}^{i}$ is the smallest $\sigma$-algebra containing
$\tl\cF_{t}$ and with respect to which $\theta_{i}$ is measurable.
His action is given by a $(\cF_{t}^{i})$-stopping time $\tau_{i}$
which we denote by $\tau_{i}\in\cT(\cF_{t}^{i})$. The payoff of Player
$1$ is 
\[
J_{1}(x,\tau_{1},\tau_{2})=\tl\E_{x}\bigg[\int_{0}^{\tau_{1}\wedge\tau_{2}}e^{-rt}D(X_{t})dt+1_{\tau_{1}\le\tau_{2}}e^{-r\tau_{1}}\theta_{1}+1_{\tau_{1}>\tau_{2}}e^{-r\tau_{2}}m(X_{\tau_{2}})\bigg],
\]
and, analogously, the payoff of Player 2 is 
\[
J_{2}(x,\tau_{1},\tau_{2})=\tl\E_{x}\bigg[\int_{0}^{\tau_{1}\wedge\tau_{2}}e^{-rt}D(X_{t})dt+1_{\tau_{2}\le\tau_{1}}e^{-r\tau_{2}}\theta_{2}+1_{\tau_{2}>\tau_{1}}e^{-r\tau_{1}}m(X_{\tau_{1}})\bigg],
\]
where 
\[
m(x)=\E_{x}\bigg[\int_{0}^{\infty}e^{-rs}M(X_{s})ds\bigg],\qquad x\in\cI.
\]
Function $D(x)$ represents the profit flow to a player in a duopoly
while $M(x)$ is the profit flow to the remaining player in the monopoly,
hence $m(x)$ is the cumulative discounted profit earned by a monopolist
given $X_{0}=x$.

\begin{remark} When $\tau_{1}=\tau_{2}$ both players exit the market
and earn their exit value. This has a clear explanation from a managerial
perspective as players make the exit decision independently. From
a mathematical perspective, it will never happen as in an equilibrium
that we study in this paper the probability of a double exit is zero.
\end{remark}

We make the following assumptions. \begin{assumption}\label{ass:theta_supp}
Random variables $\theta_{i}$, $i=1,2$, have the support $[\theta_{L},\theta_{U}]$.
\end{assumption}

\begin{assumption}\label{ass:functional} Functions $D,M:\cI\to[0,\infty)$
are continuous, increasing and bounded, and $M>D$. Furthermore, $\inf_{x\in\cI}m(x)>\theta_{U}$,
and the interest rate $r>0$. \end{assumption}

\begin{assumption}\label{ass:coeff} Coefficients $\mu$ and $b$
are Lipschitz continuous and $b>0$. \end{assumption}

Assumption \ref{ass:coeff} means that $b$ is uniformly non-degenerate
on any compact subset of $\cI$ and $(X_{t})_{t\ge0}$ is a weak Feller
process, i.e., its semigroup maps continuous bounded functions into
continuous functions. Following from this observation, thanks to Assumption
\ref{ass:functional}, function $m$ defined above and function $d$
given as 
\[
d(x)=\E_{x}\bigg[\int_{0}^{\infty}e^{-rs}D(X_{s})ds\bigg],\qquad x\in\cI,
\]
are continuous and bounded.

\begin{remark} In our model, $m(x)$ is the cumulative future profit
flow for a player who becomes a monopolist when the underlying process
is in state $x$. One might argue that the monopolist should be allowed
to exit the market, i.e., in the firm's payoffs, $m$ should be replaced
by 
\[
\hat{m}(x;\theta)=\sup_{\tau}\E_{x}\bigg[\int_{0}^{\tau}e^{-rs}M(X_{s})ds+e^{-r\tau}\theta\bigg],
\]
where $\theta\in[\theta_{L},\theta_{U}]$ is the exit value of the
monopolist. Clearly, $\hat{m}(x;\theta)\ge m(x)$. By Assumption \ref{ass:functional},
we have $m(x)>\theta_{U}$, so $\hat{m}(x;\theta)>\theta$ for every
possible exit value $\theta$. This means that the optimal stopping
time in $\hat{m}(x;\theta)$ is $\tau=\infty$, so $\hat{m}(x;\theta)=m(x)$.
%This shows that our assumption that the monopolist never exits the market is at no detriment to the generality. 
\edt{This simplification has been accounted for in the definition of player payoffs $J_1$ and $J_2$. The case when $\inf_{x \in \cI} m(x) < \theta_U$ is significantly more difficult and beyond the scope of this paper. It will certainly lead to a different behaviour of players as it is possible that both players exit the market at a finite time when it is suboptimal to continue as a monopolist even with the lowest exit value $\theta_L$.}
\end{remark}

We now introduce the notion of a Nash equilibrium in the context of
our game.

\begin{definition}\label{def:Nash} A strategy profile (a pair of
strategies) $(\tau_{1}^{*},\tau_{2}^{*})\in\cT(\cF_{t}^{1})\times\cT(\cF_{t}^{2})$
is called a \emph{Nash equilibrium} for $x$ if for any other pair
of strategies $(\tau_{1},\tau_{2})\in\cT(\cF_{t}^{1})\times\cT(\cF_{t}^{2})$
we have 
\begin{align*}
 & J_{2}(x,\tau_{1}^{*},\tau_{2}^{*})\ge J_{2}(x,\tau_{1}^{*},\tau_{2}),\quad\tl\P_{x}-a.s.,\\
 & J_{1}(x,\tau_{1}^{*},\tau_{2}^{*})\ge J_{1}(x,\tau_{1},\tau_{2}^{*}),\quad\tl\P_{x}-a.s.
\end{align*}
\end{definition}

To construct a Nash equilibrium, we need to understand the structure
of $(\cF_{t}^{i})$-stopping times. The reader is referred to \cite[Proposition 3.3]{esmaeeli2018}
for related results in a more general setting and with a different
method of proof.

\begin{proposition}\label{prop:structure} Let Assumption \ref{ass:theta_supp}
hold. If $\hat{\tau}:(\Omega,\cF)\otimes([\theta_{L},\theta_{U}],\cB([\theta_{L},\theta_{U}]))\to([0,\infty),\cB([0,\infty)))$
is measurable and the mapping $\hat{\tau}(\cdot,\theta):\Omega\to[0,\infty)$ is an $(\cF_{t})$-stopping
time for each $\theta\in[\theta_{L},\theta_{U}]$, then $\tau=\hat{\tau}(\cdot,\theta_{i})$ is an $(\cF_{t}^{i})$-stopping
time. Conversely, for every $(\cF_{t}^{i})$-stopping time $\tau$
there is a mapping $\hat{\tau}$ satisfying the above conditions such
that $\tau=\hat{\tau}(\cdot,\theta_{i})$, $\P$-a.s.
\end{proposition} 
\begin{proof}
The fact that $\tau=\hat{\tau}(\cdot,\theta_{i})$ is an $(\cF_{t}^{i})$-stopping
time is immediate.

The proof of the converse is more involved. Fix $i\in\{1,2\}$ and
let $\cG=\hat{\cF}\vee\sigma(\theta_{i})$, where $\hat{\cF}=\sigma\{A\times\Omega^{\theta_{1}}\times\Omega^{\theta_{2}}:\ A\in\cF\}$.
Recalling the definition of $\tl\Omega$ in \eqref{eqn:tl_omega},
we write its elements $\omega$ as $(\omega_{0},\omega_{1},\omega_{2})\in\Omega\times\Omega^{\theta_{1}}\times\Omega^{\theta_{2}}$.
Consider first $\tau(\omega_{0},\omega_{1},\omega_{2})=\tau'(\omega_{0})\ind{A}(\omega_{i})$
for $i \in \{1,2\}$, $A\in\sigma(\theta_{i})$ and $\tau'$ an $(\cF_{t})$-stopping
time. By \cite[p. 76]{Halmos} there is $B\in\cB([\theta_{L},\theta_{U}])$
such that $A=\theta_{i}^{-1}(B)$. Hence $\tau(\omega_{0},\omega_{1},\omega_{2})=\hat{\tau}(\omega_{0},\theta_{i}(\omega_{i}))$
for $\hat{\tau}(\omega_{0},z)=\tau'(\omega_{0})\ind{B}(z)$. This
representation extends to any $\cG$-measurable non-negative function
using Monotone Class Theorem, i.e., any such function has a representation
as $\hat{\tau}(\omega_{0},\theta_{i}(\omega_{i}))$ for a measurable
$\hat{\tau}$ as in the statement of the theorem.

It remains to show that if $\tau$ is $(\cF_{t}^{i})$-stopping time,
then $\hat{\tau}(\cdot,z)$ is an $(\cF_{t})$-stopping time for any
$z\in[\theta_{L},\theta_{U}]$. Fix $t\ge0$ and let $A=\{\tau\le t\}\in\cF_{t}^{i}$.
By analogous arguments as above applied to $\cG=\tilde{\cF}_{t}\vee\sigma(\theta_{i})$,
there is a $\cF_{t}\otimes\cB([\theta_{L},\theta_{U}])$-measurable
function $\hat{f}$ such that $\ind{A}(\omega_{0},\omega_{1},\omega_{2})=\hat{f}(\omega_{0},\theta_{i}(\omega_{i}))$.
By \cite[Prop.~3.3.2]{Bogachev}, the set $A_{z}:=\{\omega_{0}:\ \hat{f}(\omega_{0},z)=1\}$
is $\cF_{t}$-measurable for $z\in[\theta_{L},\theta_{U}]$. For any
$\omega_{i}\in\Omega^{\theta_{i}}$, we have $\{\omega_{0}:\ \hat{\tau}(\omega_{0},\theta_{i}(\omega_{i}))\le t\}=A_{\theta_{i}(\omega_{i})}\in\cF_{t}$,
so $\{\hat{\tau}(\cdot,z)\le t\}\in\cF_{t}$ for any $z$ belonging
to the support $[\theta_{L},\theta_{U}]$ of $\theta_{i}$ (c.f. Assumption
\ref{ass:theta_supp}). As $t\ge0$ is arbitrary, the above arguments
show that $\hat{\tau}(\cdot,z)$ is an $(\cF_{t})$-stopping time
for any $z\in[\theta_{L},\theta_{U}]$. 
\end{proof}
It will be convenient to define a payoff functional for a deterministic
exit value: for $\sigma\in\cT(\cF_{t})$, $\gamma\in[\theta_{L},\theta_{U}]$
and a random time $\tau$ on $(\tl\Omega,\tl\cF,\tl\P)$, 
\begin{equation}
J(x,\sigma,\tau;\gamma)=\tl\E_{x}\bigg[\int_{0}^{\sigma\wedge\tau}e^{-rt}D(X_{t})dt+\ind{\sigma\le\tau}e^{-r\sigma}\gamma+\ind{\sigma>\tau}e^{-r\tau}m(X_{\tau})\bigg].\label{eqn:J}
\end{equation}
\edt{The function $J(x,\sigma,\tau;\gamma)$ has the
meaning of the expected payoff to player $i$ whose type is $\gamma$ when
player $i$'s strategy is to stop at $\sigma$ and player $j$'s strategy
is to stop at $\tau$.}

Fundamental to our construction of the Nash equilibrium is the following
sufficient condition enabled by the structure of $(\cF_{t}^{i})$-stopping
times established in Proposition \ref{prop:structure}.

\begin{corollary}\label{corr:equil} Let Assumption \ref{ass:theta_supp}
hold and assume that $\hat{\tau}_{1},\hat{\tau}_{2}:\Omega\times[\theta_{L},\theta_{U}]\to[0,\infty)$
are as in Proposition \ref{prop:structure}. Define $\tau_{1}=\hat{\tau}_{1}(\cdot,\theta_{1})$
and $\tau_{2}=\hat{\tau}_{2}(\cdot,\theta_{2})$. If for each $\theta\in[\theta_{L},\theta_{U}]$
we have 
\[
\hat{\tau}_{1}(\cdot,\theta)\in\argmax_{\sigma\in\cT(\cF_{t})}J(x,\sigma,\tau_{2};\theta)\qquad\text{and}\qquad\hat{\tau}_{2}(\cdot,\theta)\in\argmax_{\sigma\in\cT(\cF_{t})}J(x,\sigma,\tau_{1};\theta)
\]
then $(\tau_{1},\tau_{2})$ is a Nash equilibrium. \end{corollary} 
\begin{proof}
Take any $(\cF_{t}^{1})$-stopping time $\tau'_{1}$. By Proposition
\ref{prop:structure}, it can be written as $\hat{\tau}'_{1}(\cdot,\theta_{1})$
for a $\cF\otimes\cB([\theta_{L},\theta_{U}])$-measurable function
$\hat{\tau}'_{1}$ such that $\hat{\tau}'_{1}(\cdot,z)$ is an $(\cF_{t})$-stopping
time for each $z\in[\theta_{L},\theta_{U}]$. By the tower property
of conditional expectation and the independence of $\tau_{2}=\hat{\tau}_{2}(\cdot,\theta_{2})$
from $\theta_{1}$ we have 
\begin{align*}
J_{1}(x,\tau'_{1},\tau_{2}) & =\tl\E_{x}\bigg[\tl\E_{x}\bigg[\int_{0}^{\tau'_{1}\wedge\tau_{2}}e^{-rt}D(X_{t})dt+\ind{\tau'_{1}\le\tau_{2}}e^{-r\tau'_{1}}\theta_{1}+\ind{\tau'_{1}>\tau_{2}}e^{-r\tau_{2}}m(X_{\tau_{2}})\bigg|\sigma(\theta_{1})\bigg]\bigg]\\
 & =\int_{\gamma\in[\theta_{L},\theta_{U}]}J(x,\hat{\tau}'_{1}(\cdot,\gamma),\tau_{2};\gamma)dF_{\theta_{1}}(\gamma)\le\int_{\gamma\in[\theta_{L},\theta_{U}]}J(x,\hat{\tau}_{1}(\cdot,\gamma),\tau_{2};\gamma)dF_{\theta_{1}}(\gamma)\\
 & =\tl\E_{x}\bigg[\tl\E_{x}\bigg[\int_{0}^{\tau_{1}\wedge\tau_{2}}e^{-rt}D(X_{t})dt+\ind{\tau_{1}\le\tau_{2}}e^{-r\tau_{1}}\theta_{1}+\ind{\tau_{1}>\tau_{2}}e^{-r\tau_{2}}m(X_{\tau_{2}})\bigg|\sigma(\theta_{1})\bigg]\bigg]\\
 & =J_{1}(x,\tau_{1},\tau_{2}),
\end{align*}
where $F_{\theta_{1}}$ is the cumulative distribution function of
$\theta_{1}$ and for the inequality we used that $\hat{\tau}_{1}(\cdot,\gamma)$
maximises $J(x,\cdot,\tau_{2};\gamma)$. We repeat the same arguments
for $J_{2}$. 
\end{proof}

%%%%%%%%%%%%%%%%%%%%%%%%%%%%%%%%%%%%%%%%%%%%%%%%%%%%%%%%%%%

\section{Single player problem}

\label{sec:single_player}

In this section, we assume that Player 2 never exits. Player 1's decision
problem reduces to an optimal stopping problem parametrised by $\theta$
with the payoff functional 
\begin{equation}
J(x,\tau;\theta)=\E_{x}\Big[\int_{0}^{\tau}e^{-rs}D(X_{s})ds+e^{-r\tau}\theta\Big],\label{eq:U-theta}
\end{equation}
where $\tau$ is an $(\cF_{t})$-stopping time. The solution of this
problem will be used in the construction of the equilibrium for the
exit game.

Denote the value function corresponding to \eqref{eq:U-theta} by
\begin{equation}
u(x;\theta)=\sup_{\tau\in\cT(\cF_{t})}J(x,\tau;\theta).\label{eqn:value_u}
\end{equation}
\edt{Given that the profit flow $D$ is increasing, the payoff $J(x,\tau;\theta)$
is increasing in $x$. Therefore, we conclude that the value function
$u(x;\theta)$ is increasing in $x$, so if $x$ is in the stopping
set (i.e., it is optimal to stop when $X_{t}=x$ ), then $(x_{L},x]$
is in the stopping set. This implies that the optimal strategy should
be given by the first entry time $\tau=\inf\{t\ge0:X_{t}\le\alpha\}$
for a threshold $\alpha$ depending on $\theta$. We will impose assumptions
sufficient to deduce this result and a characterisation of the threshold
$\alpha$ from \cite[Thm.~3]{Alvarez2001}.} 

Let $\phi(\cdot)$ denote the decreasing fundamental solution to the
ordinary differential equation $(\cL_{X}-r)\phi(x)=0$, where $\cL_{X}$
is the generator of $(X_{t})_{t\ge0}$ given by 
\begin{equation}
\cL_{X}:=\frac{1}{2}\sigma^{2}(x)\frac{\partial^{2}}{\partial x^{2}}+\mu(x)\frac{\partial}{\partial x}.\label{eq:charac-op}
\end{equation}

\begin{assumption} \label{assump:Li}$\ $

(i) For each $\theta\in[\theta_{L},\theta_{U}]$, there exists a critical
value $c(\theta)\in\cI\cup\{x_{U}\}$ such that $D(x)\le r\theta$
for $x<c(\theta)$ and $D(x)>r\theta$ for $x>c(\theta)$.

(ii) For each $\theta\in[\theta_{L},\theta_{U}]$, the function 
\begin{equation}
a_{\theta}(x):=\frac{\theta-d(x)}{\phi(x)}\label{eq:a-theta}
\end{equation}
attains a unique global maximum at $\alpha(\theta)\in\cI$, is differentiable
at $\alpha(\theta)$ and is increasing for $x<\alpha(\theta)$.

(iii) Function $\alpha:[\theta_{L},\theta_{U}]\to\cI$ defined in
(ii) is strictly increasing. \end{assumption}

\begin{remark}\label{rem:a'} A sufficient condition for Assumption
\ref{assump:Li}(ii)-(iii) is that, for each $\theta\in[\theta_{L},\theta_{U}]$,
the function $x\mapsto a_{\theta}(x)$ is continuously differentiable
and there is $\alpha(\theta)\in\cI$ such that $a_{\theta}'(x)>0$
for $x<\alpha(\theta)$ and $a_{\theta}'(x)<0$ for $x>\alpha(\theta)$.
These conditions immediately give (ii). For (iii), fix $\theta$ and
$x^{*}=\alpha(\theta)$. Take any $\theta'>\theta$ and notice that
\[
a_{\theta'}'(x^{*})=a_{\theta}'(x^{*})+\frac{(\theta'-\theta)\phi'(x^{*})}{\phi^{2}(x)}<a_{\theta}'(x^{*})=0,
\]
because $\phi$ is a strictly decreasing function, so $\phi'<0$.
This implies that $x^{*}<\alpha(\theta')$, hence the required monotonicity.
\end{remark}

\begin{lemma} \label{lemm:conti-alpha} Under Assumptions \ref{ass:functional},
\ref{ass:coeff} and \ref{assump:Li}, we have: 
\begin{itemize}
\item[(i)] for each $\theta\in[\theta_{L},\theta_{U}]$, the optimal policy
is to exit at the stopping time 
\begin{equation}
\tau_{\theta}^{*}=\inf\left\{ t\geq0:\,X_{t}\leq\alpha(\theta)\right\} \:,\label{eq:Tau_i}
\end{equation}
where $\alpha(\theta)$ is defined in Assumption \ref{assump:Li}.
Furthermore, $u(x;\theta)$ is given by 
\begin{equation}
u(x;\theta)=\begin{cases}
a_{\theta}(\alpha(\theta))\phi(x)+d(x) & \text{for}\;x>\alpha(\theta)\\
\theta & \text{for}\;x\le\alpha(\theta)
\end{cases},\label{eq:U*}
\end{equation}
and $u(x;\theta)>\theta$ for all $x>\alpha(\theta)$. 
\item[(ii)] $u(x;\theta)$ is continuous in $\theta$. 
\item[(iii)] Function $\alpha$ is continuous and $\alpha(\theta)\le c(\theta)$.
%\item[(iv)] $u(\cdot; \theta) \in C^1(\cI) \cap C^2(\cI \setminus \{ \alpha(\theta) \})$ and the second order derivative in $x$ lies in $L^\infty_{loc}(\cI)$.
\end{itemize}
\end{lemma} 
\begin{proof}
Statement (i) follows directly from \cite[Thm.~3]{Alvarez2001}. For
the second part of (iii), we rewrite \eqref{eq:U-theta} as 
\[
J(x,\tau;\theta):=\theta+\E_{x}\Big[\int_{0}^{\tau}e^{-rs}\big(D(X_{s})-r\theta\big)ds\Big],
\]
from which it is clear that stopping is not optimal whenever $D(X_{s})-r\theta>0$.
The continuity of $\alpha$ is proved by contradiction. Assume that
there is a sequence $(\theta_{n})\subset[\theta_{L},\theta_{U}]$
converging to $\theta$ and $\alpha(\theta_{n})\to\hat{\alpha}\ne\alpha(\theta)$.
Since $\alpha(\theta_{n})$ is a global maximum of $a_{\theta_{n}}$,
we have $a_{\theta_{n}}(\alpha(\theta_{n}))\ge a_{\theta_{n}}(\alpha(\theta))$.
The mapping $(x,\theta)\mapsto a_{\theta}(x)$ is continuous, so $a_{\theta_{n}}(\alpha(\theta))\to a_{\theta}(\alpha(\theta))$
and $a_{\theta_{n}}(\alpha(\theta_{n}))\to a_{\theta}(\hat{\alpha})$.
This means that $a_{\theta}(\hat{\alpha})\ge a_{\theta}(\alpha(\theta))$.
This contradicts that $\alpha(\theta)$ is the unique global maximum
of $a_{\theta}(\cdot)$.

Statement (ii) can be deduced from the explicit formula \eqref{eq:U*}
and the continuity of $\alpha(\cdot)$.
\end{proof}
%%%%%%%%%%%%%%%%%%%%%%%%%%%%%%%%%%%%%%%%%%%%%%%%%%%%%%%%%%%

\edt{We turn the attention to an example on which we will illustrate
our theory.}

\edt{
\begin{example}
\label{xmp:1} Consider a geometric Brownian motion
$(X_{t})$, i.e., a solution to the SDE given by $dX_{t}=\mu X_{t}dt+bX_{t}dW_{t}$
for some $\mu<0$ and $b>0$. Its generator takes the form 
\[
\mathcal{L}_{X}=\frac{1}{2}b^{2}x^{2}\frac{d^{2}}{dx^{2}}+\mu x\frac{d}{dx}.
\]
The fundamental solutions to the differential equation $(\mathcal{L}_{X}-r)\varphi(x)=0$
are $\psi(x)=x^{\gamma_{+}}$ and $\phi(x)=x^{\gamma_{-}}$, where
\[
\gamma_{\pm}=\frac{1}{2}-\frac{\mu}{b^{2}}\pm\sqrt{(\frac{1}{2}-\frac{\mu}{b^{2}})^{2}+\frac{2r}{b^{2}}}\:.
\]
From $\mu<0$, it is obvious that $\gamma_{+}>1$ and $\gamma_{-}<0$,
i.e., $\psi$ is increasing while $\phi$ is decreasing.

We state now all conditions on the coefficients which will be required
in this example: 
\begin{equation}
\beta\in(0,1),\qquad r>\delta,\qquad\beta-1>\gamma_{-}>-1,\qquad\beta b^{2}|\gamma_{-}|<2r,\label{eqn:ass_exmpl}
\end{equation}
where $\delta=\beta\mu+b^{2}\beta(\beta-1)/2$ and $\beta$ will be used in the statement of the profit flow $D$ and $M$ below. Notice that the
second condition is automatically satisfied on a declining market,
$\mu\le0$, because $\delta\le0$ while $r$ is required to be strictly
positive. We keep it for reference.

We examine the case in which the duopoly and monopoly profit flows
are given by 
\[
D(x)=\begin{cases}
x^{\beta} & x\in(0,x_{M}]\\
x_{M}^{\beta} & x>x_{M}
\end{cases}\:,\qquad M(x)=D(x)+M_{0},
\]
for some fixed (large) $x_{M}$ to be determined later. We assume
that $M_{0}>r\theta_U$ so that $m(x)>\theta$ for all $\theta \in [\theta_L, \theta_U]$. We also assume a sufficiently
large value of $x_{M}$ so that $x_{M}^{\beta}>r\theta_U$. 
The function $D(\cdot)$ is strictly increasing for $x<x_{M}$ and
constant for $x\ge x_{M}$. This form of $D$ is economically realistic
because it is impossible to achieve an unboundedly large value of
profit stream. Appendix \ref{app:example} provides the proof that
this example satisfies all the assumptions of the paper and derives
the explicit form of $d(\cdot)$ and $m(\cdot)$.

\begin{figure}[tb]
\begin{centering}
{\includegraphics[scale=0.25]{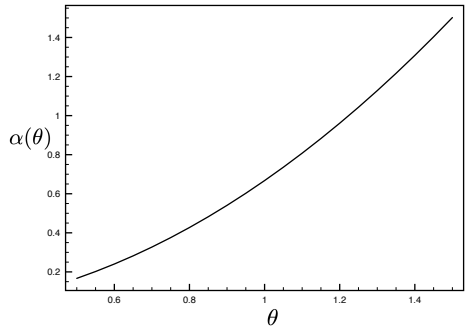}} 
\par\end{centering}
{\caption{Optimal stopping threshold $\alpha(\theta)$ for single player problem
with exit value $\theta$.}
\label{fig:alpha-theta}} 
\end{figure}

For numerical illustration, we examine the case $\mu=-0.5$, $b=1$,
$r=1$, $\beta=0.5$, $\theta_{L}=0.5$, $\theta_{U}=1.5$, $M_{0}=2$, and $x_{M}=1000$. It can
be verified that $\gamma_{-}=-0.732$, so it satisfies $\beta-1>\gamma_{-}>-1$,
and $\beta b^{2}|\gamma_{-}|=0.366<2=2r$. Furthermore,
$r>\delta=-.375$. Figure \ref{fig:alpha-theta} shows a graph
of $\alpha(\theta)$, the optimal stopping threshold for a single
player problem. As expected, it is an increasing function
because players with higher exit values exit earlier.
\end{example}
}

\section{Symmetric equilibrium}

\label{sec:equil} In this section, we construct a Nash equilibrium
in the exit game introduced in Section \ref{sec:model}. Apart from
all assumptions introduced so far in Sections \ref{sec:model} and
\ref{sec:single_player}, we make an additional standing assumption:
\begin{assumption}\label{ass:theta_cdf} Random variables $\theta_{i}$,
$i=1,2$, have the same cumulative distribution function $F$ which
is strictly increasing and continuous on its support $[\theta_{L},\theta_{U}]$.
\end{assumption}

\subsection{The strategy profile}

\label{subsec:The-strategy-profile}

We will now introduce a symmetric strategy profile which will be shown
to be a Nash equilibrium in the sense of Def.~\ref{def:Nash} \edt{as well as a perfect Bayesian equilibrium. We
start from an intuitive derivation of the form of such a strategy
profile and then provide a formal mathematical definition. Notice
that in equilibrium a player of type $\theta'$ would want to exit
earlier than a player of type $\theta''$ if $\theta'>\theta''$ because
a player would exit earlier if his outside option is more attractive.
Therefore, it is natural that the symmetric equilibrium strategy $\hat{\tau}(\cdot,\theta)$
(c.f. Corollary \ref{corr:equil}) should be monotone in the type
$\theta$. Based on the monotone property of $\hat{\tau}$, we
can hypothesise the existence of a well-defined stochastic process $(Y_{t})$ that
has the interpretation of the highest value of the remaining type for both players.
Thus, the posterior distribution of the remaining types at any point in time can be succinctly characterised by $(Y_{t})$ alone. Furthermore, it would be natural (bar technical difficulties) that $(Y_t)$ defined the strategy $\hat \tau$ via its inverse: $\hat \tau (\cdot, \theta) = \inf \{ t \ge 0: Y_t < \theta \}$.}

The above intuitive arguments motivate the introduction of player $i$'s strategy of the form 
\begin{equation}\label{eqn:tau_form}
\tau_{i} =\inf\{t\ge0:\ Y_{t}<\theta_{i}\},\quad i=1,2,
\end{equation}
where the process $(Y_{t})_{t\ge0}$ is decreasing, $(\cF_{t})$-adapted
and right-continuous with values in $[\theta_{L},\theta_{U}]$ such
that $Y_{0-}=\theta_{U}$. With an abuse of notation, we will treat
$(Y_{t})_{t\ge0}$ as a process on $(\tilde{\Omega},\tilde{\cF},\tilde{\P})$
when necessary; clearly, it is $(\tilde{\cF}_{t})$-adapted then.
Parametrisation \eqref{eqn:tau_form} of the strategy reflects the intuitive
meaning of the process $(Y_t)$ introduced above: on $\{Y_{t}=y\}$ all players of type $\theta>y$ have left
the game before or at time $t$. 

From a mathematical perspective the process $(Y_{t})$ is inconvenient
to work with as it starts from $Y_{0-}=\theta_{U}$ \edt{and decreases
to $\theta_{L}$, i.e, it depends explicitly on the support $[\theta_{L},\theta_{U}]$
and the distribution of types; see the dynamics \eqref{eqn:dynamics_Y}
of $(Y_{t})$ in the case when the type $\theta_{i}$ has an absolutely
continuous cumulative distribution function $F$.}
It turns out that a more convenient parametrisation is given by 
\[
A_{t}=\begin{cases}
-\log(F(Y_{t})), & Y_{t}>\theta_{L},\\
\infty, & Y_{t}=\theta_{L}.
\end{cases}
\]
%where $F$ is the cumulative distribution function of $\theta_{i}$, see Assumption \ref{ass:theta_cdf}. 
We postulate a strategy profile defined in terms of $(A_{t})$ as
\begin{equation}
dA_{t}=\lambda(X_{t},Y(A_{t}))dt,\label{eqn:ODE_A}
\end{equation}
where 
\begin{equation}
Y(a)=F^{-1}(e^{-a}),\qquad a\ge0,\label{eqn:Y_def}
\end{equation}
and 
\begin{equation}
\lambda(x,y)=\frac{ry-D(x)}{m(x)-y}\ind{x\le\alpha(y)}.\label{eqn:lambda}
\end{equation}
\edt{By Assumption \ref{assump:Li} and Lemma \ref{lemm:conti-alpha}, the numerator of the function $\lambda$ is non-negative. The denominator is positive as $m(x) > \theta_U$ for any $x \in \cI$, see Assumption \ref{ass:functional}.}
\edt{Intuitively, $\exp(-A_{t})=F(Y_{t})$ has the
interpretation of the proportion of the types that remain in the game
at time $t$. It follows that $\lambda(X_{t},Y(A_{t}))$ has
the interpretation of the rate of exit of an opponent at time $t$.}

\edt{The heuristic motivation for the above form of the Nash equilibrium will be provided in Section \ref{subsec:heuristic} with formal mathematical
derivation presented in Section \ref{sec:uniqeness} along the uniqueness results; arguments
used there require properties of the process $(A_{t})$ and of the
best response stopping problems which we derive Sections \ref{subsec:A_t}-\ref{subsec:proofs}. }

For notational convenience, we introduce the following
function:
\begin{equation}
A(y)=-\log(F(y)),\qquad y\in(\theta_{L},\theta_{U}],\label{eqn:notation_A}
\end{equation}
and $A(\theta_{L})=\infty$. Notice that $a\mapsto Y(a)$ and $y\mapsto A(y)$
are decreasing and continuous functions on their domains. Furthermore,
$A(Y(a))=a$ and $Y(A(y))=y$.

\edt{We remark that the stopping times $\tau_i$ have an equivalent representation (c.f. Lemma \ref{lem:S_v_rewritten})
\begin{equation}
\tau_{i}=\inf\{t\ge0:\ Y_{t} \le \theta_{i}, X_{t} \le \alpha(\theta_{i})\},\quad i=1,2,\label{eqn:tau_PBE}
\end{equation}
which naturally links with the classical form of a solution to the best response optimal stopping problem and furnishes a perfect Bayesian equilibrium (PBE). The condition $Y_t \le \theta_i$ originates from the notion that $Y_t$ is the maximum type remaining in the game; if $Y_t$ hits 
$\theta_i$, it signals that it is time for player of type $\theta_i$ to exit. The other condition suggests that a player exits only when $X_t \le \alpha(\theta_i)$; this condition originates from the interpretation of $\alpha(\theta_i)$ as the exit threshold in the single-player problem. If $X_t > \alpha(\theta_i)$, the prospective profit stream is sufficiently large that a player does not have an incentive to exit.

The representation \eqref{eqn:tau_PBE} does not allow us to derive an optimisation problem in which the type of the opponent is integrated out and the action of his stopping time replaced by appropriate functional of the process $(Y_t)$, see Lemma \ref{lem:J_int}. In the following subsections the reader will notice the importance of this detail and how it is overcome in order to establish the PBE. Indeed, in Sections \ref{subsec:A_t}-\ref{subsec:Nash} we prove that $(\tau_1, \tau_2)$ is a Nash equilibrium; this is followed in Section \ref{subsec:PBE} by arguments showing that \eqref{eqn:tau_PBE} yields a PBE.}

\begin{excont}[continued]

\edt{For intuitive understanding, we now return to the
numerical example introduced in Section \ref{sec:single_player} and
assume that the type $\theta_{i}$ is uniformly distributed within
an interval $[\theta_{L},\theta_{U}]$; recall that $\theta_{L}=0.5$ and
$\theta_{U}=1.5$. 
% Note that $\theta_{U}>\theta_{L}>0$ and $M_{0}>r\theta_{U}$ so that $m(x)>\theta$ for all $\theta\in[\theta_{L},\theta_{U}]$. Also note that $x_{M}=1000$ is sufficiently large so that $x_{M}^{\alpha}>r\theta_{U}$.
A simulated realisation of the game is presented in Figure \ref{fig:X}.
It illustrates a sample path of $(X_{t}, Y_{t})$ and $\alpha(Y_{t})$
with initial conditions $X_{0}=2.72$ and $Y_{0}=\theta_{U}$ over
a time interval $[0,2]$.
\begin{figure}[tb]
\begin{subfigure}[b]{0.47\textwidth} \centering \includegraphics[scale=0.25]{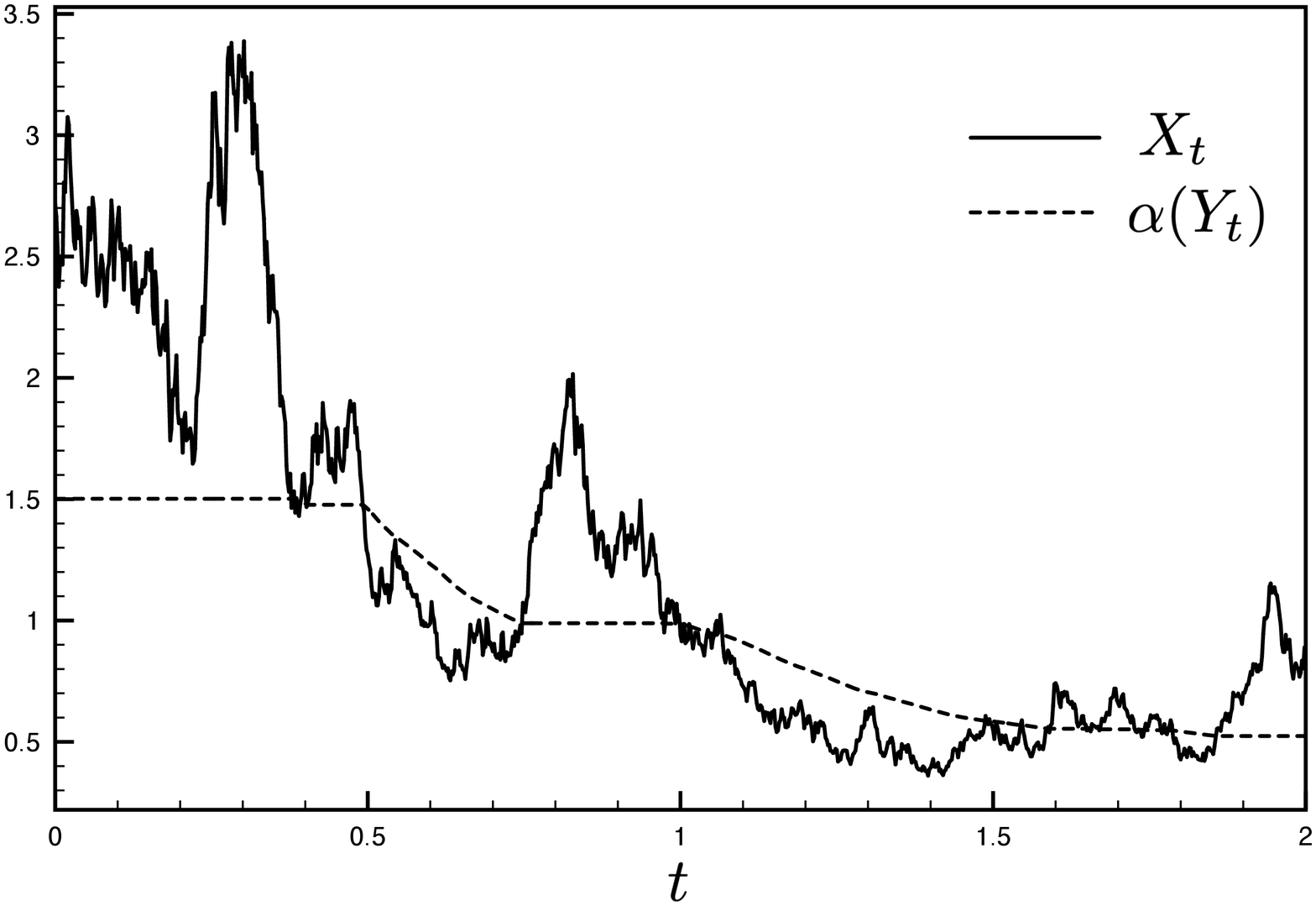}
\subcaption{A sample path of $X_{t}$ and $\alpha(Y_{t})$.} \end{subfigure}
\begin{subfigure}[b]{0.47\textwidth} \centering \includegraphics[scale=0.25]{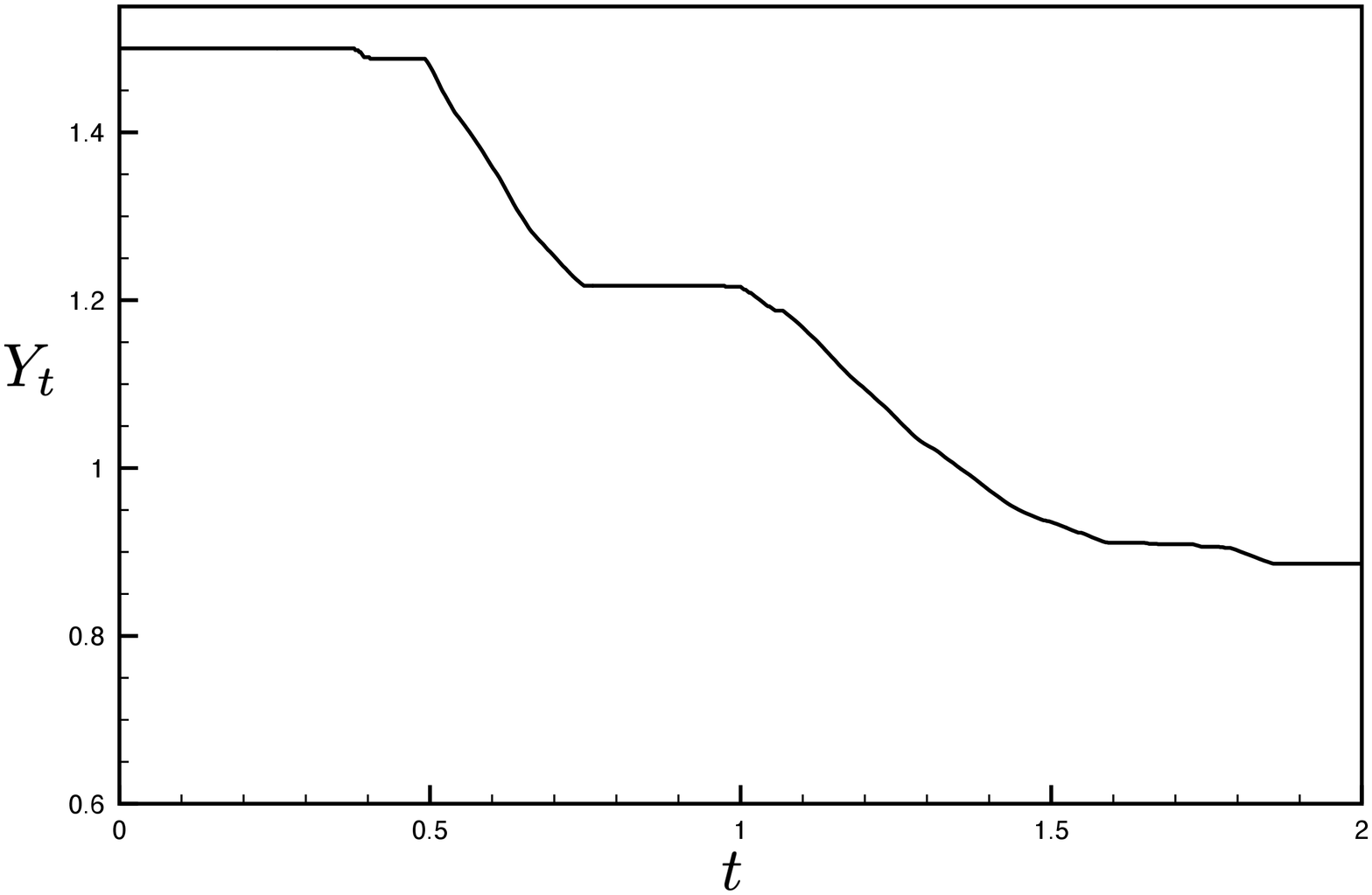}
\subcaption{A sample path of $Y_{t}$.} \end{subfigure} \caption{Example of the evolution of $(X_{t})$ and the corresponding processes
$(Y_{t})$ and $\alpha(Y_{t})$ with initial conditions $X_{0}=2.72$
and $Y_{0}=\theta_{U}$ over time interval $[0,2]$.}
\label{fig:X} 
\end{figure}

The functional form of $\lambda(\cdot,\cdot)$ given in (\ref{eqn:lambda})
suggests that there are two regions: an exit region $X_{t}\le\alpha(Y_{t})$
and a no-action region $X_{t}>\alpha(Y_{t})$. According to the strategy
profile \eqref{eqn:tau_PBE}, players may exit only when $X_{t}\le\alpha(Y_{t})$, which
results in a positive value of $\lambda(X_{t},Y_{t})$. This feature
of the strategy profile is illustrated by Figure \ref{fig:X}, where
$(Y_{t})$ decreases only when $X_{t}\le\alpha(Y_{t})$. For instance,
in the time intervals $[0,0.378]$ and $[0.748, 0.974]$, $X_{t}$
stays above $\alpha(Y_{t})$, so $Y_{t}$ stays constant within these
intervals. In contrast, $X_{t}\le\alpha(Y_{t})$ in the intervals $[0.494,
0.746]$ and $[1.068, 1.486]$, so $(Y_{t})$ declines steadily in
time.

Next, we illustrate an individual player's strategy. A
player of type $\theta\ge Y_{t}$ follows a threshold exit strategy: to exit
as soon as $X_{t}\le\alpha(\theta)$. This is intuitively consistent
with the optimal policy for a single-player case. On
the other hand, any player of type $\theta'<Y_{t}$ will wait until
$Y_{t}$ hits $\theta'$, at which point in time he will also adopt
a threshold exit strategy, i.e., to exit when $X_{t}\le\alpha(\theta)$. 
\edt{In Nash equilibrium, one would never encounter a player of type $\theta > Y_t$; when $\theta = Y_t$ and $X_t \le \alpha(\theta)$, the dynamics \eqref{eqn:ODE_A} of $A_t$ implies that $Y_s = Y(A_s) < \theta$ for any $s > t$, so infimum of times that the condition $Y_s < \theta$ is satisfied is $t$ which justifies the equivalence of definitions \eqref{eqn:tau_form} and \eqref{eqn:tau_PBE}. This equivalence does not hold when $\theta > Y_t$, so only a strategy of the form \eqref{eqn:tau_PBE} defines a PBE.}

In the example shown in Figure \ref{fig:X}, a player of type $\theta= 1.5$
does not exit until $X_{t}$ hits $\alpha(1.5)$ at $t=0.378$. Because
a player of type $1.5$ is supposed to be the first type to exit under
the prescribed strategy profile ($\theta_{U}=1.5$), he exits as soon
as $X_{t}\le\alpha(1.5)=\alpha(Y_{0})$ is satisfied. If the process
$(X_{t})$ had started out below $\alpha(1.5)$, then the player would
have exited right away at time $t=0$.

On the other hand, any player of type $\theta<Y_{0}$ has to wait
beyond $t=0.378$ because his prescribed time of exit is $\hat{\tau}(\cdot,\theta)=\inf\{t\ge0:Y_{t} < \theta\}$.
It follows that Figure \ref{fig:X}(B) can be utilised to determine
the time of exit for any given type: we simply have to invert the
$t-Y_{t}$ graph into $Y_{t}-t$ graph and relabel the horizontal
variable as $\theta$ and the vertical variable as $\hat\tau(\cdot, \theta)$.
The result is Figure \ref{fig:tau-th}. It illustrates the property
of the strategy profile that the exit time is monotonically decreasing
in the type.

\begin{figure}[tb]
\begin{centering}
{\includegraphics[scale=0.35]{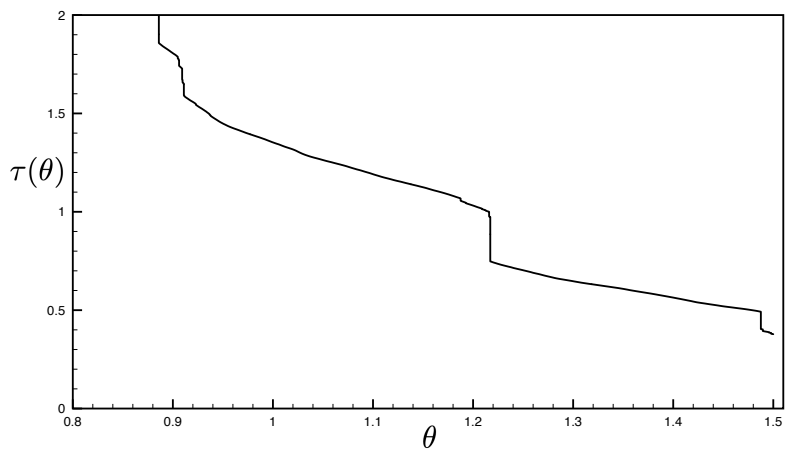}} 
\par\end{centering}
{\caption{Simulated values of $\hat{\tau}(\cdot,\theta)$.}
\label{fig:tau-th}} 
\end{figure}

The PBE strategy \eqref{eqn:tau_PBE} of type $\theta$ can be succinctly represented as an exit region in the $x$-$y$ space defined as $\{(x,y):x\le\alpha(\theta),y\le\theta\}$. Figure \ref{fig:Y-X} shows a simulated sample path of $(X_{t},Y_{t})$
as well as the exit region for type $\theta=1$ indicated by the shaded rectangle. The player of type $\theta=1$ exits
as soon as $(X_{t},Y_{t})$ hits the shaded exit region, which takes
place at time $t=1.352$ when $X_{t}=0.432$ and $Y_{t}=1$.

\begin{figure}[tb]
\begin{centering}
{\includegraphics[scale=0.35]{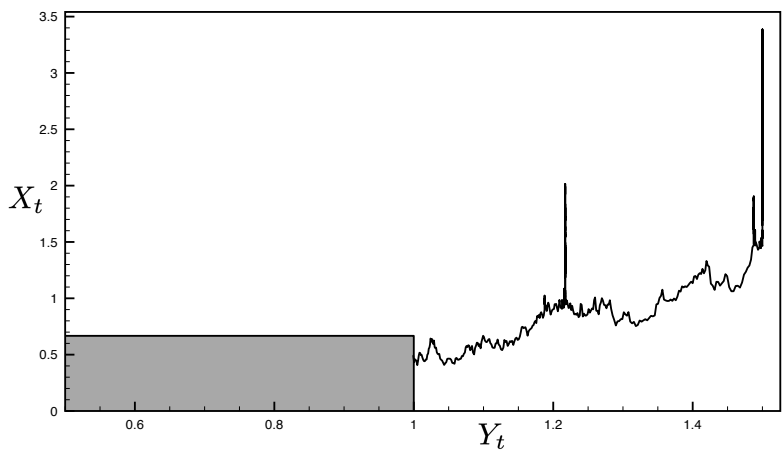}}
\par\end{centering}
{\caption{Simulated path of $(X_t, Y_t)$. The exit region for $\theta=1$ is shaded.}
\label{fig:Y-X}}
\end{figure}

Next, we illustrate the temporal evolution of
a game by analysing the dynamics of $(X_{t},Y_{t})$ and the prescribed
strategy for each player. Suppose that $\theta_{1}=1.4$ and $\theta_{2}=1$.
At the beginning of the game, each player knows his own type but not
his opponent's. However, they both know the strategy profile and the
initial probability distribution of their opponent's type, just as
in the standard game-theoretic assumption. According to the strategy
profile, player $i$ exits at the earliest time such that $Y_{t}<\theta_{i}$. As the game progresses, each player
observes whether his opponent exits or not. If both players remain
in the game until time $t$, they update their posterior beliefs about
their opponent's type by using the dynamics of $(Y_{t})$ given by \eqref{eqn:ODE_A}.
Finally, Player 1 exits at time $\hat{\tau}(\cdot,1.4)=0.564$
(see Figure \ref{fig:tau-th}), revealing his type publicly. Player 2 never exits
under this scenario, and he consequently enjoys
the monopoly from time $\hat{\tau}(\cdot,1.4)$ onwards.
}
\end{excont}

As indicated before (see Corollary \ref{corr:equil}), the proof that
the strategy profile given by \eqref{eqn:tau_form} with $Y_{t}=Y(A_{t})$,
for $A_{t}$ defined above, is a Nash equilibrium will require examination
of the optimal stopping problem with the functional $J(x,\sigma,\tau_{i};\gamma)$
defined in \eqref{eqn:J}. Recall that $\tau_{i}$ depends on $\theta_{i}$
which is not observable by the rival firm; formally, $\theta_{i}$
is independent from $\cF$. We will therefore integrate out $\theta_{i}$
in the following lemma. The statement is formulated for a general
process $(A_{t})_{t\ge0}$ as this does not lead to any additional
complications in the proof.

\begin{lemma}\label{lem:J_int} Let $\tau_{i}$ be of the form \eqref{eqn:tau_form}
with $Y_{t}=Y(A_{t})$ for a process $(A_{t})_{t\ge0}$ which is $(\cF_{t})$-adapted,
right-continuous and increasing with $A_{0-}=0$. For any $(\cF_{t})$-stopping
time $\sigma$ we have 
\begin{equation}
J(x,\sigma,\tau_{i};\gamma)=\E_{x}\Big[\int_{0}^{\sigma}e^{-rs-A_{s}}D(X_{s})ds+\gamma e^{-r\sigma-A_{\sigma-}}+\int_{[0,\sigma)}e^{-rs-A_{s}}m(X_{s})dA_{s}\Big].\label{eqn:J_int}
\end{equation}
\end{lemma} Proof of this and other technical results are collected
in Section \ref{subsec:proofs}.

\subsection{Heuristic derivation of $\lambda(x,y)$}

\label{subsec:heuristic}

\edt{Before we delve into the mathematical proof of
the equilibrium, we provide a heuristic derivation of the form of
$\lambda(x,y)$ given in (\ref{eqn:lambda}). We first assume a current
value of $Y_{t}=\theta_{c}$ at time $t$ and consider the regime
$X_{t}>\alpha(Y_{t})=\alpha(\theta_{c})$. Recall that we have established
that the optimal policy for a single-player problem is never to exit
for $X_{t}>\alpha(\theta_{c})$ if the player is of type $\theta_{c}$.
Even if his opponent is present, this optimal policy does not change
for $X_{t}>\alpha(\theta_{c})$ as the opponent's action can only increase the profit flow compared to duopoly payoff in the single-player problem. Hence, a player of type $\theta_{c}$ should not exit in equilibrium if $X_{t}>\alpha(\theta_{c})$. From
the monotone property of $\theta\mapsto\hat{\tau}(\cdot,\theta)$,
it follows that any type $\theta'<\theta_{c}$ should not exit if
$X_{t}>\alpha(\theta_{c})$. Since all types of $\theta'>\theta_{c}$
already exited in the past, no one exits for as long as $X_{t}>\alpha(Y_{t})$.
It follows that $\lambda(X_{t},Y_{t})=0$ if $X_{t}>\alpha(Y_{t})$.

We now consider the exit region $X_{t}\le\alpha(\theta_{c})$.
By the arguments established above, it is the type $\theta_{c}$ that
should decide when to exit; the types $\theta<\theta_{c}$ would first
wait until after $\theta_{c}$ exits, and the types $\theta>\theta_{c}$
should have already exited by now. Thus, we impose the condition that
the best response of a type $\theta<\theta_{c}$ is to wait at least
an infinitesimal time while the best response of a type $\theta_{c}$
is to exit immediately. 
%Hence, this condition can be expressed in terms of the payoff functions.

Next, let $X_{t}=x \le \alpha(\theta_c)$ and $A_{t}=A(\theta_{c})$. As argued
before it is suboptimal to delay exit by a small time $\delta t>0$
for a player of type $\theta_{c}$. Using \eqref{eqn:J_int} and \eqref{eqn:ODE_A}, we have
\begin{align*}
 & 0\ge\\
 & \E\Big[\int_{0}^{t+\delta t}e^{-rs-A_{s}}D(X_{s})ds+\theta_c e^{-r(t+\delta t)-A_{(t+\delta t)-}}+\int_{[0,t+\delta t)}e^{-rs-A_{s}}m(X_{s})dA_{s}\Big|X_{t}=x,A_{t}=A(\theta_{c})\Big]\\
 & -\E\Big[\int_{0}^{t}e^{-rs-A_{s}}D(X_{s})ds+\theta_c e^{-rt-A_{t-}}+\int_{[0,t)}e^{-rs-A_{s}}m(X_{s})dA_{s}\Big|X_{t}=x,A_{t}=A(\theta_{c})\Big]\\
 &\hspace{-3pt} =e^{-rt-A_{t-}}\Big(\theta_c+\delta t\big[D(x)+\lambda(x,\theta_{c})m(x)-\theta_c(r+\lambda(x,\theta_{c}))]+\mathcal{O}\big((\delta t)^{2}\big)\Big).
\end{align*}
From this, the leading-order term of $\delta t$ must be non-positive,
which yields the inequality 
\begin{equation}
\lambda(x,\theta_{c})\le\frac{r\theta_c-D(x)}{m(x)-\theta_c}.\label{eq:less-th}
\end{equation}
On the other hand, it is suboptimal for Player $i$ of type $\theta<\theta_{c}$
to exit immediately at $t$. Assuming that waiting an infinitesimally
short time $\delta t>0$ is strictly better (which we leave without
any formal justification), an analogous argument as above gives 
\begin{equation}
\lambda(x,\theta_{c})>\frac{r\theta-D(x)}{m(x)-\theta},\qquad\theta<\theta_{c}.\label{eq:larger-th}
\end{equation}
Note that both conditions (\ref{eq:less-th}) and (\ref{eq:larger-th})
automatically enforce that $\lambda(x,\theta_{c})=(r\theta_{c}-D(x))/(m(x)-\theta_{c})$.
From the arbitrariness of $\theta_{c}$ and $x\le\alpha(\theta_{c})$
and together with the condition that $\lambda(x,y)=0$ for $x>\alpha(y)$,
we finally conclude the form (\ref{eqn:lambda}) for $\lambda$.}

\subsection{Construction and properties of $(A_{t})_{t\ge0}$}\label{subsec:A_t}

Standard theory of ODEs cannot be applied to obtain existence and
uniqueness of solutions to \eqref{eqn:ODE_A} because the function
$\lambda$ is discontinuous for each trajectory of $(X_{t})$. Instead,
we construct the process $(A_{t})$ that satisfies the integral form
of \eqref{eqn:ODE_A}. We start with a number of technical results.

We first review the continuity of the process $(X_{t}^{x})_{t\ge0}$
with respect to the initial value $x$. \cite[Thm.~1, p.~102]{krylov}
implies that for any $T>0$ and $q\ge1$ we have for any $x^{n}\to x$
in $\cI$, 
\begin{equation}
\lim_{n\to\infty}\E\big[\sup_{s\in[0,T]}|X_{s}^{x_{n}}-X_{s}^{x}|^{q}\big]=0.\label{eqn:cont_X}
\end{equation}
The dependence of $X_{t}^{x}$ on the initial point $x$ is monotone
due to the comparison principle for diffusions. Hence, by the monotone
convergence theorem, \eqref{eqn:cont_X} implies that if the sequence
$(x_{n})$ is monotone, then for each $T>0$, there is a $\P$-negligible
subset of $\Omega$ outside of which 
\begin{equation}
\lim_{n\to\infty}\sup_{s\in[0,T]}|X_{s}^{x_{n}}-X_{s}^{x}|=0.\label{eqn:cont_X_mono}
\end{equation}

\edt{The infinite variation of trajectories of the process $(X_{t})$,
which follows from the non-degeneracy of the diffusion coefficient
$\sigma$, yield the following result.}

\begin{lemma}\label{lem:equal_null} For any \cadlag finite variation
$(\cF_{t})$-adapted process $(\varphi_{t})$, we have 
\[
\int_{0}^{\infty}\ind{X_{s}=\varphi_{s}}ds=0,\quad\P_{x}-a.s.
\]
for any $x\in\cI$. \end{lemma}

\edt{Thanks to this lemma, modification of $\lambda$ on the right-hand
side of \eqref{eqn:ODE_A} on a (countable) number of curves of the
form $x=h(y)$ for continuous $h$ does not affect the
solution in the sense that if a process $(A_{t})$ satisfies \eqref{eqn:ODE_A}
then so it does with the modified $\lambda$. This will play an important
role in the construction of the solution in Proposition \ref{prop:solution_A}
as well as in Section \ref{sec:uniqeness} in which uniqueness of
equilibrium is established.}

\begin{lemma}\label{lem:lambda} For any $x\le\alpha(\theta_{U})$,
the mapping 
\[
[\theta_{L},\theta_{U}]\ni y\mapsto l(x,y):=\frac{ry-D(x)}{m(x)-y}
\]
is strictly increasing with the derivative 
\begin{equation}
0<\frac{rm(x)-D(x)}{(m(x)-y)^{2}}\le\frac{r\,m_{\min}}{(m_{\min}-\theta_{U})^{2}},\label{eqn:derivative_lambda}
\end{equation}
where $m_{\min}=\inf_{x\in\cI}m(x)$. \end{lemma} \begin{lemma}\label{lem:monotonicity_lambda}
The mapping $(x,a)\mapsto\lambda(x,Y(a))$ is decreasing in $x$ and
$a$. \end{lemma}

\edt{The above basic properties of the expression defining $\lambda$ in \eqref{eqn:lambda}
are key to the construction of a solution to \eqref{eqn:ODE_A} as
well as to the study of the best response optimal stopping problems.
Notice also that $l(x,y)\ge0$ for $x\le\alpha(y)$ thanks to Lemma
\ref{lemm:conti-alpha}, so the function $\lambda$ is non-negative.
}

\begin{proposition}\label{prop:solution_A} There is a strong Markov
process $(X_{t},A_{t})_{t\ge0}$ such that $(X_{t})$ solves \eqref{eqn:X}
and $(A_{t})$ is a continuous process which satisfies 
\begin{equation}
A_{t}=A_{0}+\int_{0}^{t}\lambda(X_{s},Y(A_{s}))ds,\qquad t\ge0,\quad\P_{x}-a.s.\label{eqn:int_A}
\end{equation}
i.e., it is a Carath\'{e}odory solution to \eqref{eqn:ODE_A}. \end{proposition}

\edt{ The solution of \eqref{eqn:int_A} in the above proposition
is constructed as a limit of solutions $(A_{t}^{\ve})$ of ODEs with
the right-hand side $\lambda^{\ve}$. Functions $\lambda^{\ve}$ converge
from above to $\lambda$ and are Lipschitz continuous, so that $(A_{t}^{\ve})$
are uniquely determined for each $\omega$. The monotone limit $A^{0}=\lim_{\ve\downarrow0}A^{\ve}$
is shown to satisfy \eqref{eqn:int_A} with Lemma \ref{lem:equal_null}
playing an important role.

The process $(A_{t}^{0})_{t\ge0}$ is a Carath\'{e}odory solution to
\eqref{eqn:ODE_A}: it is a continuous process such that for
almost every $\omega$ and almost every $t$ the equality \eqref{eqn:ODE_A}
holds. The concept of Carath\'{e}odory solutions was introduced in the
theory of ODEs to make sense of equations with discontinuous right-hand
side, i.e., equations whose solutions cannot be continuously differentiable
functions. Here, we extended the notion to random ODEs by studying
the equation pathwise, for each $\omega$ (and the resulting trajectory
$t\mapsto X_{t}(\omega)$) separately. The reason that the equality
in \eqref{eqn:ODE_A} is $\P$-a.s. is due to the use of Lemma \ref{lem:equal_null}.
We finish with the following important remark.

\begin{remark} \label{rem:carath} There may be many Carath\'{e}odory
solutions to \eqref{eqn:int_A} but it follows from the above discussion
that $(A_{t}^{0})_{t\ge0}$ is the largest of them. \end{remark}
}

\edt{It should be noted, however, that this extremal property of the solution
does not play any role in our considerations below. It may well be
that there is a unique Carath\'{e}odory solution but the theory of such
equations with discontinuities driven by infinite variation process
$(X_{t})$ is not well developed and beyond the scope of this paper.

Notice that when the CDF $F$ of $\theta$ is differentiable then
$Y_{t}=Y(A_{t})$, where $(A_{t})$ satisfies \eqref{eqn:int_A},
is a Carath\'{e}odory solution to 
\begin{equation}
dY_{t}=-\frac{F(Y_{t})}{F'(Y_{t})}\lambda(X_{t},Y_{t})dt,\quad Y_{0}=Y(A_{0}).\label{eqn:dynamics_Y}
\end{equation}
As the mapping $Y$ is decreasing, it is the smallest Carath\'{e}odory
solution. }

We will indicate an initial point $(X_{0},A_{0})=(x,a)$ of the Markov
process $(X_{t},A_{t})_{t\ge0}$ either as a subscript in $\P_{xa}$
or in the process itself $X_{t}^{x},A_{t}^{x,a}$ and mix the notations
\edt{depending on which is beneficial for the clarity of exposition.} Since $(X_{t})_{t\ge0}$
is a strong solution of \eqref{eqn:X} the process $(X_{t},A_{t})_{t\ge0}$
can be considered as a family of processes on the original probability
space $(\Omega,\cF,\P)$ as well as a Markov family.

\edt{The final auxiliary result concerns the dependence of the process
$(A_{t})$ on the initial point $a$ and the initial state $x$ of $(X_{t})$.
The Lipschitz continuity with respect to $a$ will be instrumental
in proving the continuity of the value function to the best response
optimal stopping problem, a key result to prove the equilibrium property
of the postulated strategies.}

\begin{lemma}\label{lem:prop_A} For any $x\in\cI$, the mapping
$a\mapsto A_{t}^{x,a}=:A_{t}^{a}$ is increasing and for $a,a'\ge0$
\[
|A_{t}^{a}-A_{t}^{a'}|\le|a-a'|,\qquad t\ge0.
\]
For any $a\ge0$, the mapping $x\mapsto A_{t}^{x,a}$ is decreasing
and sequentially continuous in the following sense: for every $T>0$,
$x_{0}\in\cI$ and $x_{n}\to x_{0}$, there is a $\P$-negligible
subset of $\Omega$ outside of which $A_{t}^{x_{n},a}$ converges
to $A_{t}^{x_{0},a}$ in the supremum norm $\|Z\|_{T}=\sup_{t\in[0,T]}|Z_{t}|$.
\end{lemma}

\begin{remark}\label{rem:cont_X} It can be deduced from the above
lemma and its proof that the mapping $(x,a)\mapsto A_{t}^{x,a}$ is
sequentially continuous at each point $(\bar{x},\bar{a})\in\cI\times[0,\infty)$
outside of a $\P$-negligible set which depends on $(\bar{x},\bar{a})$
and the sequence. Indeed, the continuity in $a$ is uniform over $x$,
while by applying Lemma \ref{lem:prop_A} for $T=1,2,3,\ldots$, we
obtain the sequential continuity at $\bar{x}$ for any $t\ge0$. The
aforementioned negligible set arises because of the occupation measure
formula applied to a semimartingale $Z_{t}=X_{t}^{\bar{x}}-\alpha(Y(A_{t}^{\bar{x},\bar{a}}))$
which itself depends on $(\bar{x},\bar{a})$, and because of the convergence
of the process $(X_{t}^{x})$ which depends on the sequence $x_{n}\to\bar{x}$.
\end{remark}

\subsection{Best-response value function}\label{subsec:best_response}

Motivated by Corollary \ref{corr:equil}, we assume now that Player
2 follows the strategy $\tau_{2}^{*}$ given by \eqref{eqn:tau_form}
with $(A_{t})$ stated in \eqref{eqn:Y_def}. We will study the best
response of Player 1 when $\theta_{1}=\theta\in[\theta_{L},\theta_{U}]$,
i.e., the optimal stopping problem 
\[
\sup_{\sigma\in\cT(\cF_{t})}J(x,\sigma,\tau_{2}^{*};\theta).
\]
This problem does not have a structure of a Markovian optimal stopping
problem, but thanks to the strong Markov property of $(X_{t},A_{t})$
and the expression \eqref{eqn:J_int} for $J(\cdot)$, we \edt{will
solve it by considering} the following Markovian optimal stopping
problem on the extended state space $(X_{t},A_{t})$: 
\begin{equation}
v(x,a;\theta)=\sup_{\sigma\in\cT(\cF_{t})}\E_{xa}\Big[\int_{0}^{\sigma}e^{-rs-A_{s}}D(X_{s})ds+\theta e^{-r\sigma-A_{\sigma}}+\int_{0}^{\sigma}e^{-rs-A_{s}}m(X_{s})\lambda(X_{s},Y(A_{s}))ds\Big].\label{eqn:def_v}
\end{equation}
Notice that $(A_{t})$ is absolutely continuous, so $A_{t-}=A_{t}$.
By \eqref{eqn:ODE_A} and Lemma \ref{lem:J_int}, we have 
\[
v(x,0;\theta)=\sup_{\sigma\in\cT(\cF_{t})}J(x,\sigma,\tau_{2}^{*};\theta).
\]

We can rewrite the middle term of \eqref{eqn:def_v} in an integral
form 
\begin{equation}
\theta e^{-r\sigma-A_{\sigma}}=\theta e^{-A_{0}}-\int_{0}^{\sigma}e^{-rs-A_{s}}\theta\big(r+\lambda(X_{s},Y(A_{s}))\big)ds.\label{eqn:integral_middle}
\end{equation}
Letting 
\begin{equation}
\tilde{v}(x,a;\theta)=\sup_{\sigma\in\cT(\cF_{t})}\E_{xa}\Big[\int_{0}^{\sigma}e^{-rs-A_{s}}\big(D(X_{s})-r\theta+\lambda(X_{s},Y(A_{s}))(m(X_{s})-\theta)\big)ds\Big]\label{eqn:tilde_v}
\end{equation}
we have $v(x,a;\theta)=\tilde{v}(x,a;\theta)+\theta e^{-a}$. This
equivalent formulation of the optimal stopping problem will be used
often in the paper and it underlies the arguments of the following
proposition.

% \begin{lemma}\label{lem:decreasing_v}
% The value function $v$ is decreasing in $a$.
% \end{lemma}
% \begin{proof}
% Denote
% \[
% J(x, a, \sigma; \theta) = \E_{xa} \Big[ \int_0^\sigma e^{-rs - A_s} D(X_s) ds + \theta e^{-r\sigma-A_{\sigma}} + \int_0^\sigma e^{-rs - A_s} m(X_s) \lambda (X_s, Y(A_s)) ds \Big].
% \]
% Fix $a > a' \ge 0$. We have
% \begin{align*}
% J(x,a',\sigma; \theta) - J(x,a,\sigma; \theta)
% &= \E_x \bigg[ \int_0^\sigma e^{-rs} D(X_s) \big(e^{-A^{a'}_s} - e^{-A^a_s}\big)ds + \theta e^{-r\sigma} \big(e^{-A^{a'}_\sigma} - e^{-A^a_\sigma}\big) \bigg]\\
% &\hspace{12pt}+ \E_x \bigg[ \int_0\sigma e^{-rs} m(X_s) \big( \lambda(X_s, Y(A^{a'}_s)) e^{-A^{a'}_s} - \lambda(X_s, Y(A^{a}_s))e^{-A^{a}_s} \big)ds\bigg].
% \end{align*}
% By Lemma \ref{lem:prop_A}, $a \mapsto A^{a}_s$ is increasing, so $e^{-A^{a'}_s} \ge e^{-A^a_s}$. Using further that $a \mapsto \lambda(x, Y(a))$ is decreasing, we have $\lambda(X_s, Y(A^{a'}_s)) \ge \lambda(X_s, Y(A^{a}_s))$. Since $\lambda \ge 0$, we obtain $J(x,a',\sigma; \theta) - J(x,a,\sigma; \theta) \ge 0$. Hence, 
% \[
% v(x,a'; \theta) - v(x,a; \theta) \ge \inf_{\sigma \in \cT(\cF_t)} \big[ J(x,a',\sigma; \theta) - J(x,a,\sigma; \theta) \big] \ge 0.
% \]
% \end{proof}

\begin{proposition}\label{prop:continuity} The value function $v(x,a;\theta)$
is continuous in $(x,a)\in\cI\times[0,\infty)$. Furthermore, an optimal
stopping time for $v(x,a;\theta)$ is $\sigma^{*}=\inf\{t\ge0:\ (X_{t},A_{t})\in\cS_{\tilde{v}}^{\theta}\}$,
where 
\[
\cS_{\tilde{v}}^{\theta}=\{(x,a)\in\cI\times[0,\infty]:\ \tilde{v}(x,a;\theta)=0\}
\]
is a closed set. \end{proposition}

\edt{The main finding of the above proposition is the continuity of
the value function $v$, or, equivalently, the continuity of $\tilde{v}$.
The form of an optimal stopping time follows then from the standard
theory.

The continuity of $\tilde{v}$ in $(x,a)$ does not follow from standard
results because the functional is not continuous due to the discontinuity
of $\lambda$ and the process $(X_{t},A_{t})$ has not been shown to be Feller continuous. Instead, we approximate the value function $\tilde{v}$
from above and from below by value functions $\tilde{v}^{\ve}$, $\tilde{v}_{\ve}$
corresponding to optimal stopping problems with $\lambda$ in the
functional \eqref{eqn:tilde_v} (but not in the dynamics of $(A_{t})$)
replaced by continuous $\lambda^{\ve}$ from above (as in Proposition \ref{prop:solution_A})
and by continuous $\lambda_{\ve}$ from below (constructed analogously
as $\lambda^{\ve})$. We prove directly the continuity of $\tilde{v}^{\ve}$
and $\tilde{v}_{\ve}$ using Lemma \ref{lem:prop_A} and Eq.~\eqref{eqn:cont_X_mono}.
Thanks to the monotonicity of $\tilde{v}^{\ve}$ in $\ve$, which
follows immediately from the monotonicity of $\lambda^{\ve}$ in $\ve$
and the fact that $m>\theta_{U}$, the function $\tilde{v}^{0}=\liminf_{\ve\downarrow0}\tilde{v}^{\ve}$
is upper semi-continuous. Similarly, $\tilde{v}_{0}=\limsup_{\ve\downarrow0}\tilde{v}_{\ve}$
is lower semi-continuous. The proof is concluded by showing that $\tilde{v}=\tilde{v}^{0}=\tilde{v}_{0}$.}

\begin{remark} It would be tempting to apply a smoothing technique
with $\lambda^{\ve}$ and $\lambda_{\ve}$ to prove the continuity
of $(x,a)\mapsto A_{t}^{x,a}$. However, the use of $\lambda^{\ve}$
leads to the largest Carath\'{e}odory solution to \eqref{eqn:ODE_A},
while, by analogy, we expect $\lambda_{\ve}$ to yield the smallest
Carath\'{e}odory solution. This approach would, therefore, require proving
the uniqueness of Carath\'{e}odory solutions which is known to not be
true in general. \end{remark}

\subsection{Best response stopping set}

By Proposition \ref{prop:continuity}, the mapping $(x,a)\mapsto v(x,a;\theta)$
is continuous (as is $\tilde v(x,a;\theta)=v(x,a;\theta)-\theta e^{-a}$)
and the stopping sets for $\tilde{v}$ and $v$ coincide: $\cS_{\tilde{v}}^{\theta}=\cS_{v}^{\theta}$.
We denote by $\cC_{\tilde{v}}^{\theta}$ the continuation set, i.e.,
$\cC_{\tilde{v}}^{\theta}=\cI\times[0,\infty)\setminus\cS_{\tilde{v}}^{\theta}$.
We also have $u(x;\theta)=\tilde{u}(x;\theta)+\theta$, where \edt{$u$ is defined in \eqref{eq:U*} and}
\begin{equation}
\tilde{u}(x;\theta)=\sup_{\sigma\in\cT(\cF_{t})}\E_{x}\Big[\int_{0}^{\sigma}e^{-rs}\big(D(X_{s})-r\theta\big)ds\Big].\label{eqn:tl_u}
\end{equation}
Recall that the smallest optimal stopping time for $u(x;\theta)$
and $\tilde{u}(x;\theta)$ is 
\[
\eta=\inf\{t\ge0:\ \tilde{u}(X_{t};\theta)=0\}=\inf\{t\ge0:\ X_{t}\in\cS_{\tilde{u}}\},
\]
where $\cS_{\tilde{u}}=\{x\in\cI:\ x\le\alpha(\theta)\}$ is the stopping
set and its complement $\cC_{\tilde{u}}=\cI\setminus\cS$ is the continuation
set.

We will turn our attention to the study of the stopping and continuation
sets for $\tilde{v}$. \edt{In a sequence of 3 lemmas, we will establish
the regions of the state space which are subsets of the stopping or
continuation sets for $\tilde{v}$. Each lemma uses different mathematical
tools, which guided the split of the material. For the clarity of
the presentation, we outline here the steps: 
\begin{itemize}
\item $\{(x,a)\in\cI\times[0,\infty):x>\alpha(\theta)\}\subset\cC_{\tilde{v}}^{\theta}$
(Lemma \ref{lem:v_ge_0}); 
\item $\{(x,a)\in\cI\times[0,\infty):x<\alpha(Y(a)),\ a<A(\theta)\}\subset\cC_{\tilde{v}}^{\theta}$
(Lemma \ref{lem:C_v}); 
\item $\{(x,a)\in\cI\times[0,\infty):x\le\alpha(Y(a)),\ a\ge A(\theta)\}\subset\cS_{\tilde{v}}^{\theta}$
(Lemma \ref{lem:S_v}). 
\end{itemize}
Corollary \ref{cor:S_v} combines these properties into a complete
characterisation of the stopping set $S_{\tilde{v}}^{\theta}$. }

\begin{lemma}\label{lem:v_ge_0} $\tilde{v}(x,a;\theta)>0$ for $(x,a)\in\cI\times[0,\infty)$
such that $x>\alpha(\theta)$. \end{lemma} 
\begin{proof}
Fix $x>\alpha(\theta)$ and $a\ge0$. Consider a process 
\[
\begin{cases}
dA_{t}=\lambda(X_{t},Y(A_{t}))dt,\\
A_{0-}=A_{0}=a.
\end{cases}
\]
Define $F_{|Y(a)}(y)=\frac{F(y\wedge Y(a))}{F(Y(a))}$ which is the
cumulative distribution function $F$ conditioned on the outcome being
smaller than $Y(a)$. Let $\theta_{a}=F_{|Y(a)}^{-1}(F(\theta_{2}))\sim F_{|Y(a)}$,
where we recall that $\theta_{2}$ is the exit value (type) of player
$2$. We choose $\theta_{a}$ in this way so that we can integrate
it out as we did for $\theta_{i}$ using arguments as in Lemma \ref{lem:J_int}
with a different cumulative distribution function.

Let $\bar{A}_{t}=A_{t}-a$ and $\tau_{a}=\inf\{t\ge0:F_{|Y(a)}^{-1}(e^{-\bar{A}_{t}})<\theta_{a}\}$.
Taking $\sigma=\inf\{t\ge0:\ X_{t}\le\alpha(\theta)\}$, we obtain
\begin{align*}
u(x;\theta) & =\E_{x}\bigg[\int_{0}^{\sigma}e^{-rs}D(X_{s})ds+e^{-r\sigma}\theta\bigg]\\
 & =\tl\E_{x}\bigg[\int_{0}^{\sigma\wedge\tau_{a}}e^{-rs}D(X_{s})ds+\ind{\sigma\le\tau_{a}}e^{-r\sigma}\theta+\ind{\sigma>\tau_{a}}\Big(\int_{\sigma}^{\tau_{a}}e^{-rs}D(X_{s})ds+e^{-r\sigma}\theta\Big)\bigg]\\
 & \le\tl\E_{x}\bigg[\int_{0}^{\sigma\wedge\tau_{a}}e^{-rs}D(X_{s})ds+\ind{\sigma\le\tau_{a}}e^{-r\sigma}\theta+\ind{\sigma>\tau_{a}}\Big(\int_{\sigma}^{\tau_{a}}e^{-rs}M(X_{s})ds+e^{-r\sigma}m(X_{\sigma})\Big)\bigg]\\
 & =\tl\E_{x}\bigg[\int_{0}^{\sigma\wedge\tau_{a}}e^{-rs}D(X_{s})ds+\ind{\sigma\le\tau_{a}}e^{-r\sigma}\theta+\ind{\sigma>\tau_{a}}e^{-r\tau_{a}}m(X_{\tau_{a}})\bigg]=J(x,\sigma,\tau_{a};\theta),
\end{align*}
where in the inequality we used that $M(x)\ge D(x)$ and $m(x)>\theta_{U}$.
Notice also that we integrate over the extended probability space
$(\tl\Omega,\tl\cF,\tl\P)$ since $\tau_{a}$ depends on $\theta_{2}$
via $\theta_{a}$. We apply analogous arguments as in Lemma \ref{lem:J_int}
and rewrite the middle term as in \eqref{eqn:integral_middle} to
get 
\[
u(x;\theta)\le\E_{x}\bigg[\int_{0}^{\sigma}e^{-rs-\bar{A}_{s}}\big(D(X_{s})-r\theta+\lambda(X_{s},Y(A_{s}))(m(X_{s})-\theta)\big)ds\bigg]+\theta.
\]
Taking the supremum over $\sigma\in\cT(\cF_{t})$ and subtracting
$\theta$ from both sides, we see that 
\begin{align*}
\tilde{u}(x;\theta) & \le\sup_{\sigma\in\cT(\cF_{t})}\E_{x}\bigg[\int_{0}^{\sigma}e^{-rs-\bar{A}_{s}}\big(D(X_{s})-r\theta+\lambda(X_{s},Y(A_{s}))(m(X_{s})-\theta)\big)ds\bigg]\\
 & =e^{a}\sup_{\sigma\in\cT(\cF_{t})}\E_{x}\bigg[\int_{0}^{\sigma}e^{-rs-A_{s}}\big(D(X_{s})-r\theta+\lambda(X_{s},Y(A_{s}))(m(X_{s})-\theta)\big)ds\bigg]=e^{a}\tilde{v}(x,a;\theta),
\end{align*}
where in the first equality we used the relationship between $A_{t}$
and $\bar{A}_{t}$, and the last one is by the definition of $\tilde{v}$.
Since $x>\alpha(\theta)$, it is in the continuation region $\cC_{\tilde{u}}$
of $\tilde{u}$, so $\tilde{u}(x;\theta)>0$. Hence, $\tilde{v}(x,a;\theta)>0$. 
\end{proof}
\begin{lemma}\label{lem:C_v} $\tilde{v}(x,a;\theta)>0$ for $(x,a)\in\cI\times[0,\infty)$
such that $x<\alpha(Y(a))$ and $a<A(\theta)$. \end{lemma} 
\begin{proof}
For any $(x,a)\in\cI\times[0,\infty)$ satisfying $x<\alpha(Y(a))$
and $a<A(\theta)$, we have 
\begin{align*}
D(x)-r\theta+\lambda(x,Y(a))(m(x)-\theta) & =D(x)-r\theta+\frac{rY(a)-D(x)}{m(x)-Y(a)}(m(x)-\theta)\\
 & >D(x)-r\theta+\frac{r\theta-D(x)}{m(x)-\theta}(m(x)-\theta)=0,
\end{align*}
where the inequality uses $Y(a)>\theta$ and Lemma \ref{lem:lambda}.

The stopping time $\eta:=\inf\{t\ge0:\ X_{t}\ge\alpha(Y(A_{t}))\text{ and }A_{t}<A(\theta)\}$
is $\P_{xa}$-a.s. strictly positive by the continuity of $(X_{t},A_{t})$.
On the interval $[0,\eta)$ the above estimate applies. Hence 
\[
\tilde{v}(x,a;\theta)\ge\E_{xa}\bigg[\int_{0}^{\eta}e^{-rs-A_{s}}\big(D(X_{s})-r\theta+\lambda(X_{s},Y(A_{s}))(m(X_{s})-\theta)\big)ds\bigg]>0.
\]
\end{proof}
\begin{lemma}\label{lem:S_v} We have $\cS_{\tilde{v}}^{\theta}\supseteq\{(x,a)\in\cI\times[0,\infty):\ x\le\alpha(\theta)\text{ and }a\ge A(\theta)\}$.
\end{lemma} 
\begin{proof}
Consider an optimal stopping problem $\tilde{v}(x,a;\theta)$ on $(x,a)\in\cI\times[A(\theta),\infty)=:\cO$.
Then $\varphi(x,a)=e^{a}\tilde{v}(x,a;\theta)$ is the smallest non-negative
function satisfying for $(x,a)\in\cO$ the supermartingale property
(the justification of this fact is relegated to the end of the proof):
\begin{multline}
\E_{xa}\bigg[e^{-rt-\int_{0}^{t}\lambda(X_{u},Y(A_{u}))du}\varphi(X_{t},A_{t})\\
+\int_{0}^{t}e^{-rs-\int_{0}^{s}\lambda(X_{u},Y(A_{u}))du}\big(D(X_{s})-r\theta+\lambda(X_{s},Y(A_{s}))(m(X_{s})-\theta)\big)ds\bigg]\le\varphi(x,a).\label{eqn:supermart_v}
\end{multline}
We will show that \eqref{eqn:supermart_v} is satisfied by $\varphi(x,a)=\tilde{u}(x;\theta)$,
from which we immediately conclude that $\tilde{u}(x;\theta)\ge e^{a}\tilde{v}(x,a;\theta)$
for $(x,a)\in\cO$ as $\tilde{u}$ is non-negative. Since $\tilde{u}(x;\theta)=0$
for $x\le\alpha(\theta)$, the stopping region for $\tilde{v}$ must
contain $\cO\cap\big((x_L,\alpha(\theta)]\times[0,\infty)\big)$
which is the statement of the proposition.

Since $\tilde{u}(\cdot;\theta)$ is $C^{1}(\cI)$ and the second derivative
lies in $L_{loc}^{\infty}$ (see \eqref{eq:U*} and classical results
on the smoothness of the value function on the continuation set),
It\^{o}-Tanaka formula \cite[Ch.~VI, Thm.~1.5]{revuzyor} yields 
\[
e^{-rt}\tilde{u}(X_{t};\theta)=\tilde{u}(x;\theta)+\int_{0}^{t}e^{-rs}(\cL_{X}-r)\tilde{u}(X_{s};\theta)ds+\int_{0}^{t}e^{-rs}dM_{s},
\]
where $(M_{t})_{t\ge0}$ is a square integrable martingale and $\cL_{X}$
is the infinitesimal generator of $(X_{t})_{t\ge0}$. Using the supermartingale
property of the process $t\mapsto\int_{0}^{t}e^{-rs}(D(X_{s})-r\theta)ds+e^{-rt}\tilde{u}(X_{t};\theta)$,
we obtain 
\begin{equation}
\cL_{X}\tilde{u}-r\tilde{u}+D-r\theta\le0,\qquad\text{for \ensuremath{x\in\cI\setminus\{\alpha(\theta)\}}}.\label{eqn:cl_X_ineq}
\end{equation}
Let $\bar{A}_{t}=A_{t}-a$. We apply the product rule 
\[
e^{-rt-\bar{A}_{t}}\tilde{u}(X_{t};\theta)=\tilde{u}(x;\theta)+\int_{0}^{t}e^{-rs-\bar{A}_{s}}(\cL_{X}-r)\tilde{u}(X_{s};\theta)ds-\int_{0}^{t}e^{-rs-\bar{A}_{s}}\tilde{u}(X_{s};\theta)d\bar{A}_{s}+\int_{0}^{t}e^{-rs-\bar{A}_{s}}dM_{s}
\]
and take expectations on both sides to arrive at 
\begin{equation}
\E_{xa}\Big[e^{-rt-\bar{A}_{t}}\tilde{u}(X_{t};\theta)+\int_{0}^{t}e^{-rs-\bar{A}_{s}}\big(-(\cL_{X}-r)+\lambda(X_{s},Y(A_{s}))\big)\tilde{u}(X_{s};\theta)ds\Big]=\tilde{u}(x;\theta).\label{eqn:subma}
\end{equation}
Recall that $\tilde{u}(x';\theta)=0$ for $x'\le\alpha(\theta)$ and that
$\alpha(\theta)\ge\alpha(Y(A_{s}))$ since $A_{s}\ge A(\theta)$.
Hence 
\begin{align*}
 & \int_{0}^{t}e^{-rs-\bar{A}_{s}}\big(-(\cL_{X}-r)+\lambda(X_{s},Y(A_{s}))\big)\tilde{u}(X_{s};\theta)ds\\
 & =\int_{0}^{t}\ind{X_{s}>\alpha(Y(A_{s}))}e^{-rs-\bar{A}_{s}}\Big(-(\cL_{X}-r)\tilde{u}(X_{s};\theta)+\lambda(X_{s},Y(A_{s}))\tilde{u}(X_{s};\theta)\Big)ds\\
 & \ge\int_{0}^{t}\ind{X_{s}>\alpha(Y(A_{s}))}e^{-rs-\bar{A}_{s}}\Big(D(X_{s})-r\theta+\lambda(X_{s},Y(A_{s}))\tilde{u}(X_{s};\theta)\Big)ds\\
 & =\int_{0}^{t}\ind{X_{s}>\alpha(Y(A_{s}))}e^{-rs-\bar{A}_{s}}\Big(D(X_{s})-r\theta+\lambda(X_{s},Y(A_{s}))(m(X_{s})-\theta)\Big)ds,
\end{align*}
where the last equality uses that $\lambda(x',Y(a'))=0$ for $x'>\alpha(Y(a'))$.
On $x'\le\alpha(Y(a'))$ we have 
\[
D(x')-r\theta+\lambda(x',Y(a'))(m(x')-\theta)\le D(x')-r\theta+\lambda(x',\theta)(m(x')-\theta)=0,
\]
where the inequality is by Lemma \ref{lem:lambda} and $Y(a')\le\theta$.
Therefore, 
\[
0\ge\int_{0}^{t}\ind{X_{s}\le\alpha(Y(A_{s}))}e^{-rs-\bar{A}_{s}}\big(D(X_{s})-r\theta+\lambda(X_{s},Y(A_{s}))\tilde{u}(X_{s};\theta)\big)ds.
\]
Inserting the above two estimates into \eqref{eqn:subma} we get 
\[
\E_{xa}\bigg[e^{-rt-\bar{A}_{t}}\tilde{u}(X_{s};\theta)+\int_{0}^{t}e^{-rs-\bar{A}_{s}}\big(D(X_{s})-r\theta+\lambda(X_{s},Y(A_{s}))\tilde{u}(X_{s};\theta)\big)ds\bigg]\le\tilde{u}(x;\theta),
\]
which completes the proof that $\varphi(x,a)=\tilde{u}(x;\theta)$
satisfies \eqref{eqn:supermart_v}.

\textit{Derivation of \eqref{eqn:supermart_v}:} Define 
\[
F(x,a)=\E_{xa}\bigg[\int_{0}^{\infty}e^{-rs-A_{s}}\big(D(X_{s})-r\theta+\lambda(X_{s},Y(A_{s}))(m(X_{s})-\theta)\big)ds\bigg].
\]
Similar arguments as in the proof of Proposition \ref{prop:continuity}
show that $F$ is continuous and bounded. Furthermore, by the strong
Markov property of $(X_{t},Y_{t})$, we have 
\[
\tilde{v}(x,a;\theta)=\sup_{\sigma\in\cT(\cF_{t})}\E_{xa}\Big[F(x,a)-e^{-r\sigma}F(X_{\sigma},A_{\sigma})\Big].
\]
Recall that a function $\psi$ is called $r$-excessive if $\psi(x,a)\ge\E_{xa}\big[e^{-rt}\psi(X_{t},A_{t})\big]$, $t\ge0$.
Define $\eta(x,a)=\sup_{\sigma\in\cT(\cF_{t})}\E_{xa}\big[-e^{-r\sigma}F(X_{\sigma},A_{\sigma})\big]$. By \cite[Sec.~3.3, Thm.~1]{shiryaev2007optimal}, $\eta$ is the smallest
$r$-excessive function dominating $-F$. Using that $\tl v(x,a;\theta)=\eta(x,a)+F(x,a)$,
we obtain that $\tl v$ is the smallest non-negative function satisfying,
for $t\ge0$, 
\[
\tilde{v}(x,a;\theta)\ge\E_{xa}\Big[e^{-rt}\tl v(X_{t},A_{t};\theta)+\int_{0}^{t}e^{-rs-A_{s}}\big(D(X_{s})-r\theta+\lambda(X_{s},Y(A_{s}))(m(X_{s})-\theta)\big)ds\Big].
\]
Inserting $\tl v(x,a;\theta)=e^{-a}\varphi(x,a)$ above and noticing
that $A_{s}-a=\int_{0}^{s}\lambda(X_{u},Y(A_{u}))du$ yields \eqref{eqn:supermart_v}. 
\end{proof}
Having established properties of the continuation and stopping sets
for $\tilde{v}$ and a given $\theta$, we combine them into a complete
characterisation of the stopping set.

\begin{corollary}\label{cor:S_v} Stopping region for $\tilde{v}$
(and $v$) is 
\[
\cS_{\tilde{v}}^{\theta}=\{(x,a)\in\cI\times[0,\infty):\ x\le\alpha(\theta)\text{ and }a\ge A(\theta)\}.
\]
\end{corollary} 
\begin{proof}
Denote by $\tilde{\cS}^{\theta}$ the right-hand side of the equality
in the statement of the corollary. From Lemma \ref{lem:S_v}, we have
$\tilde{\cS}^{\theta}\subset\cS_{\tilde{v}}^{\theta}$. The proof
is completed when we show that $(\tilde{\cS}^{\theta})^{c}\subset\cC_{\tilde{v}}$.
Lemma \ref{lem:v_ge_0} implies that 
\begin{equation}
\{(x,a)\in\cI\times[0,\infty):\ x>\alpha(\theta)\}\subset\cC_{\tilde{v}}^{\theta}.\label{eqn:inclusion1}
\end{equation}
Lemma \ref{lem:C_v} shows that 
\begin{equation}
\{(x,a)\in\cI\times[0,\infty):\ x<\alpha(Y(a))\text{ and }a<A(\theta)\}\subset\cC_{\tilde{v}}^{\theta}.\label{eqn:inclusion2}
\end{equation}
Take $a<A(\theta)$. From the first inclusion, $(\alpha(\theta),x_{U})\times\{a\}\subset\cC_{\tilde{v}}^{\theta}$,
where we recall that $\cI=(x_{L},x_{U})$. From the second inclusion,
$(x_{L},\alpha(Y(a)))\times\{a\}\subset\cC_{\tilde{v}}^{\theta}$,
but $a<A(\theta)$ means that $Y(a)>\theta$, so $\alpha(Y(a))>\alpha(\theta)$.
Hence $\cI\times\{a\}\subset\cC_{\tilde{v}}^{\theta}$. This, together
with inclusions \eqref{eqn:inclusion1} and \eqref{eqn:inclusion2},
implies 
\[
\cC_{\tilde{v}}^{\theta}\supseteq\{(x,a)\in\cI\times[0,\infty):\ x>\alpha(\theta)\text{ or }a<A(\theta)\}=(\tilde{\cS}^{\theta})^{c}
\]
and the proof is completed. 
\end{proof}

\subsection{Nash equilibrium}\label{subsec:Nash}

\edt{The description of the stopping set in Corollary \ref{cor:S_v}
is natural in the framework of optimal stopping of two-dimensional
dynamics. The resulting optimal stopping time, however, is not of
the form \eqref{eqn:tau_form}. It turns out that due to the specific
form of the stopping region and of the dynamics of $(X_{t},A_{t})$
this optimal stopping time can be equivalently described as the first
time that $A_{t}$ exceeds $A(\theta)$, hence, in the form of \eqref{eqn:tau_form}
when one recalls that $A_{t}=A(Y_{t})$. Leaving technical complications
aside, this can be seen as follows. Denoting by $\sigma$ the first
entry time of $(X_{t},A_{t})$ to $\cS_{\tilde{v}}^{\theta}$ and
by $\tau$ the first time that $A_{t}>A(\theta)$, it is clear that
$\tau\ge\sigma$. Obviously, if $A_{\sigma}>A(\theta)$ then $\tau=\sigma$.
Assume $A_{\sigma}=A(\theta)$. When $X_{\sigma}<\alpha(\theta)$,
the process $A_{t}$ is strictly increasing at $\sigma$ as $\lambda(X_{\sigma},\theta)>0$,
so again $\tau=\sigma$. A more delicate argument is needed when $X_{\sigma}=\alpha(\theta)$,
but then the regularity of the point $\alpha(\theta)$ for the process
$(X_{t})$ implies that $(X_{t})$ enters the open interval $(x_{L},\alpha(\theta))$
immediately, so a similar argument as before can be used.}

\begin{lemma}\label{lem:S_v_rewritten} For $(x,a)\in\cI\times[0,A(\theta))$,
we have 
\[
\inf\{t\ge0:\ A_{t}>A(\theta)\}=\inf\{t\ge0:\ (X_{t},A_{t})\in\cS_{\tilde{v}}^{\theta}\},\quad\P_{xa}-a.s.
\]
\end{lemma} 
\begin{proof}
Define $\tau=\inf\{t\ge0:\ A_{t}>A(\theta)\}$ and 
\[
\sigma=\inf\{t\ge0:\ (X_{t},A_{t})\in\cS_{\tilde{v}}^{\theta}\}=\inf\{t\ge0:\ A_{t}\ge A(\theta)\text{ and }X_{t}\le\alpha(\theta)\}.
\]
Recall that $\cS_{\tilde{v}}^{\theta}$ is closed, so $\sigma$ is
a stopping time. Fix $(x,a)$ as in the statement of the lemma. We
will argue omega by omega. Fix $\omega\in\Omega$ and take any $t\ge0$
such that $A_{t}(\omega)>A(\theta)$. By the assumption that $a<A(\theta)$
we have $t>0$. Due to the dynamics of $(A_{t})$ there is $s\le t$
such that $X_{s}(\omega)\le\alpha(\theta)$ and $A_{s}(\omega)>A(\theta)$.
This implies that $s\in\{u\ge0:\ A_{u}(\omega)\ge A(\theta)\text{ and }X_{u}(\omega)\le\alpha(\theta)\}$
and shows that $\tau(\omega)\ge\sigma(\omega)$. From the arbitrariness
of $\omega$ we conclude that $\tau\ge\sigma$.

Since $\tau\ge\sigma$, $\P_{xa}$-a.s., by the strong Markov property
of $(X_{t},A_{t})$ and the continuity of the trajectories we have
$\P_{xa}(\tau>\sigma)=\E_{xa}\big[\P_{X_{\sigma}A_{\sigma}}(\tau>0)\big]$
and $X_{\sigma}\le\alpha(\theta)$, $A_{\sigma}=A(\theta)$. In order
to show that $\P_{xa}(\tau>\sigma)=0$ it suffices to demonstrate that
\begin{equation}
\P_{xa}(\tau>0)=0\quad\text{for}\quad(x,a)\in\cO_{\theta}:=\{(x,a)\in\cI\times[0,\infty):x\le\alpha(\theta),a=A(\theta)\}.\label{eqn:tau_zero}
\end{equation}
Take $(x,a)\in\cO_{\theta}$. If $x<\alpha(\theta)$ then $\eta=\inf\{t\ge0:\ X_{t}>\alpha(Y(A_{t}))\}>0$
$\P_{xa}$-a.s. and $(A_{t})$ is strictly increasing on $[0,\eta)$.
Hence, $\P_{xa}(\tau>0)=0$. Consider now $x=\alpha(\theta)$ and
$a=A(\theta)$. By the non-degeneracy of the diffusion around $x$,
we have that $\eta^{\circ}=\inf\{t\ge0:X_{t}<\alpha(\theta)\}=0$
$\P_{xa}$-a.s. Let $B=\{\tau>0\text{ and }\eta^{\circ}=0\}$. For
$\omega\in B$, $A_{t}(\omega)=A_{0}(\omega)=A(\theta)$ for $t\in[0,\tau(\omega))$.
Since $\eta^{\circ}(\omega)=0$, there is $t(\omega)<\tau(\omega)$
such that $X_{t(\omega)}<\alpha(\theta)$. By the continuity of $(X_{t})$,
the Lebesgue measure of the set $\{s\in[0,t(\omega)]:X_{s}(\omega)<\alpha(\theta)\}$
is greater than zero which contradicts that $\lambda(X_{s}(\omega),A_{s}(\omega))=0$
for $s\in[0,\tau(\omega))$ (recall the dynamics \eqref{eqn:ODE_A}
of $(A_{t})$). This contradiction shows that the set $B$ is empty.
Since $\P_{xa}(\eta^{\circ}=0)=1$, we conclude that $\P_{xa}(\tau>0)=0$. 
\end{proof}
We are now in a position to state the main result of this section.
\begin{theorem}\label{thm:Nash} Let $\tau_{1}$, $\tau_{2}$ be
given by \eqref{eqn:tau_form} with $(A_{t})$ stated in \eqref{eqn:Y_def}.
Then $(\tau_{1},\tau_{2})$ is a Nash equilibrium in the sense of
Definition \ref{def:Nash}. \end{theorem} 
\begin{proof}
We apply Corollary \ref{corr:equil}. Take $\hat{\tau}_{1}(\cdot,\theta)=\hat{\tau}_{2}(\cdot,\theta)=\inf\{t\ge0:\ Y(A_{t})<\theta\}$.
These are stopping times thanks to Lemma \ref{lem:S_v_rewritten}.
Due to the symmetry of the problem, it is sufficient to show that
for every $\theta\in[\theta_{L},\theta_{U}]$ and $x\in\cI$, the
stopping time $\sigma^{*}=\hat{\tau}_{1}(\cdot,\theta)$ solves the
optimal stopping problem $\sup_{\sigma\in\cT(\cF_{t})}J(x,\sigma,\tau_{2};\theta)$,
where $\tau_{2}=\hat{\tau}_{2}(\cdot,\theta_{2})$ is an $\cT(\cF_{t}^{2})$-stopping
time. Recall that by Lemma \ref{lem:J_int} and \eqref{eqn:def_v}
we have 
\[
v(x,0;\theta)=\sup_{\sigma\in\cT(\cF_{t})}J(x,\sigma,\tau_{2};\theta).
\]
We show in Corollary \ref{cor:S_v} that the optimal stopping time
for $v$ is given by 
\begin{align*}
\sigma^{*} & =\inf\{t\ge0:\ X_{t}\le\alpha(\theta)\text{ and }A_{t}\ge A(\theta)\}\\
 & =\inf\{t\ge0:\ A_{t}>A(\theta)\}\\
 & =\inf\{t\ge0:\ Y(A_{t})<\theta\},
\end{align*}
where the second equality follows from Lemma \ref{lem:S_v_rewritten}
and the last equality is because $Y=A^{-1}$. 
\end{proof}

%%%%%%%%%%%%%%%%%%%%%%%%

\subsection{Perfect Bayesian Equilibrium}\label{subsec:PBE}

\edt{
The equilibrium established above is not only a Nash equilibrium, but it also furnishes a {\em perfect Bayesian
equilibrium} via \eqref{eqn:tau_PBE}. A perfect Bayesian equilibrium requires two conditions
\cite{Osborne1994}: sequential rationality and the correct dynamics
(Bayesian updating) of the beliefs of the players. The correct belief
dynamics is automatically taken care of through the dynamics of $Y_{t}$
that unambiguously determines the posterior probability distribution
of $\theta_{1}$ and $\theta_{2}$. The sequential rationality stipulates
that each player's strategy is a best response at any time $t$ with
the knowledge of his own type at any value of $X_{t}$ and any belief
represented by $Y_{t}$. In particular, a perfect Bayesian
equilibrium has to take into account the best response even in case
of deviations. In the context of the strategy profile we defined,
a deviation happens only if a player of type $\theta$ fails to exit
even though $Y_{t}<\theta$. 
%From the form of $\lambda(\cdot,\cdot)$, a player of type $\theta$ exits only if $X_{t}\le \alpha(Y_{t})$, so at the time $Y_{t}$ decreases to $\theta$, $X_{t}\le\alpha(Y_{t})=\alpha(\theta)$ is automatically satisfied. This means that the second inequality $X_{t}\le\alpha(\theta)$ of \eqref{eqn:tau_PBE} is automatically satisfied on an equilibrium path. It follows that a player failing to exit when $Y_{t}<\theta$ is a deviation from the strategy profile; 
If deviation takes place, the dynamics of $(Y_{t})$ is unaltered because
his opponent can never find out if deviation occurred or not.

To formally define a perfect Bayesian equilibrium in our setting with asymmetric information, we adapt the notions of extended strategy and time-consistent extended strategy from Riedel and Steg \cite{Riedel2017}. For the sake of simplicity, we present the framework for equilibria in pure strategies which can be written as hitting times of a measurable set $\Upsilon$ by the process $(X_t, A_t, \theta_i)$, where the last component is the type of the player; the strategy \eqref{eqn:tau_PBE} is of this form. The stopping time employed by the player $i$ of type $\theta_i$, $i=1,2$, for the game starting at time $t \ge 0$ is given by
\[
\tau(t; \theta_i) = \inf\{s \ge t: (X_s, A_s, \theta_i) \in \Upsilon \}.
\]
Notice that such a family of strategies for player $i$ and a fixed value of $\theta_i$ satisfies the conditions of time-consistent extended strategy in \cite[Def.~2.13]{Riedel2017}, i.e., it is an extension of that concept to games with private information.

To this end, for any $t\ge0$,
we define the posterior distribution $F_{|Y(A_{t})}$ of $\theta_{i}$,
$i=1,2$, given the value of the belief process $A_{t}$, as $F_{|y}(z)=F(z\wedge y)/F(y)$.
Therefore, the random variables denoting the remaining types of player
$i$ conditional that the game has not ended by time $t$, are constructed
as 
\begin{equation}
\hat{\theta}_{i}^{t}=F_{|Y(A_{t})}^{-1}(F(\theta_{i}))\sim F_{|Y(A_{t})},\qquad i=1,2.\label{eq:posterior_PBE}
\end{equation}
Define a functional $\hat{J}_{t}$ by 
\begin{equation}
\hat{J}_{t}(x,\tau_{1},\tau_{2};\theta)=\tl\E_{x}\bigg[\int_{t}^{\tau_{1}\wedge\tau_{2}}e^{-r(s-t)}D(X_{s})ds+\ind{\tau_{1}\le\tau_{2}}e^{-r(\tau_{1}-t)}\theta+\ind{\tau_{1}>\tau_{2}}e^{-r(\tau_{2}-t)}m(X_{\tau_{2}})\bigg|\tilde{\cF}_{t}\bigg].\label{eqn:PDE_func_mod}
\end{equation}

\begin{definition} \label{def:PBE2} The symmetric strategy profile $\tau(t;\theta)$ given by $\Upsilon$ and the posterior probability distribution (\ref{eq:posterior_PBE}) constitute a perfect Bayesian equilibrium if 
\begin{align*}
\hat{J}_{t}(x,\sigma,\tau(t;\hat{\theta}_{i}^{t});\theta) & \le\hat{J}_{t}(x,\tau(t;\theta),\tau(t;\hat{\theta}_{i}^{t});\theta), \quad \text{$\P_x$-{a.s.}},
%\hat{J}_{t}(x,\sigma,\tau(t;\hat{\theta}_{2}^{t});\theta) & \le\hat{J}_{t}(x,\tau(t;\theta),\tau(t;\hat{\theta}_{2}^{t});\theta), 
\end{align*}
for any $i = 1,2$, $x \in \cI$, $\theta \in [\theta_L, \theta_U]$, $t \ge 0$, and $\sigma \in \cT(\tilde \cF_t)$, $\sigma \ge t$.
\end{definition}

A formal treatment of the conditional random distribution $F_{|A(Y_t)}$ is beyond the scope of this paper but the functional $\hat J_t$ evaluated at $\tau(t, \hat\theta^t_i)$ can be given the formal meaning:
\begin{multline*}
\hat J_{t}(x,\sigma,\tau(t;\hat\theta^t_{i});\theta)=\E_{x}\bigg[\int_{\theta_L}^{Y(A_t)} \bigg(\int_{t}^{\sigma \wedge \tau(t;\gamma)}e^{-r(s-t)}D(X_{s})ds+\ind{\sigma\le\tau(t;\gamma)}e^{-r(\sigma-t)}\theta\\
+\ind{\sigma>\tau(t;\gamma)}e^{-r(\tau(t;\gamma)-t)}m(X_{\tau(t;\gamma)}) \bigg)\frac{dF(\gamma)}{F(Y(A_t))} \bigg|{\cF}_{t}\bigg].
\end{multline*}
For technical reasons and mathematical convenience, we define
\begin{equation}\label{eqn:PDE_func}
J_t(x, \tau_1, \tau_2; \theta) = \tl\E_{x} \bigg[\ind{\tau_2 > t} \bigg(\int_{t}^{\tau_{1}\wedge\tau_{2}}e^{-r(s-t)}D(X_{s})ds+\ind{\tau_{1}\le\tau_{2}}e^{-r(\tau_{1}-t)}\theta+\ind{\tau_{1}>\tau_{2}}e^{-r(\tau_{2}-t)}m(X_{\tau_{2}})\bigg) \bigg| \tilde \cF_t \bigg]
\end{equation}
and notice that 
\begin{equation}\label{eqn:J_t_equiv}
J_t(x, \sigma, \tau(0; \theta_i); \theta) = e^{-A_t} \hat J_t (x,\sigma,\tau(t;\hat\theta^t_{i});\theta),
\end{equation}
where we used $F(Y(A_t)) = e^{-A_t}$. The symmetry of the condition in Definition \ref{def:PBE2}, the identical distribution of $\theta_i$, and equality \eqref{eqn:J_t_equiv} imply that $\tau(t; \theta)$ furnishes a perfect Bayesian equilibrium if 
\begin{equation}\label{eqn:J_t_cond}
J_t(x, \sigma, \tau(0; \theta_2); \theta) \le J_t(x, \tau(t;\theta), \tau(0; \theta_2); \theta), \qquad \text{$\P_x$-a.s.,}
\end{equation}
for any $x \in \cI$, $\theta \in [\theta_L, \theta_U]$, $t \ge 0$ and $\sigma \in \cT(\tilde \cF_t)$ with $\sigma \ge t$.

Consider now $\Upsilon$ corresponding to the Nash equilibrium \eqref{eqn:tau_PBE}:
\begin{equation}\label{eqn:PBE_Upsilon}
\Upsilon^* = \{ (x,a, \theta) \in \cI \times [0, \infty) \times [\theta_L, \theta_U]:\ x \le \alpha(\theta), a \ge A(\theta) \}.
\end{equation}
In order to verify \eqref{eqn:J_t_cond}, we apply arguments similar as in the proof of Lemma \ref{lem:J_int} to integrate out $\theta_2$ in $J_t$:
\[
J_t(x, \sigma, \tau(0; \theta_2); \theta)
=
\E_x \Big[\int_{t}^{\sigma}e^{-r(s-t)-A_{s}}D(X_{s})ds+\theta e^{-r(\tau-t)-A_{\tau}}+\int_{[t,\tau)}e^{-r(s-t)-A_{s}}m(X_{s})dA_{s}\Big| \cF_t \Big],
\]
where we used that $\sigma \ge t$ and $\sigma \in \cT(\cF_t)$. To show that $\Upsilon^*$ defines a perfect Bayesian equilibrium, it is enough to prove that
for any $t \ge 0$ and $\theta \in [\theta_L, \theta_U]$, the stopping time $\sigma^* = \tau(t; \theta)$ solves the optimal stopping problem
\begin{equation}\label{eqn:PBE_OS}
\esssup_{\sigma \ge t} \E_x \Big[\int_{t}^{\sigma}e^{-r(s-t)-A_{s}}D(X_{s})ds+\theta e^{-r(\sigma-t)-A_{\sigma}}+\int_{[t,\sigma)}e^{-r(s-t)-A_{s}}m(X_{s})dA_{s}\Big| \cF_t \Big].
\end{equation}
Due to the Markov property of $(X_t, A_t)$ and the boundedness and continuity of $D$ and $m$, the classical theory of optimal stopping yields the optimal stopping time of the form $\sigma^* = \inf \{ s \ge t: U(X_s, A_s; \theta) = \theta e^{-A_s} \}$, where
\[
U(x,a;\theta) = \sup_{\sigma \in \cT(\cF_t)} \E_{x a} \Big[\int_{0}^{\sigma}e^{-rs-A_{s}}D(X_{s})ds+\theta e^{-r\sigma-A_{\sigma}}+\int_{[0,\sigma)}e^{-rs-A_{s}}m(X_{s})dA_{s}\Big].
\]
Recalling the form of $A_t$ in \eqref{eqn:ODE_A}, notice that $U(x,a; \theta) = v(x, a; \theta)$, where $v$ is defined in \eqref{eqn:def_v}. Corollary \ref{cor:S_v} implies that the solution of the stopping problem \eqref{eqn:PBE_OS} is indeed given by $\sigma^* = \tau(t; \theta)$. This completes the proof that the strategy profile derived in previous sections gives rise to a perfect Bayesian equilibrium.
\begin{theorem}\label{thm:PBE}
Symmetric strategy profile given by $\Upsilon^*$ is a perfect Bayesian equilibrium in the sense of Definition \ref{def:PBE2}.
\end{theorem}

We emphasise again the difference between the definition of the strategy profile \eqref{eqn:tau_form} that forms a Nash equilibrium and the strategy profile \eqref{eqn:tau_PBE} that corresponds to $\Upsilon^*$ and forms the perfect Bayesian equilibrium. These strategies coincide along the equilibrium path for the game started at time $0$. The former definition is fundamental for the reformulation of the best response problem where the type of the opponent is integrated out (Lemma \ref{lem:J_int}). We further exploit it above where we rewrite the functional $\hat J_t$ as $J_t$ in \eqref{eqn:J_t_equiv} with the opponent following the equilibrium path $\tau(0; \theta_i)$. The reader can further notice the interaction between these definitions in the following remark.

\begin{remark}
The equality \eqref{eqn:J_t_equiv} enables a mathematically equivalent formulation of Definition \ref{def:PBE2}: a symmetric strategy profile $\tau(t; \theta)$ given by $\Upsilon$ is a perfect Bayesian equilibrium if
\[
\{ \tau(0; \theta_i) < t \} = \{ Y(A_t) < \theta_i\}, \quad \text{$\P_x$-{a.s.}}
\]
and
\[
J_t(x, \sigma, \tau(0; \theta_i); \theta) \le J_t(x, \tau(t;\theta), \tau(0; \theta_i); \theta),\quad \text{$\P_x$-{a.s.},}
\]
for any $i = 1,2,$ $x \in \cI$, $\theta \in [\theta_L, \theta_U]$, $t \ge 0$,  and $\sigma \in \cT(\tilde \cF_t)$, $\sigma \ge t$.
\end{remark} 

}

%%%%%%%%%%%%%%%%%%%%%%%%

\subsection{Remaining proofs} \label{subsec:proofs} 

\begin{proof}[Proof of Lemma \ref{lem:J_int}]
Since $(A_{t})_{t\ge0}$ is $(\cF_{t})$-adapted, right-continuous
and increasing, the process $(Y_{t})_{t\ge0}$ retains the same adaptivity
and right-continuity but is decreasing. From $A_{0-}=0$ and $F(\theta_{U})=1$,
we obtain $Y_{0-}=\theta_{U}$. We also have $Y_{t}\in(\theta_{L},\theta_{U}]$,
$t\ge0$.

The rest of the proof follows similar arguments as in \cite[Section 4]{DAMP2021}.
For simplicity of notation, we omit the index $i$ in $\tau_{i}$
and $\theta_{i}$. Here, we treat $(Y_{t})_{t\ge0}$ and $(A_{t})_{t\ge0}$
as stochastic processes on $(\tilde{\Omega},\tilde{\cF},\tilde{\P})$
in an obvious way due to the product form of the probability space.
Let $\tilde{\cF}_{\infty}=\bigvee_{t\ge0}\tilde{\cF}_{t}$. We first
note that 
\[
\{t>\tau\}\subseteq\{Y_{t}<\theta\}\subseteq\{t\ge\tau\}.
\]
From the first inclusion, we get that 
\[
\tl\E\big[\ind{\sigma>\tau}\big|\tilde{\cF}_{\infty}\big]=\lim_{\ve\downarrow0}\tl\E\big[\ind{\sigma-\ve>\tau}\big|\tilde{\cF}_{\infty}\big]\le\lim_{\ve\downarrow0}\big(1-F(Y_{\sigma-\ve})\big)=1-F(Y_{\sigma-}),
\]
where we used that $(Y_{t})$ is monotone, so the limits exist. The
second inclusion gives the opposite estimate: 
\[
\tl\E\big[\ind{\sigma>\tau}\big|\tilde{\cF}_{\infty}\big]=\lim_{\ve\downarrow0}\tl\E\big[\ind{\sigma-\ve\ge\tau}\big|\tilde{\cF}_{\infty}\big]\ge\lim_{\ve\downarrow0}\big(1-F(Y_{\sigma-\ve})\big)=1-F(Y_{\sigma-}).
\]
By \eqref{eqn:notation_A}, we conclude that 
\begin{equation}
\tl\E\big[\ind{\sigma>\tau}\big|\tilde{\cF}_{\infty}\big]=1-e^{-A_{\sigma-}}\qquad\text{and}\qquad\tl\E\big[\ind{\sigma\le\tau}\big|\tilde{\cF}_{\infty}\big]=e^{-A_{\sigma-}}.\label{eqn:rep_ind}
\end{equation}

Let $\varphi:[0,\infty]\times\R\to\R$ be a measurable bounded function.
We shall show that 
\begin{equation}
\tl\E_{x}\big[\ind{\sigma>\tau}\,\varphi(\tau,X_{\tau})\big|\tilde{\cF}_{\infty}\big]=\int_{[0,\sigma)}e^{-A_{s}}\varphi(s,X_{s})dA_{s}.\label{eqn:rep_phi}
\end{equation}
Define $\hat{\tau}(u)=\inf\{t\ge0:\ Y_{t}<u\}$ so that $\tau=\hat{\tau}(\theta)$;
notice that this is the decomposition of $\tau$ from Proposition
\ref{prop:structure} in which we suppress in the notation the dependence
on $\omega$. To shorten notation, define an $(\tilde{\cF}_{t})$-adapted
process 
\[
Z_{t}=\varphi(t,X_{t})\ind{\sigma>t},\quad t\ge0.
\]
Using the independence of $\theta$ from $\tilde{\cF}_{\infty}$,
we have 
\[
\tl\E_{x}\big[\ind{\sigma>\tau}\,\varphi(\tau,X_{\tau})\big|\tilde{\cF}_{\infty}\big]=\tl\E_{x}\big[Z_{\tau}\big|\tilde{\cF}_{\infty}\big]=\int_{\theta_{L}}^{\theta_{U}}Z_{\hat{\tau}(v)}dF(v)=\int_{0}^{1}Z_{\hat{\tau}(F^{-1}(u))}du,
\]
where $F^{-1}(\cdot)$ is the inverse of $F$ (which exists by Assumption
\ref{ass:theta_cdf}) and in the last equality we used \cite[Ch.~0, Prop.~4.9]{revuzyor}.
We rewrite $\hat{\tau}(F^{-1}(u))$ as follows: 
\begin{align*}
\hat{\tau}(F^{-1}(u)) & =\inf\{t\ge0:\ Y_{t}<F^{-1}(u)\}=\inf\{t\ge0:\ F(Y_{t})<u\}=\inf\{t\ge0:\ e^{-A_{t}}<u\}\\
 & =\inf\{t\ge0:\ 1-e^{-A_{t}}>1-u\}=:\tilde{\tau}(1-u).
\end{align*}
This allows us to write 
\[
\int_{0}^{1}Z_{\hat{\tau}(F^{-1}(u))}du=\int_{0}^{1}Z_{\tilde{\tau}(1-u)}du=\int_{0}^{1}Z_{\tilde{\tau}(u)}du=\int_{0}^{\infty}Z_{s}d(1-e^{-A_{s}})=\int_{0}^{\infty}e^{-A_{s}}Z_{s}dA_{s},
\]
where we apply \cite[Ch.~0, Prop.~4.9]{revuzyor} in the third equality.
Recalling the definition of $(Z_{t})_{t\ge0}$ completes the derivation
of \eqref{eqn:rep_phi}.

Using \eqref{eqn:rep_ind}-\eqref{eqn:rep_phi}, the functional $J$
takes an equivalent form 
\[
J(x,\sigma,\tau;\gamma)=\tl\E_{x}\Big[\int_{0}^{\sigma}e^{-rs-A_{s-}}D(X_{s})ds+\gamma e^{-r\sigma-A_{\sigma-}}+\int_{[0,\sigma)}e^{-rs-A_{s}}m(X_{s})dA_{s}\Big].
\]
Since $(A_{t})$ is increasing, it has only a countable number of
jumps, so 
\[
\int_{0}^{\sigma}e^{-rs-A_{s-}}D(X_{s})ds=\int_{0}^{\sigma}e^{-rs-A_{s}}D(X_{s})ds
\]
and \eqref{eqn:J_int} is proved. 
\end{proof}
\begin{proof}[Proof of Lemma \ref{lem:lambda}]
Take any $x\le\alpha(\theta_{U})$. The formula \eqref{eqn:derivative_lambda}
for the derivative of the mapping from the statement of the lemma
follows by straightforward differentiation. Denote it by $g(y;x)$.
By assumptions, we have $m(x)\ge m_{\min}>\theta_{U}$ and $D(x)\le r\theta_{U}$ since $x \le \alpha(\theta_U)$.
Hence $g(y;x)>0$. We also note that 
\[
\frac{rm(x)-D(x)}{(m(x)-y)^{2}}\le\frac{rm(x)}{(m(x)-\theta_{U})^{2}}\le\frac{r\,m_{\min}}{(m_{\min}-\theta_{U})^{2}},
\]
since the function $z\mapsto z/(z-\theta_{U})^{2}$ is decreasing
for $z>\theta_{U}$. 
\end{proof}
\begin{proof}[Proof of Lemma \ref{lem:monotonicity_lambda}]
Recalling that $a\mapsto Y(a)$ is decreasing, Lemma \ref{lem:lambda}
shows that $a\mapsto\lambda(x,Y(a))$ is decreasing on $\{a\ge0:\ x\le\alpha(Y(a))\}$
which is either a closed interval $[0,a^{*}(x)]$ or an empty set
(in which case we set $a^{*}(x)=0$), since $a\mapsto\alpha(Y(a))$
is continuous and decreasing. As $\lambda\ge0$ and $\lambda(x,\cdot)\equiv0$
on $(a^{*}(x),\infty)$, we conclude that $a\mapsto\lambda(x,a)$
is decreasing.

Fix now $a\ge0$ and notice that 
\[
x\mapsto\frac{rY(a)-D(x)}{m(x)-Y(a)}
\]
is decreasing on $x\le\alpha(Y(a))$. Indeed, $rY(a)-D(x)>0$ for
such $x$ and decreasing and the numerator is increasing in $x$.
We also have $\lambda(x,a)=0$ for $x>\alpha(Y(a))$ and $\lambda\ge0$,
so a potential jump at $x=\alpha(Y(a))$ is downward. This completes
the proof of monotonicity. 
\end{proof}
\begin{proof}[Proof of Lemma \ref{lem:equal_null}]
Define a semimartingale $Z_{t}=X_{t}-\varphi_{t}$. The occupation
times formula \cite[Cor.~1, p. 216]{Protter} shows that for any $\ve>0$
and $t>0$, $\P_{x}$-a.s., 
\begin{equation}
\int_{-\ve}^{\ve}L_{t}^{u}du=\int_{0}^{t}\ind{Z_{s-}\in[-\ve,\ve]}d[Z,Z]_{s}^{c}=\int_{0}^{t}\indd{Z_{s-}\in[-\ve,\ve]}b^{2}(X_{s})ds,\label{eqn:local_time}
\end{equation}
where $L_{t}^{u}$ is the local time of $(Z_{t})_{t\ge0}$ at the
level $u$ and $[Z,Z]^{c}$ is the path-by-path continuous part of
the quadratic variation $[Z,Z]$ and equals to the quadratic variation
$[Z^{c},Z^{c}]$ of the continuous local martingale part of $Z$ (see
\cite[p. 70]{Protter}). Clearly, the continuous local martingale
part $Z^{c}$ of $Z$ equals to the continuous local martingale part
of $X^{c}$ and, using \cite[Thm.~29, p.~75]{Protter}, we have $[Z^{c},Z^{c}]_{t}=[X^{c},X^{c}]_{t}=\int_{0}^{t}b^{2}(X_{s})ds$.
The equality \eqref{eqn:local_time} holds outside of $\P_{x}$-negligible
set common for every $t\ge0$; indeed, it is sufficient to apply the
above formula for a sequence $t_{n}\to\infty$.

Taking the limit in \eqref{eqn:local_time} as $\ve\downarrow0$,
the dominated convergence theorem implies that
\edt{
\[
0 = \lim_{\ve \downarrow0} \int_{-\ve}^{\ve}L_{t}^{u}du = \lim_{\ve \downarrow 0} \int_{0}^{t}\indd{(X_{s-} - \varphi_{s-}) \in[-\ve,\ve]}b^{2}(X_{s})ds = \int_{0}^{t}\ind{X_{s}=\varphi_{s}}b^{2}(X_{s})ds,
\]
where in the last equality we used the continuity of $(X_t)$.
}
To conclude, we recall that $b(\cdot)>0$. 
\end{proof}
\begin{proof}[Proof of Proposition \ref{prop:solution_A}]
\textbf{Construction of $(A_{t})$:} Function $\alpha$ is strictly
increasing and continuous (Lemma \ref{lemm:conti-alpha}), so its
inverse $\alpha^{-1}$ is well defined, continuous and strictly increasing.
Hence, $x\le\alpha(y)$ can be equivalently written as $y\ge\alpha^{-1}(x)$.
Define 
\begin{equation}
\lambda^{\ve}(x,y)=l(x,y)\ind{y\ge\alpha^{-1}(x)}+l(x,\alpha^{-1}(x))\ind{y<\alpha^{-1}(x)}\frac{1}{\ve}\big(y-\alpha^{-1}(x)+\ve\big)^{+},\label{eqn:lambda_ve_def}
\end{equation}
where $l(x,y)$ is defined in Lemma \ref{lem:lambda}. Using this
lemma and the above definition, the mapping $y\mapsto\lambda^{\ve}(x,y)$
is Lipschitz with the constant independent of $x$. Due to the continuity
of $\alpha$ and its inverse, $\lambda^{\ve}$ is continuous. Hence,
for any $\omega\in\Omega$, using the continuity of trajectories of
$(X_{t})$, there is a unique solution of the ODE 
\begin{equation}
dA_{t}^{\ve}=\lambda^{\ve}(X_{t},Y(A_{t}^{\ve}))dt,\qquad A_{0}^{\ve}=a\ge0,\label{eqn:ODE_A_ve}
\end{equation}
and it depends continuously on $a$. Since $\lambda^{\ve}$ is increasing
in $\ve$ and non-negative, by the comparison principle for ODEs,
the solution $A_{t}^{\ve}$ is increasing in $\ve$ and non-negative.
Hence the limit $A_{t}^{0}:=\lim_{\ve\downarrow0}A_{t}^{\ve}$ exists,
is increasing and right-continuous. We will show that it satisfies
\eqref{eqn:int_A} $\P_{x}$-a.s. It suffices to show that 
\begin{equation}
\lim_{\ve\downarrow0}\int_{0}^{t}\lambda^{\ve}(X_{s},Y(A_{s}^{\ve}))ds=\int_{0}^{t}\lambda(X_{s},Y(A_{s}^{0}))ds,\quad P_{x}-a.s.\label{eqn:ve_limit}
\end{equation}
From \eqref{eqn:lambda_ve_def}, we have the lower bound $\lambda^{\ve}(x,y)\ge l(x,y)\ind{y>\alpha^{-1}(x)}$.
As $l$ is non-negative, Fatou's lemma implies 
\[
\liminf_{\ve\downarrow0}\int_{0}^{t}\lambda^{\ve}(X_{s},Y(A_{s}^{\ve}))ds\ge\int_{0}^{t}l(X_{s},Y(A_{s}^{0}))\ind{Y(A_{s}^{0})>\alpha^{-1}(X_{s})}ds
\]
using that $l$ is continuous and $\ve\mapsto Y(A_{s}^{\ve})$ is
continuous and increasing for each $s\ge0$. For the upper bound,
we write 
\[
\lambda^{\ve}(x,y)\le l(x,y)\ind{y>\alpha^{-1}(x)}+l(x,\alpha^{-1}(x))\ind{\alpha^{-1}(x)-\ve<y\le\alpha^{-1}(x)}.
\]
This and the fact that $\limsup_{\ve\downarrow0}\{\alpha^{-1}(X_{s})-\ve<Y(A_{s}^{\ve})\le\alpha^{-1}(X_{s})\}\subset\{Y(A_{s}^{0})=\alpha^{-1}(X_{s})\}$
yield 
\[
\limsup_{\ve\downarrow0}\lambda^{\ve}(X_{s},Y(A_{s}^{\ve}))\le l(X_{s},Y(A_{s}^{0}))\ind{Y(A_{s}^{0})>\alpha^{-1}(X_{s})}+l(X_{s},\alpha^{-1}(X_{s}))\ind{Y(A_{s}^{0})=\alpha^{-1}(X_{s})}.
\]
Recall that $l$ is bounded, so we can apply reverse Fatou's lemma
\[
\limsup_{\ve\downarrow0}\int_{0}^{t}\lambda^{\ve}(X_{s},Y(A_{s}^{\ve}))ds\le\int_{0}^{t}l(X_{s},Y(A_{s}^{0}))\ind{Y(A_{s}^{0})>\alpha^{-1}(X_{s})}+l(X_{s},\alpha^{-1}(X_{s}))\ind{Y(A_{s}^{0})=\alpha^{-1}(X_{s})}ds.
\]
The process $s\mapsto A_{s}^{0}$ is increasing, hence of finite variation,
and right-continuous. Functions $Y$ and $\alpha$ are continuous
and $\{Y(A_{s}^{0})=\alpha^{-1}(X_{s})\}=\{\alpha\big(Y(A_{s}^{0})\big)=X_{s}\}$.
Since function $l$ is bounded, Lemma \ref{lem:equal_null} implies
that the integral $\int_{0}^{\infty}l(X_{s},\alpha^{-1}(X_{s}))\ind{Y(A_{s}^{0})=\alpha^{-1}(X_{s})}ds=0$
$\P_{x}$-a.s. We can therefore conclude that the following limit
exists $\P_{x}$-a.s. (with the measure zero set independent of $t$)
\[
\lim_{\ve\downarrow0}\int_{0}^{t}\lambda^{\ve}(X_{s},Y(A_{s}^{\ve}))ds=\int_{0}^{t}l(X_{s},Y(A_{s}^{0}))\ind{Y(A_{s}^{0})\ge\alpha^{-1}(X_{s})}ds=\int_{0}^{t}\lambda(X_{s},Y(A_{s}^{0}))ds.
\]
Hence $(A_{t}^{0})_{t\ge0}$ satisfies \eqref{eqn:int_A} and we will
use it as a definition of the process $(A_{t})_{t\ge0}$ from the
statement of the proposition. From \eqref{eqn:int_A} we deduce that
$(A_{t})$ is continuous $\P_{x}$-a.s.

\textbf{Markov property:} As the process $(A_{t})$ is constructed
$\omega$ by $\omega$, it is sufficient to show that for any $u\ge0$
and $s\ge0$ we have $A_{u+s}^{0}=\bar{A}_{s}^{0}$, where $\bar{A}_{t}^{0}=\lim_{\ve\downarrow0}\bar{A}_{t}^{\ve}$
with $(\bar{A}_{t}^{\ve})$ being the unique solution of 
\[
d\bar{A}_{t}^{\ve}=\lambda^{\ve}(X_{u+t},Y(\bar{A}_{t}^{\ve}))dt,\qquad\bar{A}_{0}^{\ve}=A_{u}^{0}.
\]
This is not immediate as it is possible that $A_{u}^{\ve}>A_{u}^{0}$
for all $\ve>0$ which implies $\bar{A}_{t}^{\ve}<A_{u+t}^{\ve}$,
at least for $t\le T(\omega)$ for some $T(\omega)>0$. To overcome
this problem, define $(\bar{A}_{t}^{\ve,\delta})$ as a solution to
\eqref{eqn:ODE_A_ve} with the initial condition $\bar{A}_{0}^{\ve,\delta}=A_{u}^{0}+\delta$,
for $\delta>0$. By the continuous dependence of the solution to \eqref{eqn:ODE_A_ve}
on the initial condition and the comparison principle for ODEs, we
have $\bar{A}_{t}^{\ve}=\inf_{\delta>0}\bar{A}_{t}^{\ve,\delta}$.
The mapping $\ve\mapsto\bar{A}_{t}^{\ve}$ is increasing, hence, $\bar{A}_{t}^{0}=\inf_{\ve,\delta>0}\bar{A}_{t}^{\ve,\delta}$.

Fix $\omega\in\Omega$ (we omit it in the notation for the clarity
of exposition). For any $\delta>0$ there is $\ve>0$ such that $A_{u}^{\ve}\le A_{u}^{0}+\delta$.
By the uniqueness of solutions to \eqref{eqn:ODE_A_ve} and the comparison
principle, we have $A_{u+s}^{\ve}\le\bar{A}_{s}^{\ve,\delta}$, $s\ge0$,
so $A_{u+s}^{0}\le\bar{A}_{s}^{\ve,\delta}$. This implies that $A_{u+s}^{0}\le\bar{A}_{s}^{0}$.
The opposite inequality follows from $\bar{A}_{0}^{\ve}\le A_{u}^{\ve}$
and analogous arguments as above.

\textbf{Strong Markov property:} We apply \cite[Ch.~I, Prop.~8.2]{Blumenthal}.
Condition (S.R.) for the process $(X_{t},A_{t})_{t\ge0}$ is immediate.
Indeed, fix a stopping time $\sigma$. Then $X_{\sigma}$ is $\cF_{\sigma}$-measurable
since $(X_{t})$ is strong Markov. Due to the boundedness of $\lambda$,
the process $(A_{t})$ does not explode at a finite time. Furthermore,
for any $\ve>0$, $A_{\sigma}^{\ve}(\omega)$ is defined based on
the trajectory $(X_{t}(\omega))_{t\le\sigma(\omega)}$, so $A_{\sigma}^{\ve}$
is $\cF_{\sigma}$-measurable. The random variable $A_{\sigma}^{0}$
is also $\cF_{\sigma}$-measurable as an $\P_{x}$-a.s. limit of a
sequence of $\cF_{\sigma}$-measurable random variables; recall that
the probability zero set from Lemma \ref{lem:equal_null} is universal
for all $\sigma$.

Condition (S.M.)' of \cite[Ch.~I, Prop.~8.2]{Blumenthal} is proved
analogously as the Markov property but with $u$ replaced by $\sigma(\omega)$. 
\end{proof}
\begin{proof}[Proof of Lemma \ref{lem:prop_A}]
The monotonicity of the mapping $a\mapsto A_{t}^{a}$ follows from
the comparison principle for ODEs applied to \eqref{eqn:ODE_A_ve}.
To prove the second statement, we take $a>a'$ and write 
\[
A_{t}^{a}-A_{t}^{a'}=A_{0}^{a}-A_{0}^{a'}+\int_{0}^{t}\big[\lambda(X_{s},Y(A_{s}^{a}))-\lambda(X_{s},Y(A_{s}^{a'}))\big]ds\le A_{0}^{a}-A_{0}^{a'}=a-a',
\]
where we used that the integrand is non-positive because $A_{s}^{a}\ge A_{s}^{a'}$
and the mapping $a\mapsto\lambda(x,Y(a))$ is decreasing (by Lemma
\ref{lem:monotonicity_lambda}). The difference $A_{t}^{a}-A_{t}^{a'}$
is bounded from below by $0$ from the first part of the statement.

The monotonicity in $x$ follows from the observation that $X_{t}^{x}\ge X_{t}^{x'}$
for $x>x'$ and $t\ge0$ by the comparison principle for SDEs, so,
using Lemma \ref{lem:monotonicity_lambda}, 
\[
\lambda(X_{t}^{x},y)\le\lambda(X_{t}^{x'},y)\qquad\text{for all \ensuremath{y\in[\theta_{L},\theta_{U}]}.}
\]

For the continuity, fix $a\ge0$ and take $x_{n}\uparrow x_{0}$.
Then $A_{t}^{x_{n}}\ge A_{t}^{x_{0}}$ and so $\alpha(Y(A_{t}^{x_{n}}))\le\alpha(Y(A_{t}^{x_{0}}))$.
Notice that $\lambda(x,a)$ is bounded from above by $r\theta_{U}/(m_{\min}-\theta_{U})<\infty$,
where $m_{\min}=\inf_{x\in\cI}m(x)$. By Fatou's lemma 
\begin{align*}
A_{t}^{\infty}:=\limsup_{n\to\infty}A_{t}^{x_{n}} & \le a+\int_{0}^{t}\limsup_{n\to\infty}\lambda(X_{s}^{x_{n}},Y(A_{s}^{x_{n}}))ds\le a+\int_{0}^{t}\limsup_{n\to\infty}\lambda(X_{s}^{x_{n}},Y(A_{s}^{x_{0}}))ds\\
 & =a+\int_{0}^{t}\lambda(X_{s}^{x_{0}},Y(A_{s}^{x_{0}}))ds=A_{t}^{x_{0}},
\end{align*}
where the second inequality is based on the monotonicity of $a\mapsto\lambda(x,Y(a))$
(Lemma \ref{lem:monotonicity_lambda}) and the penultimate equality
is because $\lambda(z_{n},y)\downarrow\lambda(z,y)$ for $z_{n}\uparrow z$,
and $X_{s}^{x_{n}}\uparrow X_{s}^{x_{0}}$ by \eqref{eqn:cont_X_mono}
(the convergence is for $\omega$ outside of $\P$-negligible set
independent from $s$). Combined with the opposite inequality as $A_{t}^{x_{n}}\ge A_{t}^{x_{0}}$,
we obtain $A_{t}^{x}=A_{t}^{\infty}$. By the arbitrariness of $t$,
$A_{t}^{x_{n}}\downarrow A_{t}^{x_{0}}$ for all $t\ge0$. Dini's
theorem implies that the convergence is uniform on compact sets.

We now turn our attention to the case of $x_{n}\downarrow x_{0}$.
Then $A_{t}^{x_{n}}\le A_{t}^{x_{0}}$ and so $\alpha(Y(A_{t}^{x_{n}}))\ge\alpha(Y(A_{t}^{x_{0}}))$.
By Fatou's lemma 
\begin{align*}
A_{t}^{\infty}:=\liminf_{n\to\infty}A_{t}^{x_{n}} & \ge a+\int_{0}^{t}\liminf_{n\to\infty}\lambda(X_{s}^{x_{n}},Y(A_{s}^{x_{n}}))ds\ge a+\int_{0}^{t}\liminf_{n\to\infty}\lambda(X_{s}^{x_{n}},Y(A_{s}^{x_{0}}))ds\\
 & \ge a+\int_{0}^{t}\frac{rY(A_{s}^{x_{0}})-D(X_{s}^{x_{0}})}{m(X_{s}^{x_{0}})-Y(A_{s}^{x_{0}})}\indd{X_{s}^{x_{0}}<\alpha(Y(A_{s}^{x_{0}}))}ds,
\end{align*}
where in the last inequality, we use $X_{s}^{x_{n}}\downarrow X_{s}^{x_{0}}$
by \eqref{eqn:cont_X_mono} to argue the convergence of the fraction
due to its continuity, and the convergence of the indicator functions
because of $\liminf_{n\to\infty}\{X_{s}^{x^{n}}\le\alpha(Y(A_{s}^{x_{0}}))\}\supset\{X_{s}^{x_{0}}<\alpha(Y(A_{s}^{x_{0}}))\}$.
Denote by $\bar{A}_{t}$ the right-hand side of the above estimate.
When we recall that $A_{t}^{\infty}\le A_{t}^{x_{0}}$, we obtain
$A_{t}^{x_{0}}\ge\bar{A}_{t}$ and both processes are continuous.
For any $t\ge0$, we compute 
\begin{multline*}
A_{t}^{x_{0}}-\bar{A}_{t}=\int_{0}^{t}\frac{rY(A_{s}^{x_{0}})-D(X_{s}^{x_{0}})}{m(X_{s}^{x_{0}})-Y(A_{s}^{x_{0}})}\big(\indd{X_{s}^{x_{0}}\le\alpha(Y(A_{s}^{x_{0}}))}-\indd{X_{s}^{x_{0}}<\alpha(Y(A_{s}^{x_{0}}))}\big)ds\\
\le\lambda_{\max}\int_{0}^{t}\indd{X_{s}^{x_{0}}=\alpha(Y(A_{s}^{x_{0}}))}ds,
\end{multline*}
where $\lambda_{\max}$ is the upper bound for $\lambda$. By Lemma
\ref{lem:equal_null}, the right-hand side is $\P$-a.s. zero and
the null set can be taken independent of $t$. Hence $A_{t}^{x_{0}}$
and $\bar{A}_{t}$ are $\P$-indistinguishable. Since $A_{t}^{x_{0}}\ge A_{t}^{x_{n}}$
and $A_{t}^{\infty}=\liminf_{n\to\infty}A_{t}^{x_{n}}\ge\bar{A}_{t}$,
we have that $A_{t}^{x_{n}}$ converges to $A_{t}^{x_{0}}$ for all
$t$ outside of $\P$-negligible set. By Dini's theorem, the convergence
is uniform for $t$ on compact sets. Due to the monotonicity of $x\mapsto A_{t}^{x}$
the convergence over monotone sequences $x_{n}$ extends to general
sequences. 
\end{proof}
\begin{proof}[Proof of Proposition \ref{prop:continuity}]
There are continuous bounded functions $\lambda^{\ve},\lambda_{\ve}$
such that $\lambda^{\ve}\ge\lambda\ge\lambda_{\ve}$, and $\lambda^{\ve}\downarrow\lambda$
and $\lambda_{\ve}\uparrow\lambda$ pointwise as $\ve\downarrow0$;
see \eqref{eqn:lambda_ve_def} for an explicit definition of $\lambda^{\ve}$.
This is because the discontinuity of $\lambda$ is on a continuous
curve $\{(x,a):\ x=\alpha(Y(a))\}$. Consider 
\[
\tilde{v}^{\ve}(x,a;\theta)=\sup_{\sigma\in\cT(\cF_{t})}\E_{xa}\Big[\int_{0}^{\sigma}e^{-rs-A_{s}}\big(D(X_{s})-r\theta+\lambda^{\ve}(X_{s},Y(A_{s}))(m(X_{s})-\theta)\big)ds\Big],
\]
and analogously $\tilde{v}_{\ve}$; \edt{notice that the dynamics
of the process $(A_{t})$ does not depend on $\ve$}. Since $m(x)>\theta_{U}$
for all $x\in\cI$, we have $\tilde{v}^{\ve}\ge\tilde{v}\ge\tilde{v}_{\ve}$.

Assume that $\tilde{v}^{\ve}$ and $\tilde{v}_{\ve}$ are continuous;
the proof will come shortly. For any $\delta$-optimal stopping time
$\sigma_{\delta}$ for $\tilde{v}$, we have 
\begin{align*}
\tilde{v}(x,a;\theta)-\delta & \le\E_{xa}\Big[\int_{0}^{\sigma_{\delta}}e^{-rs-A_{s}}\big(D(X_{s})-r\theta+\lambda(X_{s},Y(A_{s}))(m(X_{s})-\theta)\big)ds\Big]\\
 & =\lim_{\ve\downarrow0}\E_{xa}\Big[\int_{0}^{\sigma_{\delta}}e^{-rs-A_{s}}\big(D(X_{s})-r\theta+\lambda_{\ve}(X_{s},Y(A_{s}))(m(X_{s})-\theta)\big)ds\Big]\\
 & \le\liminf_{\ve\downarrow0}\tilde{v}_{\ve}(x,a;\theta),
\end{align*}
where the equality is by the dominated convergence theorem and the
last inequality is because the expectation in the second line is dominated
by $\tilde{v}_{\ve}(x,a;\theta)$. Combining this for arbitrary $\delta>0$
with $\tilde{v}\ge\tilde{v}_{\ve}$, we obtain that $\tilde{v}_{\ve}$
converges pointwise from below to $\tilde{v}$, hence $\tilde{v}$
is lower semicontinuous.

The mapping $\ve\mapsto\tilde{v}^{\ve}(x,a;\theta)$ is increasing
for each fixed $x,a,\theta$ hence the limit $\tilde{v}^{0}(x,a;\theta):=\lim_{\ve\downarrow0}\tilde{v}^{\ve}(x,a;\theta)$
exists. Using the continuity of $\tilde{v}^{\ve}$, by the general
theory of optimal stopping (see, e.g., \cite{elkaroui1981,peskir2006optimal}),
the stopping time $\sigma^{\ve}=\inf\{t\ge0:\ \tilde{v}^{\ve}(X_{t},A_{t};\theta)=0\}$
is optimal for $\tilde{v}^{\ve}$. It is also increasing in $\ve$
due to the monotonicity of $\tilde{v}^{\ve}$ in $\ve$. Hence 
\begin{align*}
\tilde{v}^{0}(x,a;\theta) & =\lim_{\ve\downarrow0}\E_{xa}\Big[\int_{0}^{\infty}\ind{t\le\sigma^{\ve}}e^{-rs-A_{s}}\big(D(X_{s})-r\theta+\lambda^{\ve}(X_{s},Y(A_{s}))(m(X_{s})-\theta)\big)ds\Big]\\
 & \le\E_{xa}\Big[\int_{0}^{\infty}\ind{t\le\sigma^{0}}e^{-rs-A_{s}}\big(D(X_{s})-r\theta+\lambda(X_{s},Y(A_{s}))(m(X_{s})-\theta)\big)ds\Big]\\
 & \le\tilde{v}(x,a;\theta),
\end{align*}
where in the first inequality we used the reverse Fatou's lemma (as
the integrand is bounded from above) and $\sigma^{0}=\lim_{\ve\downarrow0}\sigma^{\ve}=\inf_{\ve>0}\sigma^{\ve}$
is a stopping time as an infimum of stopping times. The inequality
$\tilde{v}^{0}\ge\tilde{v}$ is immediate from $\tilde{v}^{\ve}\ge\tilde{v}$.
Hence $\tilde{v}^{\ve}$ converges to $\tilde{v}$ pointwise from
above, which implies that $\tilde{v}$ is upper semicontinuous.

Combining the above two semicontinuity results shows the continuity
of $\tilde{v}$. By the general optimal stopping theory, see \cite{elkaroui1981,peskir2006optimal},
the optimal stopping time is given by the formula in the statement
of the proposition.

In remains to show the continuity of $\tilde{v}^{\ve}$. The proof
for $\tilde{v}_{\ve}$ is analogous. Take $(x_{n},a_{n})$ converging
to $(x,a)$ and such that $(x_{n})$ is monotone (which can be assumed
without loss of generality). For any $T>0$, we have the following
estimate 
\begin{align*}
&\big|\tilde{v}^{\ve}(x_{n},a_{n};\theta)-\tilde{v}^{\ve}(x,a;\theta)\big| \\
&\le \sup_{\sigma \in \cF(\cF_t), \sigma \le T} \bigg\{
\E\bigg[\int_{0}^{\sigma}e^{-rs}\Big|e^{-A_{s}^{x_{n},a_{n}}}\big(D(X_{s}^{x_{n}})-\theta\big)-e^{-A_{s}^{x,a}}\big(D(X_{s}^{x})-\theta\big)\Big|ds\bigg]\\
 & \hspace{92pt}+\E\bigg[\int_{0}^{\sigma}e^{-rs}\Big|e^{-A_{s}^{x_{n},a_{n}}}\lambda^{\ve}(X_{s}^{x_{n}},Y(A_{s}^{x_{n},a_{n}}))\big(m(X_{s}^{x_{n}})-\theta\big)\\
 & \hspace{160pt}-e^{-A_{s}^{x,a}}\lambda^{\ve}(X_{s}^{x},Y(A_{s}^{x,a}))\big(m(X_{s}^{x})-\theta\big)\Big|ds\bigg]\bigg\}+2e^{-rT}\frac{1}{r}C\\
& \le\E\bigg[\int_{0}^{T}e^{-rs}\Big|e^{-A_{s}^{x_{n},a_{n}}}\big(D(X_{s}^{x_{n}})-\theta\big)-e^{-A_{s}^{x,a}}\big(D(X_{s}^{x})-\theta\big)\Big|ds\bigg]\\
 & \hspace{12pt}+\E\bigg[\int_{0}^{T}e^{-rs}\Big|e^{-A_{s}^{x_{n},a_{n}}}\lambda^{\ve}(X_{s}^{x_{n}},Y(A_{s}^{x_{n},a_{n}}))\big(m(X_{s}^{x_{n}})-\theta\big)\\
 & \hspace{80pt}-e^{-A_{s}^{x,a}}\lambda^{\ve}(X_{s}^{x},Y(A_{s}^{x,a}))\big(m(X_{s}^{x})-\theta\big)\Big|ds\bigg]+2e^{-rT}\frac{1}{r}C,
\end{align*}
where 
\[
C:=\sup_{(x,a)\in\cI\times[0,\infty)}\big(D(x)-r\theta+\lambda^{\ve}(x,Y(a))(m(x)-\theta)\big)<\infty.
\]
By Lemma \ref{lem:prop_A} (c.f. Remark \ref{rem:cont_X}), $A_{s}^{x_{n},a_{n}}$
converges pointwise to $A_{s}^{x,a}$ for all $s\in[0,T]$ and $\omega$
outside of a $\P$-negligible set. We also have convergence of $(X_{s}^{x_{n}})_{s\in[0,T]}$
to $(X_{s}^{x})_{s\in[0,T]}$ outside of a $\P$-negligible set. We
can therefore conclude, by the dominated convergence theorem, that
\[
\lim_{n\to\infty}\big|\tilde{v}^{\ve}(x_{n},a_{n};\theta)-\tilde{v}^{\ve}(x,a;\theta)\big|\le2e^{-rT}\frac{1}{r}C.
\]
Since $T$ is arbitrary, this shows the continuity of $\tilde{v}^{\ve}$. 
\end{proof}
%%%%%%%%%%%%%%%%%%%%%%%%%%%%%%%%%%%%%%%%%%%%%%%%%%%%%%%%%%%%%%%%%%%%%%%%%%%%%%%%%%%%%%%%%%%%%

\section{Two special cases} \label{sec:Discussions}
\edt{
In this section, we explore how our solution behaves
when we remove either the stochastic state variable or the private
information from our model. These are two special cases that have
been studied in the previous literature on exit games. By comparing
our results with known results from the literature, we establish the
robustness and generality of our solution.

First, we remove the dynamics of the stochastic
state variable from the solution by setting (formally) $\mu(\cdot)=b(\cdot)=0$
with an initial value of the state variable set to $x$. We assume
that $D(x)/r<\theta_{L}$ so that all types of players have an incentive
to exit. In this case, $\alpha(\theta)=\infty$ for all $\theta\in[\theta_{L},\theta_{U}]$,
so the rate of exit reduces to 
\[
\lambda(x,y)=\frac{ry-D(x)}{m(x)-y}.
\]
One striking difference from the case of a dynamic state variable
is that $\lambda(x,y)$ is always strictly positive. Furthermore, from (\ref{eqn:tau_form}),
the exit strategy of a player of type $\theta$ is given by 
\[
\hat{\tau}(\theta)=\inf\{t\ge0:\ Y_{t}<\theta\}.
\]
Because $Y_{t}$ possesses a deterministic dynamics, $\hat{\tau}(\theta)$
is also deterministic. Thus, each type of a player chooses a deterministic
time to exit at the outset of the game. This feature is consistent
with \cite{Milgrom1985,Fudenberg1986}, who examined deterministic
exit games with private types.}

\edt{Next, we keep the stochastic state variable but
remove the uncertainty in the exit value. More precisely, we take
the limits $\theta_{L}\rightarrow\theta$ and $\theta_{U}\rightarrow\theta$
and study the behaviour of the equilibrium strategy profile. This requires care as in the limit the distribution of a player's type degenerates to a deterministic quantity $\theta$. To get around this difficulty, we
re-express (\ref{eqn:tau_form}) in terms of the process $A_{t}$ as follows:
\[
\tau_{i}=\inf\{t\ge0:\ e^{-A_{t}} < F(\theta_{i})\}.
\]
Recall that $F(\theta_{i})$ is uniformly distributed within the interval
$[0,1]$, and hence, we can reformulate the condition $e^{-A_{t}}<F(\theta_{i})$
as $e^{-A_{t}}<\hat{\epsilon}_{i}$, where $\hat{\epsilon}_{i}$ is
a random variable uniformly distributed on $[0,1]$. We now take the limits $\theta_{L}\rightarrow\theta$ and $\theta_{U}\rightarrow\theta$:
\begin{equation}\label{eq:tau-limit}
\tau_{i}=\inf\{t\ge0:\ e^{-A_{t}} < \hat{\epsilon}_{i}\}.
\end{equation}
In the limit the dynamics of $(A_t)$ takes the form
\[
dA_t = \frac{r\theta-D(X_t)}{m(X_t)-\theta}\ind{X_t \le\alpha(\theta)} =: \tilde\lambda (X_t).
\]
We can identify \eqref{eq:tau-limit} with a mixed
strategy equilibrium for stochastic exit games with known exit values.
Indeed, in accordance with the standard definition
of a mixed strategy, each player can be viewed as having a randomisation device uniformly distributed on $[0,1]$ which generates a random value
in the beginning of the game; we can identify this randomisation device with $\hat{\epsilon}_{i}$.
The player then exits at the first instance that $e^{-A_{t}}$ falls
below $\hat{\epsilon}_{i}$. Because the exit rate $\tilde\lambda(X_{t})$
is positive only when $X_{t}\le\alpha(\theta)$, the strategy \eqref{eq:tau-limit}
coincides with the mixed strategy found by \cite{Steg2015,Georgiadis2019}. We conclude
that our equilibrium converges to the established results as $\theta_{L}$
and $\theta_{U}$ approach $\theta$.
}

\section{Uniqueness of absolutely continuous symmetric Bayesian equilibria}

\label{sec:uniqeness}

In the previous section, we have constructed a symmetric equilibrium.
Here we show that this is the unique equilibrium in a certain subclass
of symmetric equilibria in which individual strategies are of the
form \eqref{eqn:tau_form}. 

\edt{Difficulties in the proof of uniqueness of symmetric equilibria driven by a belief process stem from the continuum of player types and the diffusive dynamics of the underlying state process. The first technical result shows that the player's strategy is a solution to a Markovian optimal stopping problem for almost every value of player type, i.e., the best response optimal stopping problem. The Markovian structure is provided by the state process and the belief process that defines the equilibrium. Classically, the solution of a stopping problem is given by the hitting time of a stopping set on which the value function coincides with the payoff. This is the smallest optimal stopping time, however it may not be the only one, so we cannot assume that the player's strategy determines the stopping set for the best response problem. Instead, we work with an action set which collects all points in which the equilibrium strategy prescribes to stop immediately and describe it uniquely in terms of the function $\alpha(\cdot)$ from Section \ref{sec:single_player}. Hence, every symmetric equilibrium is described by the same action set in the $2$-dimensional space comprising the state process and the belief, but there may be potentially many equilibria driven by different belief processes. The uniqueness of the latter is the second technical result of this section. The optimality of the equilibrium strategy for the best response optimal stopping problem yields that the equilibrium strategy established in the paper provides an upper bound for infinitesimal changes of the equilibrium generating process $(A_t)$. The lower bound for infinitesimal changes of $(A_t)$ arises from examining deviations in equilibrium strategies for players with nearly identical types.}

As in the previous section, for the convenience
of presentation, we rewrite strategies are of the
form \eqref{eqn:tau_form} in terms of an increasing
right-continuous process $(A_{t})_{t\ge0}$ with $A_{0}=0$: 
\begin{equation}
\tau_{i}=\inf\{t\ge0:\ A_{t}>A(\theta_{i})\},\label{eqn:tau_A}
\end{equation}
where $A(\cdot)$ is defined in \eqref{eqn:notation_A}. We also impose
Assumption \ref{ass:theta_cdf}. \edt{In this section, we restrict
our attention to absolutely continuous equilibria. The reader is referred
to \cite[Section~3.2]{Decamps2022} for a complete characterisation
of stopping times of Markovian type of the form \eqref{eqn:tau_A}.}

\begin{definition}\label{def:abs_cont_perfect_equil} A strategy
profile $(\tau_{1},\tau_{2})$ with $\tau_{i}$ given by \eqref{eqn:tau_A} \edt{equipped with the prior distribution $F$}
is called an \emph{absolutely continuous symmetric Bayesian Nash
equilibrium} if 
\begin{enumerate}
\item[(i)] the process $(X_{t},A_{t})_{t\ge0}$ is a strong Markov process; 
\item[(ii)] for $x\in\cI$ and $a\ge0$, the process $(A_{t})_{t\ge0}$ is absolutely
continuous with respect to the Lebesgue measure and satisfies for
$t\ge0$ 
\begin{equation}
A_{t}=a+\int_{0}^{t}\varphi(X_{s},Y(A_{s}))ds,\qquad\P_{x}-a.s.\label{eqn:symm_A}
\end{equation}
\edt{for a measurable function $\varphi:\cI\times[\theta_{L},\theta_{U}]\to[0,\infty)$
called a \emph{generator}.} %  \item[(iii)] there is an open set $\lambdaSet \subset \cI \times [\theta_L, \theta_U]$ and a continuous function $\widehat \varphi: \cI \times [\theta_L, \theta_U] \to [0, \infty)$ such that
%  \begin{equation}\label{eqn:varphi}
% \varphi(x,y) = \ind{(x, y) \in \overline \lambdaSet}\, \widehat \varphi(x,y).
%  \end{equation}
\item[(iii)] For any $a\ge0$ and $x\in\cI$, stopping times $\tau_{1},\tau_{2}$
given by 
\begin{equation}
\tau_{i}=\inf\{t\ge0:\ A_{t}>A(\hat{\theta}_{i})\},\quad i=1,2,\label{eqn:tau_A_a}
\end{equation}
with $A_{t}$ satisfying \eqref{eqn:symm_A}, form a Nash equilibrium
for $x$ in the sense of Def.~\ref{def:Nash} with $\hat{\theta}_{i}=F_{|Y(a)}^{-1}(F(\theta_{i}))\sim F_{|Y(a)}$,
$i=1,2$, where $F_{|y}(z)=F(z\wedge y)/F(y)$. 
\end{enumerate}
\end{definition} 
\edt{ \begin{definition} An absolutely continuous
symmetric Bayesian Nash equilibrium is called \emph{lower semi-continuous}
if the generator $\varphi$ is a lower semi-continuous function, and
\emph{upper semi-continuous} if the generator $\varphi$ is an upper
semi-continuous function. \end{definition} }

\edt{The above definition of an absolutely continuous symmetric Bayesian Nash equilibrium is weaker than a PBE adopted in Subsection \ref{subsec:PBE} as it does not require the strategy profile to be an equilibrium at any time $t$ even in the case of a deviation of one of a player. This weaker notion is however sufficient to prove uniqueness.}

\begin{remark} Since $F$ is continuous and strictly increasing,
$F(\theta_{i})\sim U(0,1)$, so $\hat{\theta}_{i}$ as defined above
has the distribution $F_{|Y(a)}$. The cumulative distribution function
$F_{|y}$ is strictly increasing on its domain, hence $\sigma(\theta_{i})=\sigma(\hat{\theta}_{i})$
(here $\sigma(Z)$ is the $\sigma$-algebra generated by $Z$), so
the information of a player observing $\theta_{i}$ is identical to
the information obtained from observing $\hat{\theta}_{i}$. \end{remark}

\begin{remark} Notice that $\varphi$ on the right-hand side of \eqref{eqn:symm_A}
can be discontinuous, so the solution does not exist in the classical
sense. It may also not be unique as the discontinuity may be crossed
infinitely many times in any time interval (due to the infinite variation
of $(X_{t})$) and $\varphi$ is not assumed to be Lipschitz. This lack
of uniqueness does not pose any mathematical difficulties for the analysis as we
only need that $(X_{t},A_{t})$ is a strong Markov process and $(A_{t})$
is increasing and absolutely continuous with a lower semi-continuous
weak derivative. An example of such an absolutely continuous symmetric
Bayesian Nash equilibrium can be found in Section \ref{sec:equil}.
\end{remark}

\begin{remark} Notice that for any $a\ge0$ and $x\in\cI$, $\tau_{i}$
given by \eqref{eqn:tau_A} or \eqref{eqn:tau_A_a} with $A_{0}=a$
is an $(\cF_{t}^{i})$-stopping time, $i=1,2$, because the filtration
$(\cF_{t})_{t\ge0}$ is complete with respect to $\P_{x}$, $x\in\cI$,
see \cite[Chapter I, Theorem 10.7]{Blumenthal}. \end{remark}

\edt{Before formulating main results of this section, we need to introduce
the following notation. We define an upper semi-continuous envelope
$\varphi^{*}$ of $\varphi$ by 
\[
\varphi^{*}(x,y)=\limsup_{(x',y')\to(x,y)}\varphi(x',y')
\]
and a lower semi-continuous envelope $\varphi_{*}$ of $\varphi$
by 
\[
\varphi_{*}(x,y)=\liminf_{(x',y')\to(x,y)}\varphi(x',y').
\]
}

\edt{ \begin{theorem}\label{thm:unique} For an absolutely continuous
Bayesian Nash equilibrium with a lower semi-continuous generator
$\varphi$, its upper semi-continuous envelope $\varphi^{*}$ coincides
with $\lambda$, i.e., we have $\varphi^{*}=\lambda$. \end{theorem}

Before proving this theorem, we discuss its consequences. Theorem \ref{thm:unique} establishes the uniqueness of the absolutely continuous
symmetric Bayesian Nash equilibrium with a lower semi-continuous
generator. This does not cover the case of the equilibrium defined by $\lambda$
in the previous section which is upper semi-continuous, but we will
be able to strengthen this result. \begin{definition}\label{def:semi-cont}
Assume that a measurable function $\varphi$ is a generator of an
absolutely continuous Bayesian Nash equilibrium. If the set $\Delta\varphi:=\{(x,y)\in\cI\times[\theta_{L},\theta_{U}]:\varphi^{*}(x,y)\ne\varphi_{*}(x,y)\}$
is a countable union of graphs of the form $x=h(y)$ for a function
$h:[\theta_{L},\theta_{U}]\to\cI$ of finite variation, we will call
the generator \emph{semi-continuously bounded}. \end{definition}
Notice that the generator $\lambda$ from the previous section is
semi-continuously bounded as the set $\Delta\lambda$ from the above
definition consists of a graph of function $\alpha$, the stopping
boundary of the single-player problem.

\begin{lemma}\label{lem:semi_cont} Let $\varphi$ be a semi-continuously
bounded generator and $(X_{t},A_{t})$ the pair of processes from
Definition \ref{def:abs_cont_perfect_equil}. The process $(A_{t})_{t\ge0}$
satisfies \eqref{eqn:symm_A} with the upper semi-continuous envelope
$\varphi^{*}$ and with the lower semi-continuous envelope $\varphi_{*}$
of $\varphi$. \end{lemma} 
\begin{proof}
The result is an immediate consequence of Lemma \ref{lem:equal_null}. 
\end{proof}
We recall, see Remark \ref{rem:carath}, that the process $(A_{t})$
may not be uniquely determined. However, the process constructed in
Section \ref{sec:equil} is the maximal solution, so it yields the
smallest stopping times.

We conclude with the main uniqueness result of the paper.

\begin{corollary} Generator $\lambda$ determines the unique absolutely
continuous  Bayesian Nash equilibrium in the family of equilibria
with semi-continuously bounded generators. \end{corollary} 
\begin{proof}
Let $\varphi$ be a semi-continuously bounded generator of an absolutely
continuous Bayesian Nash equilibrium. By applying Theorem \ref{thm:unique}
to $\varphi_{*}$, we have $\varphi^{*}=\lambda$. Furthermore, Lemma
\ref{lem:semi_cont} implies that the process $(A_{t})$ satisfies
\eqref{eqn:symm_A} with $\lambda$ as well. Noting that $(A_{t})$
itself determines player's strategies, we obtain uniqueness. 
\end{proof}
The above corollary is the main uniqueness result of the paper. We
argue that the assumption of semi-continuous boundedness of a generator
is quite natural. As remarked, the function $\lambda$ generating
the equilibrium of the previous section is semi-continuously bounded
with the jump set $\Delta\lambda$ consisting of only the graph of the stopping
boundary of a single player problem. As the generator $\varphi$ from
\eqref{eqn:symm_A} has the interpretation of the intensity of exiting
of an opponent, it is unlikely that one can explicitly construct $\varphi$ with
a more complex jump set $\Delta\varphi$ than stipulated in Definition
\ref{def:semi-cont}, so our result may be viewed as a guarantee that
there will be no other \emph{explicitly} constructed absolutely continuous
symmetric  Bayesian Nash equilibrium in the problem.

The remaining of this section is divided into two parts. In the first
part, we define the action sets for the best response
problem and establish their properties. The second part is devoted to the derivation of an upper
and lower bounds for $\varphi$ and its upper semi-continuous envelope
and the proof of Theorem \ref{thm:unique}. }

\subsection{Properties of action sets for best response problems}

\edt{Before proceeding further, we recall properties of lower semi-continuous
functions. \begin{lemma}{\cite[p. 51]{Royden1988}}\label{lem:semicont}
The following hold: 
\begin{enumerate}
\item Function $\varphi$ is lower semi-continuous iff $\liminf_{(x',y')\to(x,y)}\varphi(x',y')\ge\varphi(x,y)$
for any $(x,y)\in\cI\times[\theta_{L},\theta_{U}]$; 
\item Function $\varphi$ is lower semi-continuous iff the set $\{(x,y)\in\cI\times[\theta_{L},\theta_{U}]:\ \varphi(x,y)>z\}$
is open in $\cI\times[\theta_{L},\theta_{U}]$ (i.e., its complement
is closed) for any $z\in\mathbb{R}$. 
\item Function $\varphi$ is upper semi-continuous iff $(-\varphi)$ is
lower semi-continuous. 
\end{enumerate}
\end{lemma}}

Throughout the remaining of this section, we assume that $(A_{t})_{t\ge0}$
given by (\ref{eqn:symm_A}) is the process that characterises a \emph{lower
semi-continuous} absolutely continuous symmetric Bayesian Nash
equilibrium $(\tau_{1},\tau_{2})$. For $\theta\in[\theta_{L},\theta_{U}]$
and $a\le A(\theta)$, define $v(x,a;\theta)$ by \eqref{eqn:def_v}
with $\lambda$ replaced by $\varphi$. In the remainder of this section,
we will refer to equations from Section \ref{sec:equil} without further
mentioning that $\lambda$ is to be replaced by $\varphi$. The first
key step is proving a converse of Corollary \ref{corr:equil}. \begin{proposition}\label{prop:converse}
The stopping time 
\begin{equation}
\tau_{\theta}=\inf\{t\ge0:A_{t}>A(\theta)\},\label{eqn:tau_theta}
\end{equation}
with $A_{0}=a$ and the dynamics \eqref{eqn:symm_A}, is an optimal
stopping time for $v(x,a;\theta)$ for $\theta\in(\theta_{L},Y(a)]$.
\end{proposition} 
\begin{proof}
Fix $a\ge0$ and consider an equilibrium $(\tau_{1},\tau_{2})$ in
Definition \ref{def:abs_cont_perfect_equil}(iii). The decomposition
of the stopping time $\tau_{i}$ from Proposition \ref{prop:structure}
is $\hat{\tau}_{i}(\omega,\theta)=\tau_{\theta}(\omega)$ with $\tau_{\theta}$
from \eqref{eqn:tau_theta}. We further have 
\[
J_{1}(x,\tau_{1},\tau_{2})=\int_{\theta_{L}}^{Y(a)}J(x,\tau_{\theta},\tau_{2};\theta)F_{|Y(a)}(d\theta),
\]
where $J(x,\sigma,\tau_{2};\theta)$ is defined in \eqref{eqn:J}.
Using representation \eqref{eqn:J_int} of $J$ and taking into account
that $A_{0}=a\ge0$ while Lemma \ref{lem:J_int} assumed $a=0$, we
have 
\begin{equation}
J(x,\sigma,\tau_{i};\theta)=e^{a}\E_{x}\Big[\int_{0}^{\sigma}e^{-rs-A_{s}}D(X_{s})ds+\theta e^{-r\sigma-A_{\sigma-}}+\int_{[0,\sigma)}e^{-rs-A_{s}}m(X_{s})dA_{s}\Big].\label{eqn:J_int_a}
\end{equation}
Hence, from \eqref{eqn:def_v} 
\[
v(x,a;\theta)=\sup_{\sigma\in\cT(\cF_{t})}e^{-a}J(x,\sigma,\tau_{2};\theta),\qquad\theta\in[\theta_{L},Y(a)].
\]
The proof will now follow by contradiction. Assume that there is $\hat{\theta}\in(\theta_{L},Y(a)]$
such that 
\begin{equation}
J(x,\tau_{\hat{\theta}},\tau_{2};\hat{\theta})\le e^{a}v(x,a;\hat{\theta})-\ve\label{eqn:J_eps_bound}
\end{equation}
for some $\ve>0$. Notice that the mapping $(t,\theta)\mapsto J(x,t,\tau_{2};\theta)$
is continuous due to the continuity of $A_{t}$, see \eqref{eqn:J_int_a}.
We also have that the mapping $\theta\mapsto\tau_{\theta}$ is left-continuous.
Indeed, take a sequence $\theta_{n}\uparrow\theta$ and fix $\omega\in\Omega$;
we will argue pointwise. Fix any $t>\tau_{\theta}(\omega)$. We have
$A_{t}(\omega)>A(\theta)$. Using that $A(\theta_{n})\downarrow A(\theta)$,
there is $k$ such that $A_{t}(\omega)>A(\theta_{k})$, which implies
$t>\tau_{\theta_{k}}(\omega)\ge\inf_{n}\tau_{\theta_{n}}(\omega)=\lim_{n\to\infty}\tau_{\theta_{n}}(\omega)$,
where the last equality follows from the fact that the sequence $\tau_{\theta_{n}}(\omega)$
is decreasing in $n$. From the arbitrariness of $t$, we obtain that
$\tau_{\theta}(\omega)\ge\lim_{n\to\infty}\tau_{\theta_{n}}(\omega)$.
The opposite inequality is obvious as $A(\theta_{n})>A(\theta)$.

The above two observations as well as the boundedness of the terms
under the integrals in \eqref{eqn:J_int_a} imply that 
\[
\theta\mapsto J(x,\tau_{\theta},\tau_{2};\theta)
\]
is left-continuous. Hence, there is $\delta>0$ such that $\hat{\theta}-\delta\ge\theta_{L}$
and $J(x,\tau_{\theta},\tau_{2};\theta)\le e^{a}v(x,a;\theta)-\ve/2$
for $\theta\in[\hat{\theta}-\delta,\hat{\theta}]$. This will allow
us to improve the equilibrium strategy $\tau_{1}$ defined in \eqref{eqn:tau_A_a}
and lead to a contradiction. Take $\sigma^{*}\in\cT(\cF_{t})$ such
that $J(x,\sigma^{*},\tau_{2};\hat{\theta})>e^{a}v(x,a;\hat{\theta})-\ve/4$.
By the continuity of $\theta\mapsto J(x,\sigma^{*},\tau_{2};\theta)$,
there is $\delta'\in(0,\delta)$ such that 
\[
J(x,\sigma^{*},\tau_{2};\theta)>J(x,\tau_{\theta},\tau_{2};\theta),\qquad\theta\in[\hat{\theta}-\delta',\hat{\theta}].
\]
Hence the strategy $\tau_{1}'=\hat{\tau}(\cdot,\theta_{1})$ (c.f.
Proposition \ref{prop:structure}) with 
\[
\hat{\tau}(\cdot,\theta)=\begin{cases}
\sigma^{*}, & \theta\in[\hat{\theta}-\delta',\hat{\theta}],\\
\tau_{\theta}, & \text{otherwise},
\end{cases}
\]
is a strictly better response to $\tau_{2}$ than $\tau_{1}=\tau_{\theta_{1}}$,
which contradicts that $(\tau_{1},\tau_{2})$ given by \eqref{eqn:tau_A_a}
is an absolutely continuous Bayesian Nash equilibrium, contradicting
Definition \ref{def:abs_cont_perfect_equil}(iii). 
\end{proof}

For $\theta\in[\theta_{L},\theta_{U}]$, denote $S_{\theta}=\{x\in\cI:\P_{xA(\theta)}(\tau_{\theta}=0)=1\}$
\edt{and $\lambdaSet_{\theta}=\{x\in\cI:\varphi(x,\theta)>0\}$, which
is an open set by Lemma \ref{lem:semicont}.} Recall that by the
0-1 law, $\P_{xA(\theta)}(\tau_{\theta}=0)\in\{0,1\}$, so on the
complement of $S_{\theta}$ we have $\P_{xA(\theta)}(\tau_{\theta}>0)=1$. \edt{We will call $S_\theta$ \emph{the action set} for reasons explained in the remark below.}

\begin{remark} \edt{We cannot assume that the stopping time defined in \eqref{eqn:tau_A} coincides with the first hitting time of the stopping set on which the value function $v(x,a;\theta)$ coincides with the payoff as there may be many optimal stopping times; the aforementioned hitting time is the smallest of them. Therefore, the optimality of the stopping rule \eqref{eqn:tau_A} does not determine the stopping set. This motivates our less direct approach and the introduction of the set $S_{\theta}$ of those values of $x,a,\theta$ for which the stopping
rule \eqref{eqn:tau_A} stops immediately with probability one.}
\end{remark} 

\edt{ \begin{lemma}\label{lem:incl}
We have $\lambdaSet_{\theta}\subset S_{\theta}.$
\end{lemma} 
\begin{proof}
If $\varphi(x,\theta)>0$, then the first inclusion follows from the
fact that $\varphi>0$ in an open neighbourhood of $(x,\theta)$ by
Lemma \ref{lem:semicont}.
\end{proof}
}

\begin{lemma}\label{lem:S_theta_closed} Set $S_{\theta}$ is closed
in $\cI$. \end{lemma} 
\begin{proof}
Assume that $S_{\theta}$ is not closed. There is $x\in\cI\setminus S_{\theta}$
such that $B_{\ve}(x)\cap S_{\theta}\ne\emptyset$ for all $\ve>0$,
where $B_{\ve}(x)=\{x'\in\cI:\ |x-x'|<\ve\}$. We can find a monotone
sequence $(x_{n})\subset S_{\theta}$ converging to $x$. We will
assume that the sequence is increasing; arguments for a decreasing
sequence are analogous. By the regularity of $(X_{t})$, we have $\P_{x}(\sigma_{(x_{L},x)}=0)=1$,
where we write $\sigma_{B}=\inf\{t\ge0:\ X_{t}\in B\}$ for a Borel
set $B\subset\cI$ and $\sigma_{z}$ when $B=\{z\}$. For any $\ve\in(0,1)$,
we have 
\begin{align*}
\E_{xA(\theta)}\big[\tau_{\theta}\wedge1\big] & =\E_{xA(\theta)}\big[\ind{\sigma_{x_{n}}\le\ve}(\tau_{\theta}\wedge1)+\ind{\sigma_{x_{n}}>\ve}(\tau_{\theta}\wedge1)\big]\\
 & \le\E_{xA(\theta)}\big[\ind{\sigma_{x_{n}}\le\ve}(\tau_{\theta}\wedge1)+\ind{\sigma_{x_{n}}>\ve}\big]\\
 & \le\E_{xA(\theta)}\big[\ind{\sigma_{x_{n}}\le\ve}\big(\ind{\tau_{\theta}<\sigma_{x_{n}}}\sigma_{x_{n}}+\ind{\tau_{\theta}\ge\sigma_{x_{n}}}(\tau_{\theta}\wedge1)\big)\big]+\P_{xA(\theta)}(\sigma_{x_{n}}>\ve)\\
 & \le\E_{xA(\theta)}\big[\ind{\sigma_{x_{n}}\le\ve}\big(\ind{\tau_{\theta}<\sigma_{x_{n}}}\ve+\ind{\tau_{\theta}\ge\sigma_{x_{n}}}(\sigma_{x_{n}}+\E_{x_{n}A_{\sigma_{x_{n}}}}[\tau_{\theta}]\big)\big]+\P_{xA(\theta)}(\sigma_{x_{n}}>\ve)\\
 & \le\ve\P_{xA(\theta)}(\sigma_{x_{n}}\le\ve)+\P_{xA(\theta)}(\sigma_{x_{n}}>\ve),
\end{align*}
where we used the strong Markov property of $(X_{t},A_{t})$ and $\E_{x_{n}A_{\sigma_{x_{n}}}}[\tau_{\theta}]=0$
since $x_{n}\in S_{\theta}$ and $A_{\sigma_{x_{n}}}\ge A(\theta)$.
By the dominated convergence theorem and the regularity of $(X_{t})$,
we have 
\[
\lim_{n\to\infty}\P_{xA(\theta)}(\sigma_{x_{n}}\le\ve)=1,\qquad\lim_{n\to\infty}\P_{xA(\theta)}(\sigma_{x_{n}}>\ve)=0.
\]
Inserting this into the above estimates gives $\E_{xA(\theta)}\big[\tau_{\theta}\wedge1\big]\le\ve$.
This means that $\E_{xA(\theta)}\big[\tau_{\theta}\wedge1\big]=0$
as $\ve\in(0,1)$ was arbitrary. We therefore conclude that $\P_{xA(\theta)}(\tau_{\theta}=0)=1$,
which contradicts the assumption that $x\notin S_{\theta}$. 
\end{proof}
\begin{lemma}\label{lem:tau_in_S} We have $\P_{xA(\theta)}(X_{\tau_{\theta}}\in S_{\theta}\text{ or }\tau_{\theta}=\infty)=1$.
\end{lemma} 
\begin{proof}
Recall that the filtration $(\cF_{t})$ is right continuous. Define
$\sigma_{\theta}=\tau_{\theta}+\ind{\tau_{\theta}<\infty}\tau_{\theta}\circ\theta_{\tau_{\theta}}$,
where $\theta_{t}$ is the shift operator for the Markov process $(X_{t},A_{t})_{t\ge0}$.
This is a stopping time thanks to \cite[Chapter I, Thm.~8.7]{Blumenthal}.
We have $A_{\tau_{\theta}}=A(\theta)$ because of the continuity of
$(A_{t})$ and the definition of $\tau_{\theta}$. Combining this
with the strong Markov property of $(X_{t},A_{t})$ we can write 
\begin{align*}
\E_{xA(\theta)}[A_{\sigma_{\theta}}] & =\E_{xA(\theta)}\big[\ind{\tau_{\theta}=\infty}A(\theta)+\ind{\tau_{\theta}<\infty}\E_{X_{\tau_{\theta}}A_{\tau_{\theta}}}[A_{\tau_{\theta}}]\big]\\
 & =\E_{xA(\theta)}\big[\ind{\tau_{\theta}=\infty}A(\theta)+\ind{\tau_{\theta}<\infty}\E_{X_{\tau_{\theta}}A(\theta)}[A_{\tau_{\theta}}]\big]\\
 & =\E_{xA(\theta)}\big[\ind{\tau_{\theta}=\infty}A(\theta)+\ind{\tau_{\theta}<\infty}\E_{X_{\tau_{\theta}}A(\theta)}[A(\theta)]\big]=A(\theta).
\end{align*}
Hence, recalling the definition of $\tau_{\theta}$, this implies
$\P_{xA(\theta)}(\sigma_{\theta}>\tau_{\theta}\text{ and }\tau_{\theta}<\infty)=0$
and, consequently, $\P_{xA(\theta)}(\tau_{\theta}\circ\theta_{\tau_{\theta}}=0\,|\,\tau_{\theta}<\infty)=1$.
Notice now that $\{X_{\tau_{\theta}}\notin S_{\theta},\tau_{\theta}<\infty\}=\{\tau_{\theta}\circ\theta_{\tau_{\theta}}>0,\tau_{\theta}<\infty\}$.
As the latter event has probability zero, we conclude that $\P_{xA(\theta)}(X_{\tau_{\theta}}\notin S_{\theta}\,|\,\tau_{\theta}<\infty)=0$. 
\end{proof}
\begin{lemma}\label{lem:S_theta_interval} For any $\theta\in(\theta_{L},\theta_{U}]$,
we have $S_{\theta}=(x_{L},\alpha(\theta)]$. \end{lemma} 
\begin{proof}
We first show that $S_{\theta}\cap(\alpha(\theta),x_{U})=\emptyset$.
Take $x>\alpha(\theta)$. By assumption $\tau_{\theta}$ is optimal
for $v(x,A(\theta);\theta)$, so also for $\tilde{v}(x,A(\theta);\theta)$
defined in \eqref{eqn:tilde_v}. Following arguments of Lemma \ref{lem:v_ge_0}
with $\lambda$ replaced by $\varphi$ we obtain $\tilde{v}(x,A(\theta);\theta)>0$.
Hence, stopping immediately is suboptimal, so $\tau_{\theta}>0$ $\P_{xA(\theta)}$-a.s.
and $x\notin S_{\theta}$.

Assume that there are $b,c\in\cI$, $b<c$, such that $[b,c]\cap S_{\theta}=\{b,c\}$.
Take any $x\in(b,c)$. Thanks to Lemma \ref{lem:tau_in_S}, we have
$\sigma_{b,c}\le\tau_{\theta}$, where $\sigma_{b,c}$ is the first
entry time to the set $\{b,c\}$. By the definition of $S_{\theta}$
we further have that $\sigma_{b,c}=\tau_{\theta}$. By the optimality
of $\sigma_{b,c}$ for $\tilde{v}(x,A(\theta);\theta)$ we have 
\begin{equation}
\tilde{v}(x,A(\theta);\theta)=\E_{x}\Big[\int_{0}^{\sigma_{b,c}}e^{-rs-A(\theta)}\big(D(X_{s})-r\theta\big)ds\Big]<0,\label{eqn:tl_v_negative}
\end{equation}
where the last inequality is because $X_{s}\le\alpha(\theta)\le c(\theta)$
(Lemma \ref{lemm:conti-alpha}) and hence $D(X_{s})-r\theta<0$ (Assumption
\ref{assump:Li}). This contradicts the lower bound $\tilde{v}\ge0$
which can be obtained by stopping immediately. This means that $S_{\theta}$
is either an empty set or an interval.

The set $S_{\theta}$ cannot be empty as by Lemma \ref{lem:tau_in_S}
that would mean $\tau_{\theta}=\infty$, $\P_{xA(\theta)}$-a.s. But
clearly the best response to an opponent who never stops is the optimal
stopping time from Section \ref{sec:single_player} which is not infinite
$\P_{x}$-a.s. Hence, in equilibrium $S_{\theta}\ne\emptyset$.

We shall prove that $x_{L}$ is the left endpoint of $S_{\theta}$.
Assume that $\inf S_{\theta}=:b>x_{L}$. Take any $x\in(x_{L},b)$.
As above, the stopping time $\tau_{b}$ is optimal for $\tilde{v}(x,A(\theta);\theta)$.
The estimate \eqref{eqn:tl_v_negative} with $\sigma_{b,c}$ replaced
by $\sigma_{b}:=\inf\{t\ge0:\ X_{t}=b\}$ holds true and contradicts
$\tilde{v}(x,A(\theta);\theta)\ge0$.

It remains to show that $\sup S_{\theta}=:c=\alpha(\theta)$. Assume,
by contradiction, that $c<\alpha(\theta)$ and take any $x\in(c,\alpha(\theta))$.
Then $\sigma_{c}$ is optimal for $\tilde{v}(x,A(\theta);\theta)$
and yields the payoff 
\[
e^{-A(\theta)}\E_{x}\Big[\int_{0}^{\sigma_{c}}e^{-rs}\big(D(X_{s})-r\theta\big)ds\Big]\le\sup_{\sigma\in\cT(\cF_{t})}e^{-A(\theta)}\E_{x}\Big[\int_{0}^{\sigma}e^{-rs}\big(D(X_{s})-r\theta\big)ds\Big]=e^{-A(\theta)}\tilde{u}(x;\theta),
\]
where $\tilde{u}(x;\theta)$ is defined in \eqref{eqn:tl_u}. Since
$x<\alpha(\theta)$ then $\tilde{u}(x;\theta)=0$. We will show that
the inequality above is strict, i.e., $\sigma_{c}$ is suboptimal
for $\tilde{u}(x;\theta)$. Assume optimality of $\sigma_{c}$ and
select $z_{1},z_{2}$ so that $c<z_{1}<x<z_{2}<\alpha(\theta)$. From
the dynamic programming principle for $\tilde{u}$ we obtain 
\[
\E_{x}\Big[\int_{0}^{\sigma_{c}}e^{-rs}\big(D(X_{s})-r\theta\big)ds\Big]\le\E_{x}\Big[\int_{0}^{\sigma_{z_{1},z_{2}}}e^{-rs}\big(D(X_{s})-r\theta\big)ds+e^{-r\sigma_{z_{1},z_{2}}}\tilde{u}(X_{\sigma_{z_{1},z_{2}}};\theta)\Big]<0,
\]
where we used that $\tilde{u}(X_{\sigma_{z_{1},z_{2}}};\theta)=0$
and the integrand is strictly negative for $X_{s}\le\alpha(\theta)$.
This contradicts that $\tilde{u}\ge0$. 
\end{proof}

\subsection{Properties of $\varphi$}

\edt{We start from an immediate corollary which follows by combining
Lemma \ref{lem:incl} and Lemma \ref{lem:S_theta_interval}.} \begin{corollary}\label{cor:varphi_zero}
$\varphi(x,y)=0$ for $y\in(\theta_{L},\theta_{U}]$ and $x>\alpha(y)$.
\end{corollary} % \begin{proof}
% It suffices to combine .
% % we obtain that $\varphi(x, y) = 0$ for
% % \[
% % (x,y) \in \mathcal{A} := \{ (x,y) \in \cI \times [\theta_L, \theta_U]:\ x > \alpha(y) \}.
% % \]
% % By the continuity of $\varphi$ on $\overline\lambdaSet$, we have $\varphi \equiv 0$ on the closure of $\cA$ in $\overline\lambdaSet$ which contains $\{ (x,y) \in \overline\lambdaSet:\ x > \alpha(\theta) \}$. By definition, $\varphi \equiv 0$ on the complement of $\overline{\lambdaSet}$. 
% \end{proof}

\edt{We turn attention to upper and lower bounds for $\varphi$.}

\begin{lemma}\label{lem:phi_bound} We have $\varphi(x,y)\le\lambda(x,y)$
for $(x,y)\in\cI\times[\theta_{L},\theta_{U}]$. \end{lemma} 
\begin{proof}
Fix any $(x,\theta)$ such that $\varphi(x,\theta)>0$ and set $a=A(\theta)$.
We must have $x\le\alpha(\theta)$ thanks to Corollary \ref{cor:varphi_zero}.
From Lemma \ref{lem:incl}, we have that the optimal stopping time
$\tau_{\theta}$ for $\tilde{v}(x,a;\theta)$ satisfies $\tau_{\theta}=0$
$\P_{xa}$-a.s. Then, for any $t>0$, we have 
\[
\E_{xa}\Big[\int_{0}^{t}e^{-rs-A_{s}}\big(D(X_{s})-r\theta+\varphi(X_{s},Y(A_{s}))(m(X_{s})-\theta)\big)ds\Big]\le\tilde{v}(x,a;\theta)=0.
\]
\edt{We divide both sides by $t$ and change the variable of integration
to $z=s/t$: 
\[
\E_{xa}\Big[\int_{0}^{1}e^{-rtz-A_{tz}}\big(D(X_{tz})-r\theta+\varphi(X_{tz},Y(A_{tz}))(m(X_{tz})-\theta)\big)dz\Big]\le0.
\]
Since $D$ is bounded from below, $\varphi\ge0$ and $m>\theta$,
we can apply Fatou's lemma 
\begin{align*}
0 & \ge\liminf_{t\to0}\E_{xa}\Big[\int_{0}^{1}e^{-rtz-A_{tz}}\big(D(X_{tz})-r\theta+\varphi(X_{tz},Y(A_{tz}))(m(X_{tz})-\theta)\big)dz\Big]\\
 & \ge D(x)-r\theta+\varphi(x,Y(a))(m(x)-\theta),
\end{align*}
where in the last inequality follows from the continuity of trajectories
$(X_{t},A_{t})$, the continuity of functions $D$, $m$, $Y$, and
the lower semi-continuity of $\varphi$. The above inequality is equivalent
to $\varphi(x,\theta)\le\lambda(x,\theta)$, where we also used that
$x\le\alpha(\theta)$. The proof is concluded when we notice that
$\lambda\ge0$, so $\varphi(x,a)\le\lambda(x,a)$ when $\varphi(x,a)=0$.
}
\end{proof}
\begin{lemma}\label{lem:phi_lower} \edt{We have $\varphi^{*}(x,y)\ge\lambda(x,y)$
for $(x,y)\in\cI\times[\theta_{L},\theta_{U}]$.} \end{lemma} 
\begin{proof}
Denote by $\tilde{v}$ the value function of the problem \eqref{eqn:tilde_v}
with $(A_{t})$ given by \eqref{eqn:symm_A}. First notice that there
exists $(x,\theta)\in\cI\times(\theta_{L},\theta_{U}]$ such that $\varphi(x,\theta)>0$.
Otherwise, we would have a contradiction with $S_{\theta}=(x_L,\alpha(\theta)]$
asserted in Lemma \ref{lem:S_theta_interval}.

Fix $(x,\theta)$ such that $\varphi(x,\theta)>0$. Denote $a=A(\theta)$
and \edt{ 
\[
\Gamma(x',a')=e^{-a'}\big(D(x')-r\theta+\varphi^{*}(x',Y(a'))(m(x')-\theta)\big).
\]
} Take $\ve>0$. \edt{By the lower semi-continuity of $\varphi$
and the upper semi-continuity of $\Gamma$, using Lemma \ref{lem:semicont},}
there is $\delta>0$ such that $U:=[x-\delta,x+\delta]\times[\theta-\delta,\theta]\subset\cI\times[\theta_{L},\theta_{U}]$,
\begin{equation}
\inf_{(x',y')\in U}\varphi(x',y'))>\varphi(x,\theta)/2,\qquad\text{and}\qquad\sup_{(x',y')\in U} \Gamma(x',A(y'))\le\Gamma(x,a) + \ve.\label{eqn:U_bounds}
\end{equation}
Recall the definition \eqref{eqn:tau_A_a} of an optimal stopping
time $\tau_{\gamma}$ for $\tilde{v}(x,a;\gamma)$ for $\theta-\delta\le\gamma\le\theta$.
Let $\sigma_{\delta}=\inf\{t\ge0:\ X_{t}\notin(x-\delta,x+\delta)\}$.
By the optimality of $\tau_{\gamma}$ we have 
\begin{align*}
\tilde{v}(x,a;\gamma) & =\E_{xa}\Big[\int_{0}^{\tau_{\gamma}}e^{-rs-A_{s}}\big(D(X_{s})-r\gamma+\varphi(X_{s},Y(A_{s}))(m(X_{s})-\gamma)\big)ds\Big]\\
 & =\E_{xa}\Big[\int_{0}^{\tau_{\gamma}\wedge\sigma_{\delta}}e^{-rs-A_{s}}\big(D(X_{s})-r\gamma+\varphi(X_{s},Y(A_{s}))(m(X_{s})-\gamma)\big)ds\\
 & \hspace{34pt}+\ind{\sigma_{\delta}<\tau_{\gamma}}\int_{\tau_{\gamma}\wedge\sigma_{\delta}}^{\tau_{\gamma}}e^{-rs-A_{s}}\big(D(X_{s})-r\gamma+\varphi(X_{s},Y(A_{s}))(m(X_{s})-\gamma)\big)ds\Big]\\
 & \le\E_{xa}\Big[\int_{0}^{\tau_{\gamma}\wedge\sigma_{\delta}}e^{-rs-A_{s}}\big(D(X_{s})-r\gamma+\varphi(X_{s},Y(A_{s}))(m(X_{s})-\gamma)\big)ds\\
 & \hspace{34pt}+\ind{\sigma_{\delta}<\tau_{\gamma}}e^{-r\sigma_{\delta}}\tilde{v}(X_{\sigma_{\delta}},A_{\sigma_{\delta}};\gamma)\Big]\\
 & \edt{\le\E_{xa}\Big[\int_{0}^{\tau_{\gamma}\wedge\sigma_{\delta}}e^{-rs-A_{s}}\big(D(X_{s})-r\gamma+\varphi^{*}(X_{s},Y(A_{s}))(m(X_{s})-\gamma)\big)ds}\\
 & \hspace{34pt}+\edt{\ind{\sigma_{\delta}<\tau_{\gamma}}e^{-r\sigma_{\delta}}\tilde{v}(X_{\sigma_{\delta}},A_{\sigma_{\delta}};\gamma)\Big],}
\end{align*}
where the first inequality follows from the strong Markov property
and the definition of $\tilde{v}(X_{\sigma_{\delta}},A_{\sigma_{\delta}};\gamma)$,
\edt{and the second inequality uses $\varphi^{*}\ge\varphi$.} %Note that it is natural to expect the equality above but proving it requires, for example, the continuity of $\tilde v$ which we have not established. For the purpose of the proof, the inequality is sufficient for us.

We transform the final expression above and use $\tilde{v}\ge0$ to
obtain 
\begin{equation}
\begin{aligned}0 & \le\E_{xa}\Big[\int_{0}^{\tau_{\gamma}\wedge\sigma_{\delta}}e^{-rs}\Gamma(X_{s},A_{s})ds\Big]\\
 & \hspace{12pt}+\E_{xa}\Big[\int_{0}^{\tau_{\gamma}\wedge\sigma_{\delta}}e^{-rs-A_{s}}(r+\varphi^{*}(X_{s},Y(A_{s}))(\theta-\gamma)ds\Big]\\
 & \hspace{12pt}+\E_{xa}\Big[\ind{\sigma_{\delta}<\tau_{\gamma}}e^{-r\sigma_{\delta}}\tilde{v}(X_{\sigma_{\delta}},A_{\sigma_{\delta}};\gamma)\Big]=(I)+(II)+(III).
\end{aligned}
\label{eqn:split}
\end{equation}
We divide both sides by $\E_{xa}[\tau_{\gamma}\wedge\sigma_{\delta}]$;
this can be done as $\P_{xa}(\tau_{\gamma}\wedge\sigma_{\delta}>0)=1$
due to the 0-1 law. Using the bounds \eqref{eqn:U_bounds} the first
terms yields 
\[
\frac{(I)}{\E_{xa}[\tau_{\gamma}\wedge\sigma_{\delta}]}\le\Gamma(x,a)+\ve.
\]
Recall that $\lambda$ is bounded from above, while Lemma \ref{lem:phi_bound}
shows that $\varphi\le\lambda$. Since $\varphi^{*}\le\sup_{(x,y)}\varphi(x,y)$,
we conclude that $r+\varphi^*$ is bounded above by some constant $C_{1}$
and 
\[
\frac{(II)}{\E_{xa}[\tau_{\gamma}\wedge\sigma_{\delta}]}\le\frac{\E_{xa}[\tau_{\gamma}\wedge\sigma_{\delta}]C_{1}(\theta-\gamma)}{\E_{xa}[\tau_{\gamma}\wedge\sigma_{\delta}]}=C_{1}(\theta-\gamma).
\]
To estimate the last term, notice that $\tilde{v}(x',a';\gamma)=v(x',a';\gamma)-e^{-a'}\gamma\le m(x')-e^{-a'}\gamma\le C_{2}$
for some $C_{2}>0$ since the function $m$ is bounded. Then 
\[
\frac{(III)}{\E_{xa}[\tau_{\gamma}\wedge\sigma_{\delta}]}\le C_{2}\frac{\P_{xa}(\sigma_{\delta}<\tau_{\gamma})}{\E_{xa}[\tau_{\gamma}\wedge\sigma_{\delta}]}.
\]
We apply the above estimates to the right-hand side of \eqref{eqn:split}
and take limit as $\gamma\uparrow\theta$: 
\begin{equation}
0\le\Gamma(x,a)+\ve+0+C_{2}\lim_{\gamma\uparrow\theta}\frac{\P_{xa}(\sigma_{\delta}<\tau_{\gamma})}{\E_{xa}[\tau_{\gamma}\wedge\sigma_{\delta}]}.\label{eqn:5.9}
\end{equation}
Let $\underline{\varphi}=\inf_{(x',y')\in U}\varphi(x',y')\ge\frac{1}{2}\varphi(x,\theta)$
by \eqref{eqn:U_bounds} and $\overline{\varphi}=\sup_{(x',y')\in U}\varphi(x',y')<\infty$
by $\varphi \le \lambda$. Those bounds on
$\varphi$ allow us to bound the numerator and denominator under the
limit: 
\[
\{\sigma_{\delta}<\tau_{\gamma}\}=\{A_{\sigma_{\delta}}\le A(\gamma)\}\subseteq\{a+\underline{\varphi}\sigma_{\delta}\le A(\gamma)\}=\Big\{\sigma_{\delta}\le\frac{A(\gamma)-a}{\underline{\varphi}}\Big\},
\]
and 
\[
\E_{xa}[\tau_{\gamma}\wedge\sigma_{\delta}]\ge\E_{xa}[\ind{\sigma_{\delta}>\tau_{\gamma}}\tau_{\gamma}]\ge\E_{xa}\Big[\ind{\sigma_{\delta}>\tau_{\gamma}}\frac{A(\gamma)-a}{\overline{\varphi}}\Big]=\frac{A(\gamma)-a}{\overline{\varphi}}\P_{xa}(\sigma_{\delta}>\tau_{\gamma}).
\]
Combining these estimates yields 
\[
\lim_{\gamma\uparrow\theta}\frac{\P_{xa}(\sigma_{\delta}<\tau_{\gamma})}{\E_{xa}[\tau_{\gamma}\wedge\sigma_{\delta}]}\le\lim_{\gamma\uparrow\theta}\frac{\P_{xa}\Big(\sigma_{\delta}\le\frac{A(\gamma)-a}{\underline{\varphi}}\Big)}{\frac{A(\gamma)-a}{\overline{\varphi}}\P_{xa}(\sigma_{\delta}>\tau_{\gamma})}=\underline{\varphi}/\overline{\varphi}\lim_{\gamma\uparrow\theta}\frac{\P_{xa}\Big(\sigma_{\delta}\le\frac{A(\gamma)-a}{\underline{\varphi}}\Big)}{\frac{A(\gamma)-a}{\underline{\varphi}}},
\]
where the last equality follows from $\lim_{\gamma\uparrow\theta}\P_{xa}(\sigma_{\delta}>\tau_{\gamma})=1$
(as $a=A(\theta)$). It remains to study the asymptotic behaviour
of $\sigma_{\delta}$ near $0$, i.e., the limit $\lim_{t\downarrow0}\P_{x}(\sigma_{\delta}<t)/t$.
It is well known that $\P_{x}(\sigma_{\delta}<t)$ decreases exponentially
fast as it does for the Brownian motion, but we could not find a direct
reference for this fact. For completeness, we provide a derivation
of a bound for this probability in the appendix. From \eqref{eqn:decay},
we have 
\[
\P_{xa}(\sigma_{\delta}<t)\le C_{3}e^{-C_{4}/t^{2}},
\]
for some constants $C_{3},C_{4}>0$. It is now immediate to see that
\[
\lim_{t\downarrow0}\P_{xa}(\sigma_{\delta}<t)/t\le C_{3}\lim_{t\downarrow0}e^{-C_{4}/t^{2}}/t=0.
\]

Returning to \eqref{eqn:5.9}, we have shown that $\Gamma(x,a)+\ve\ge0$
for any $\ve>0$, i.e., $\Gamma(x,a)\ge0$. This implies that 
\[
D(x)-r\theta+\varphi^{*}(x,Y(a))(m(x)-\theta)\ge0,
\]
which is equivalent to $\varphi^{*}(x,\theta)\ge\lambda(x,\theta)$
upon recollection that $a=A(\theta)$.

We have therefore demonstrated that \edt{$\varphi^{*}(x,\theta)\ge\lambda(x,\theta)$
for $(x,\theta)$ such that $\varphi(x,\theta)>0$.} Define
$\cO_{+}=\{(x,\theta)\in\cI\times[\theta_{L},\theta_{U}]:\ \varphi(x,\theta)>0\}$
and $\cO_{\alpha}=\{(x,\theta)\in\cI\times[\theta_{L},\theta_{U}]:x\le\alpha(\theta)\}$.
Due to the upper bound $\varphi\le\lambda$ (see Lemma \ref{lem:phi_bound})
and the fact that $\lambda\equiv0$ on the complement of $\cO_{\alpha}$,
we have $\cO_{+}\subset\cO_{\alpha}$. \edt{By the upper semi-continuity
of $\varphi^{*}$ and the continuity of $\lambda$ on $\cO_{\alpha}$,
we further have that $\varphi^{*}\ge\lambda$ on $\text{cl}(\cO_{+})$,}
where $\text{cl}(\cdot)$ denotes the closure. Let $U_{0}=\cO_{\alpha}\setminus\text{cl}(\cO_{+})$.
This is a relatively open set in $\cO_\alpha$. Furthermore, $\varphi\equiv0$
on $U_{0}$. Assume that $U_{0}$ is non-empty. Due to the closedness of $\cO_\alpha$, it has a non-empty interior. Take any $(x,\theta)$ in the interior of $U_{0}$. We immediately have that $\P_{xA(\theta)}(\tau_{\theta}>0)=1$. However,
$(x,\theta)\in\cO_{\alpha}$, so $x\le\alpha(\theta)$ and, by Lemma
\ref{lem:S_theta_interval}, $x\in S_{\theta}$. This means that $\P_{xA(\theta)}(\tau_{\theta}=0)=1$,
a contradiction. This completes the proof that $\varphi^{*}(x,\theta)\ge\lambda(x,\theta)$
for $(x,\theta)\in\cO_{\alpha}$. Recall that $\lambda\equiv0$ on
the complement of $\cO_{\alpha}$. Since $\varphi^{*}$ is non-negative,
it trivially dominates $\lambda$ on the complement of $\cO_{\alpha}$,
which finishes the proof. 
\end{proof}
Combining Lemma \ref{lem:phi_bound} and \ref{lem:phi_lower} yields
the proof of the main result of this section.

\edt{
\begin{proof}[Proof of Theorem \ref{thm:unique}]
Notice that $\lambda$ is upper semi-continuous and it majorises
$\varphi$ by Lemma \ref{lem:phi_bound}. Hence, it also majorises
$\varphi^{*}$ which is the smallest upper semi-continuous function
dominating $\varphi$. However, $\lambda$ also bounds $\varphi^{*}$
from below, which completes the proof. 
\end{proof}
}

\appendix

\section{Asymptotics of $\sigma_{\delta}$ near $0$}

We provide a sketch of an asymptotic bound for the behaviour of $\P_{x}(\sigma_{\delta}<u)$
as $u\downarrow0$, where $\sigma_{\delta}=\inf\{t\ge0:X_{t}\notin(x-\delta,x+\delta)\}$.
Notice that the probability is identical when the coefficients of
$X_{t}$ are replaced with 
\[
\tilde{\mu}(y)=\mu(y\wedge(x+\delta)\vee(x-\delta)),\qquad\tilde{\volatility}(y)=\volatility(y\wedge(x+\delta)\vee(x-\delta)),
\]
i.e., we can assume that $\mu$ and $\volatility$ in \eqref{eqn:X}
are bounded, continuous and $\volatility$ is uniformly bounded away
from $0$. Consider the change of measure given by 
\[
\frac{d\tilde{\P}_{x}}{d\P_{x}}=\eta_{1},
\]
where 
\[
\eta_{t}=\exp\Big(-\int_{0}^{t}\mu(X_{s})/\volatility(X_{s})dW_{s}-\int_{0}^{t}\mu^{2}(X_{s})/\volatility^{2}(X_{s})ds\Big).
\]
Then for $u\le1$, we have 
\[
\P_{x}(\sigma_{\delta}<u)=\tilde{\E}_{x}(\ind{\sigma_{\delta}<u}\eta_{1}^{-1})\le\big(\tilde{P}_{x}(\sigma_{\delta}<u)\big)^{1/2}\|\eta_{1}^{-1}\|_{L^{2}}=c_{1}\big(\tilde{P}_{x}(\sigma_{\delta}<u)\big)^{1/2}
\]
for some constant $c_{1}>0$ and 
\[
dX_{t}=b(X_{t})d\tilde{W}_{t},\quad X_{0}=x
\]
for $\tilde{\P}_{x}$-Brownian motion $\tilde{W}_{t}$. Consider $\Lambda_{t}=\int_{0}^{t}\volatility^{2}(X_{s})ds$
and the time change $T_{t}$ being the inverse of $\Lambda_{t}$ which
exists since $\volatility$ is separated from $0$. Then $Y_{t}=X_{T_{t}}$
is a Brownian motion. We have the sequence of inclusions: 
\[
\{\sigma_{\delta}<t\}=\Big\{\sup_{u\in[0,t]}|X_{u}-x|\ge\delta\Big\}\subset\Big\{\sup_{u\in[0,\esssup_{\omega\in\Omega}\Lambda_{t}(\omega)]}|Y_{u}-x|\ge\delta\Big\},
\]
where we used that $X_{t}=Y_{\Lambda_{t}}$. Let $\bar{b}=\sup_{y\in[x-\delta,x+\delta]}\volatility(y)>0$.
Then $\Lambda_{t}\le\bar{b}t$ and 
\[
\Big\{\sup_{u\in[0,\esssup_{\omega\in\Omega}\Lambda_{t}(\omega)]}|Y_{u}-x|\ge\delta\Big\}\subset\Big\{\sup_{u\in[0,t\bar{b}]}|Y_{u}-x|\ge\delta\Big\}.
\]
This gives us the following estimate 
\begin{align*}
\tilde{\P}_{x}(\sigma_{\delta}<t) & \le\tilde{P}_{x}\Big(\sup_{u\in[0,t\bar{b}]}|Y_{u}-x|\ge\delta\Big)=\tilde{P}_{0}\Big(\sup_{u\in[0,t\bar{b}]}|Y_{u}|\ge\delta\Big)\\
 & \le2\tilde{P}_{0}\Big(\sup_{u\in[0,t\bar{b}]}Y_{u}\ge\delta\Big)\le2\sqrt{\frac{2}{\pi}}\int_{\delta/(t\bar{b})}^{\infty}e^{-z^{2}/2}dz\le c_{2}e^{-c_{3}/t^{2}},
\end{align*}
where the penultimate inequality follows from \cite[Proposition~3.7, Chapter III]{revuzyor}
and the last inequality holds for sufficiently small $t$ (precisely,
$t$ such that $\delta/(t\bar{b})\ge1$). Combining together the above
estimates yields 
\begin{equation}
\P_{x}(\sigma_{\delta}<u)\le c_{1}\sqrt{c_{2}}e^{-0.5c_{3}/u^{2}}\label{eqn:decay}
\end{equation}
for sufficiently small $u$.

\section{Example from the text}

\label{app:example}

Below we provide the proof that the model
in Example \ref{xmp:1} satisfies all the assumptions of the paper.
First, we obtain the explicit form of $d(x)$ and $m(x)$. Clearly,
$m(x)=d(x)+\frac{M_{0}}{r}$. To find $d$, we utilise the fact that
$(\mathcal{L}_{X}-r)d(x)+D(x)=0$ for $x\in(0,x_{M})\cup(x_{M},\infty)$
and that $d(x)$ is a continuously differentiable bounded function.
It can be verified that (see \cite{Alvarez2008}) 
\begin{align*}
d(x) & =\begin{cases}
\frac{x^{\beta}}{r-\delta}+c_{1}(x_{M})\psi(x), & x\in(0,x_{M}],\\
\frac{x_{M}^{\beta}}{r}+c_{2}(x_{M})\phi(x), & x>x_{M},
\end{cases}
\end{align*}
where $\delta$ was defined after \eqref{eqn:ass_exmpl}.
Because $d(\cdot)$ must be a bounded function,
the fundamental solutions $\phi(\cdot)$ and $\psi(\cdot)$ do not
show respectively in the general form of the solution $d(x)$ in the
intervals $(0,x_{M}]$ and $(x_{M},\infty)$. The coefficients $c_{1}(x_{M})$
and $c_{2}(x_{M})$ are determined by the continuity and the differentiability
of $d(\cdot)$ at $x_{M}$: 
\begin{align*}
\frac{x_{M}^{\beta}}{r-\delta}+c_{1}(x_{M})\psi(x_{M}) & =\frac{x_{M}^{\beta}}{r}+c_{2}(x_{M})\phi(x_{M})\\
\frac{\beta x_{M}^{\beta-1}}{r-\delta}+c_{1}(x_{M})\psi'(x_{M}) & =c_{2}(x_{M})\phi'(x_{M})\:,
\end{align*}
from which we obtain 
\begin{align*}
c_{1}(x_{M}) & =\frac{1}{\xi(x_{M})}[-\phi'(x_{M})x_{M}^{\beta}(\frac{1}{r}-\frac{1}{r-\delta})-\beta\phi(x_{M})\frac{x_{M}^{\beta-1}}{r-\delta}],\\
c_{2}(x_{M}) & =\frac{1}{\xi(x_{M})}[-\psi'(x_{M})x_{M}^{\beta}(\frac{1}{r}-\frac{1}{r-\delta})-\beta\psi(x_{M})\frac{x_{M}^{\beta-1}}{r-\delta}]
\end{align*}
with $\xi(x)=\phi(x)\psi'(x)-\psi(x)\phi'(x)=(\gamma_{+}-\gamma_{-})x^{\gamma_{+}+\gamma_{-}-1}$.
Using the definition of $\phi(\cdot)$, we simplify expression for
$c_{1}(x_{M})$ as 
\[
c_{1}(x_{M})=\frac{\gamma_{-}\delta-\beta r}{(\gamma_{+}-\gamma_{-})r(r-\delta)}x_{M}^{\beta-\gamma_{+}}\:.
\]
Its numerator can be further rewritten as follows: 
\begin{align}
\gamma_{-}\delta-\beta r & =\beta[\frac{b^{2}}{2}\gamma_{-}(\beta-1)+\mu\gamma_{-}-r]\nonumber \\
 & =\beta[\frac{b^{2}}{2}\gamma_{-}(\gamma_{-}-1)+\mu\gamma_{-}-r]+\beta\frac{b^{2}}{2}\gamma_{-}(\beta-\gamma_{-})=\beta\frac{b^{2}}{2}\gamma_{-}(\beta-\gamma_{-}),\label{eq:gamma-delta}
\end{align}
where we used the equality $\frac{b^{2}}{2}\gamma_{-}(\gamma_{-}-1)+\mu\gamma_{-}-r=0$
satisfied by $\gamma_{-}$. Since $\gamma_{-}<0$ and $\beta-\gamma_{-}>0$,
it follows that $c_{1}(x_{M})<0$. Furthermore, using that $\gamma_{+}>1>\beta$,
we have $\lim_{x_{M}\to\infty}c_{1}(x_{M})=0$.

We now show that $D(\cdot)$ and $M(\cdot)$ satisfy
all the assumptions given in the paper if $x_{M}$ is taken sufficiently
large. First, Assumptions \ref{ass:theta_supp}--\ref{ass:coeff}
are trivially satisfied. Our remaining task is to show that Assumption
\ref{assump:Li} is satisfied. Assumption \ref{assump:Li}(i) holds
because $D(\cdot)$ is increasing from zero to $x_{M}^{\beta}>r\theta_{U}$.
By Remark \ref{rem:a'}, Assumption \ref{assump:Li}(ii) and (iii)
are satisfied if there is $x^{*}=\alpha(\theta)$ such that $a_{\theta}'(x)>0$
for $x<x^{*}$ and $a_{\theta}'(x)<0$ for $x>x^{*}$. Hence, we examine
the derivative of $a_{\theta}(x)=(\theta-d(x))/\phi(x)$ given by
\begin{equation}
a_{\theta}'(x)=\begin{cases}
-\gamma_{-}\theta x^{-\gamma_{-}-1}-(\beta-\gamma_{-})\frac{x^{\beta-\gamma_{-}-1}}{r-\delta}-(\gamma_{+}-\gamma_{-})c_{1}(x_{M})x^{\gamma_{+}-\gamma_{-}-1}, & x\le x_{M},\\
-\gamma_{-}(\theta-\frac{x_{M}^{\beta}}{r})x^{-\gamma_{-}-1}, & x>x_{M}.
\end{cases}\label{eqn:exp_1}
\end{equation}
From \eqref{eqn:exp_1}, $a_{\theta}'(x)<0$ for $x>x_{M}$, because
$\theta-x_{M}^{\beta}/r\le\theta_{U}-x_{M}^{\beta}/r<0$ by the
assumption that $x_{M}^{\beta}>r\theta_{U}$. We also note that $a_{\theta}'(x)>0$
for sufficiently small values of $x$ because of \eqref{eqn:ass_exmpl}
and $\gamma_{+}-\gamma_{-}>0$. We shall show that $a_{\theta}'(x)$
is decreasing on $(0,x_{M})$ for sufficiently large $x_{M}$, which,
together with the above observations about $a_{\theta}'(x)$ for $x$
close to $0$ and $x>x_{M}$ allows us to conclude that there is $x^{*}=\alpha(\theta)$
such that $a_{\theta}'(x)>0$ for $x<x^{*}$ and $a_{\theta}'(x)<0$
for $x>x^{*}$.

For convenience, we express $a_{\theta}'(x)=A(x)-B(x)+C(x)$
for $x\in(0,x_{M})$, where $A(\cdot)$, $B(\cdot)$, $C(\cdot)$
are positive functions given by 
\[
A(x)=-\gamma_{-}\theta x^{-\gamma_{-}-1};\quad B(x)=(\beta-\gamma_{-})\frac{x^{\beta-\gamma_{-}-1}}{r-\delta};\quad C(x)=-(\gamma_{+}-\gamma_{-})c_{1}(x_{M})x^{\gamma_{+}-\gamma_{-}-1}\,.
\]
Note that $A(\cdot)$ is decreasing while $B(\cdot)$ and $C(\cdot)$
are increasing. Furthermore, it can be easily checked that $C(x)/B(x)$
increases in $x$ because $\gamma_{+}>1>\beta$ and $c_{1}(x_{M})<0$.
Hence, within the interval $(0,x_{M}]$, $C(x)/B(x)$ takes the maximum
value at $x_{M}$, which is given by 
\begin{align*}
g & :=\frac{C(x_{M})}{B(x_{M})}=\frac{(r-\delta)}{(\beta-\gamma_{-})}\cdot(\gamma_{+}-\gamma_{-})\frac{-(\gamma_{-}\delta-\beta r)}{(\gamma_{+}-\gamma_{-})r(r-\delta)}=\frac{\beta b^{2}\vert\gamma_{-}\vert}{2r},
\end{align*}
where we used the alternative expression of $\gamma_{-}\delta-\beta r$
in (\ref{eq:gamma-delta}). Here we have $g<1$ from the assumption
that $1>\beta b^{2}\vert\gamma_{-}\vert/(2r)$. Thus, we obtain for
$x\le x_{M}$ 
\begin{equation}
B(x)-C(x)=B(x)\Big(1-\frac{C(x)}{B(x)}\Big)\ge B(x)\Big(1-\frac{C(x_{M})}{B(x_{M})}\Big)>0\:.\label{eqn:B_C}
\end{equation}

Fix $\epsilon>0$. There is $y_{\epsilon}>0$
such that $A(y_{\epsilon})<\epsilon$ and $B(y_{\epsilon})(1-g)>2\epsilon$.
We compute the derivative of $B(x)-C(x)$ to judge its monotonicity
on $(0,y_{\epsilon})$: 
\[
B'(x)-C'(x)=x^{\beta-\gamma_{-}-2}\Big[\frac{(\beta-\gamma_{-})(\beta-\gamma_{-}-1)}{r-\delta}+(\gamma_{+}-\gamma_{-})(\gamma_{+}-\gamma_{-}-1)C_{1}(x_{M})x^{\gamma_{+}-\beta}\Big].
\]
By \eqref{eqn:ass_exmpl}, we have $\beta-\gamma_{-}-1>0$, $\gamma_{+}-\gamma_{-}-1>0$
and $\gamma_{+}-\beta>0$. Recalling that $C_{1}(x_{M})<0$ and converges
to $0$ as $x_{M}\to\infty$, there is $x_{M}>y_{\epsilon}$ that
satisfies assumptions stated previously and such that $B(x)-C(x)$
is increasing on $(0,y_{\epsilon})$. Since $A(x)$ is decreasing,
this implies that $a'_{\theta}(x)$ is decreasing for $x\in(0,y_{\epsilon})$.
Using \eqref{eqn:B_C} and the fact that $B$ is increasing, we have
$B(x)-C(x)>B(x)(1-g)>B(y_{\epsilon})(1-g)>2\epsilon$ for $x\in(y_{\epsilon},x_{M}]$.
Thus, $A(x)-B(x)+C(x)<\epsilon-2\epsilon<0$ for $x\in(y_{\epsilon},x_{M}]$,
i.e., $a'_{\theta}$ is decreasing. We conclude that there is a unique
value of $x=\alpha(\theta)$ that satisfies $a_{\theta}'(x)=0$ within
the interval $(0,x_{M}]$. We further recall that $a'_{\theta}(x)<0$
for $x>x_{M}$. This establishes that $a_{\theta}'(x)>0$ for $x<\alpha(\theta)$
and $a_{\theta}'(x)<0$ for $x>\alpha(\theta)$.

 \bibliographystyle{MOR}
\bibliography{reference}

\begin{thebibliography}{10}

\bibitem{Alvarez2001}
Alvarez LHR (2001) Reward functionals, salvage values and optimal stopping.
\newblock \emph{Mathematical Methods of Operations Research} 54(2):315--337.

\bibitem{Alvarez2008}
Alvarez LHR, Lempa J (2008) On the optimal stochastic impulse control of linear
  diffusions.
\newblock \emph{SIAM Journal on Control and Optimization} 47(2):703--732.

\bibitem{Attard2018}
Attard N (2018) Nonzero-sum games of optimal stopping for {Markov} processes.
\newblock \emph{Applied Mathematics \& Optimization} 77:567–597.

\bibitem{Bensoussan1977}
Bensoussan A, Friedman A (1977) Nonzero-sum stochastic differential games with
  stopping times and free boundary value problem.
\newblock \emph{Transactions of American Mathematical Society} 213:275--327.

\bibitem{Blumenthal}
Blumenthal R, Getoor R (1968) \emph{Markov Processes and Potential Theory}
  (Academic Press, New York).

\bibitem{Bogachev}
Bogachev VI (2007) \emph{Measure theory. Volume 1} (Springer-Verlag Berlin
  Heidelberg).

\bibitem{Cattaiaux1990}
Cattiaux P, Lepeltier JP (1990) Existence of a quasi-markov nash equilibrium
  for non-zero sum {M}arkov stopping games.
\newblock \emph{Stochastics and Stochastic Reports} 30(2):85--103.

\bibitem{DeAngelis2020ghosts}
{De Angelis} T, Ekstr\"{o}m E (2020) Playing with ghosts in a {D}ynkin game.
\newblock \emph{Stochastic Processes and their Applications}
  130(10):6133--6156.

\bibitem{DeAngelis2020}
De~Angelis T, Ekstr{\"o}m E, Glover K (2022) {D}ynkin games with incomplete and
  asymmetric information.
\newblock \emph{Mathematics of Operations Research} 47:560--586.

\bibitem{DAMP2021}
De~Angelis T, Merkulov N, Palczewski J (2022) On the value of non-{M}arkovian
  {D}ynkin games with partial and asymmetric information.
\newblock \emph{Annals of Applied Probability} 32(3):1774--1813.

\bibitem{Decamps2022}
Décamps JP, Gensbittel F, Mariotti T Mixed-{{Strategy Equilibria}} in the
  {{War}} of {{Attrition}} under {{Uncertainty}} .

\bibitem{Decamps2004}
Décamps JP, Mariotti T (2004) Investment timing and learning externalities.
\newblock \emph{Journal of Economic Theory} 118:80--102.

\bibitem{Ekstrom2022Salami}
Ekstr\"{o}m E, Lindensj\"{o} K, Olofsson M (2022) How to detect a salami
  slicer: A stochastic controller-and-stopper game with unknown competition.
\newblock \emph{SIAM Journal on Control and Optimization} 60(1):545--574.

\bibitem{Elfenbein2015}
Elfenbein DW, Knott AM (2015) Time to exit: Rational, behavioral, and
  organizational delays.
\newblock \emph{Strategic Management Journal} 36(7):957--975.

\bibitem{esmaeeli2018}
Esmaeeli N, Imkeller P (2018) American options with asymmetric information and
  reflected {BSDE}.
\newblock \emph{Bernoulli} 24(2):1394--1426.

\bibitem{Feinstein2022}
Feinstein Z, Rudloff B, Zhang J (2022) Dynamic set values for nonzero-sum games
  with multiple equilibriums.
\newblock \emph{Mathematics of Operations Research} 47(1):616--642.

\bibitem{Filippov}
Filippov A (1988) \emph{Differential Equations with Discontinuous Righthand
  Sides: Control Systems}.
\newblock Mathematics and its Applications (Springer Netherlands).

\bibitem{Fudenberg1986}
Fudenberg D, Tirole J (1986) A theory of exit in duopoly.
\newblock \emph{Econometrica} 54(4):943--960.

\bibitem{Gensbittel2019}
Gensbittel F, Gr{\"u}n C (2019) Zero-sum stopping games with asymmetric
  information.
\newblock \emph{Mathematics of Operations Research} 44(1):277--302.

\bibitem{Georgiadis2019}
Georgiadis G, Kim Y, Kwon HD (2022) The absence of attrition in a war of
  attrition under complete information.
\newblock \emph{Games and Economic Behavior} 131:171--185.

\bibitem{Ghemawat1985}
Ghemawat P, Nalebuff B (1985) Exit.
\newblock \emph{Rand Journal of Economics} 16(2):184--194.

\bibitem{Grun2013}
Gr{\"u}n C (2013) On {Dynkin} games with incomplete information.
\newblock \emph{SIAM Journal on Control and Optimization} 51(5):4039--4065.

\bibitem{Halmos}
Halmos PR (1974) \emph{Measure theory} (Springer-Verlag New York).

\bibitem{Hamadene2014}
Hamad\`{e}ne S, Hassani M (2014) The multiplayer nonzero-sum {Dynkin} game in
  continuous time.
\newblock \emph{SIAM Journal on Control and Optimization} 52(2):821--835.

\bibitem{Hamadene2010}
Hamad\`{e}ne S, Zhang J (2010) The continuous time nonzero-sum {Dynkin} game
  problem and application in game options.
\newblock \emph{SIAM Journal on Control and Optimization} 48(5):3659--3669.

\bibitem{Harrigan2003}
Harrigan KR (2003) \emph{Declining demand, divestiture, and corporate strategy}
  (Beard Books, Washington, D.C.).

\bibitem{Horn2006}
Horn JT, Lovallo DP, Viguerie SP (2006) Learning to let go: Making better exit
  decisions.
\newblock \emph{McKinsey Quarterly} 2:64.

\bibitem{elkaroui1981}
Karoui NE (1981) Les aspects probabilistes du contr{\^o}le stochastique.
\newblock In \emph{{\'E}cole d’{\'e}t{\'e} de Probabilit{\'e}s de Saint-Flour
  IX-1979}, pages 73--238 (Springer).

\bibitem{krylov}
Krylov N (1980) \emph{Controlled Diffusion Processes} (Springer).

\bibitem{Lempa2013}
Lempa J, Matom{\"a}ki P (2013) A {Dynkin} game with asymmetric information.
\newblock \emph{Stochastics} 85(5):763--788.

\bibitem{Milgrom1985}
Milgrom PR, Weber RJ (1985) Distributional strategies for games with incomplete
  information.
\newblock \emph{Mathematics of Operations Research} 10(4):619--632.

\bibitem{Murto2004}
Murto P (2004) Exit in duopoly under uncertainty.
\newblock \emph{RAND Journal of Economics} 35(1):111--127.

\bibitem{Nagai1987}
Nagai H (1987) Non zero-sum stopping games of symmetric markov processes.
\newblock \emph{Probability Theory and Related Fields} 75:487--497.

\bibitem{Nalebuff1985}
Nalebuff B, Riley J (1985) Asymmetric equilibria in the war of attrition.
\newblock \emph{Journal of Theoretical Biology} 113(3):517--527.

\bibitem{Osborne1994}
Osborne MJ, Rubinstein A (1994) \emph{A course in game theory} (MIT press,
  Cambridge, MA).

\bibitem{Perez2021}
P\'{e}rez JL, Rodosthenous N, Yamazaki K (2021) Non-zero-sum optimal stopping
  game with continuous versus periodic observations.
\newblock \emph{arXiv:2107.08243} .

\bibitem{peskir2006optimal}
Peskir G, Shiryaev A (2006) \emph{Optimal stopping and free-boundary problems}
  (Springer).

\bibitem{Ponsati1995}
Ponsati C, S{\'{a}}kovics J (1995) The war of attrition with incomplete
  information.
\newblock \emph{Mathematical Social Sciences} 29(3):239--254.

\bibitem{Protter}
Protter PE (2005) \emph{Stochastic integration and differential equations}.
\newblock 2nd edition (Springer).

\bibitem{revuzyor}
Revuz D, Yor M (1999) \emph{Continuous Martingales and Brownian Motion}
  (Springer-Verlag Berlin Heidelberg).

\bibitem{Riedel2017}
Riedel F, Steg JH (2017) Subgame-perfect equilibria in stochastic timing games.
\newblock \emph{Journal of Mathematical Economics} 72:36--50.

\bibitem{Riley1980}
Riley JG (1980) Strong evolutionary equilibrium and the war of attrition.
\newblock \emph{Journal of Theoretical Biology} 82(3):383 -- 400.

\bibitem{Royden1988}
Royden HL (1988) \emph{{Real analysis}}.
\newblock 3rd edition (Macmillan, New York).

\bibitem{shiryaev2007optimal}
Shiryaev AN (2008) \emph{Optimal stopping rules} (Springer-Verlag, Berlin).

\bibitem{Steg2015}
Steg JH (2015) Symmetric equilibria in stochastic timing games.
\newblock \emph{arXiv:1507.04797} .

\end{thebibliography}

\end{document}